\documentclass[AMA,STIX1COL]{WileyNJD-v2}
\articletype{Article Type}%
\usepackage{amsmath}
\usepackage{amssymb}
\usepackage{paralist}
\usepackage{dsfont}
\usepackage{graphics} 
\usepackage{epsfig} 
\usepackage{float}
\usepackage{graphicx}  
\usepackage{epstopdf}
\usepackage{bigints}
\usepackage{empheq}
\allowdisplaybreaks
\usepackage{dsfont}
\usepackage{amsfonts}
\usepackage{graphicx}
\usepackage{shadow}
\usepackage{color}
\usepackage{tcolorbox}
\usepackage[all]{xy}
\usepackage{tikz}
\usepackage{graphicx}
\usepackage{enumitem}
\usepackage{cases}

\def \d {\mathrm{~d}}

\def \n {\mathbf{n}}

\received{31 October 2023}
\revised{xxxxxx}
\accepted{xxxxx}

\begin{document}

\title{On uniform null controllability of transport-diffusion equations with vanishing viscosity limit}

\author[1]{Fouad ET-TAHRI}

\author[2]{Jon Asier B\'arcena-Petisco}

\author[3]{Idriss Boutaayamou*}

\author[4,5]{Lahcen Maniar}

\authormark{Fouad ET-TAHRI \textsc{et al}}

\address[1]{\orgdiv{Lab-SIV, Faculty of Sciences-Agadir}, \orgname{Ibnou Zohr University}, \orgaddress{\state{B.P. 8106, Agadir}, \country{Morocco}}}

\address[2]{\orgdiv{Department of Mathematics}, \orgname{University of the Basque Country UPV/EHU}, \orgaddress{\state{Barrio Sarriena s/n, 48940, Leioa}, \country{Spain}}}

\address[3]{\orgdiv{Lab-SIV, Multidisciplinary Faculty-Ouarzazate}, \orgname{Ibnou Zohr University}, \orgaddress{\state{BP 638, Ouarzazate 45000}, \country{Morocco}}}

\address[4]{\orgdiv{Cadi Ayyad University, Faculty of Sciences Semlalia}, \orgname{LMDP, UMMISCO (IRD-UPMC)}, \orgaddress{\state{B.P. 2390, Marrakesh}, \country{Morocco}}}

\address[5]{\orgdiv{University Mohammed VI Polytechnic, Vanguard Center}, \orgname{LMDP, UMMISCO (IRD-UPMC)}, \orgaddress{\state{Benguerir}, \country{Morocco}}}

\corres{*Idriss Boutaayamou. \email{dsboutaayamou@gmail.com}}


\abstract[Summary]{This paper aims to address an interesting open problem,
	posed in the paper "Singular Optimal Control for a
	Transport-Diffusion Equation" of Sergio Guerrero and Gilles Lebeau in 2007. The problem involves studying the null controllability cost of a transport-diffusion equation with Neumann conditions, where the diffusivity coefficient is denoted by $\varepsilon>0$ and the velocity by $\mathfrak{B}(x,t)$. Our objective is twofold. First, we investigate the scenario where each velocity trajectory $\mathfrak{B}$ originating from $\overline{\Omega}$ enters the control region in a shorter time at a fixed entry time.
	By employing Agmon and dissipation inequalities, and Carleman estimate in the case $\mathfrak{B}(x,t)$ is the gradient of a time-dependent scalar field, we establish that the control cost remains bounded for sufficiently small $\varepsilon$ and large control time. Secondly, we explore the case where at least one trajectory fails to enter the control region and remains in $\Omega$. In this scenario, we prove that the control cost explodes exponentially when the diffusivity approaches zero and the control time is sufficiently small for general velocity.}


\maketitle

\section{Introduction and main results}
Transport-diffusion equations with vanishing diffusivity are widely used to model various physical and biological phenomena. They play a significant role in fluid dynamics by describing the movement of particles while accounting for transport and diffusion effects, as it is explained in Chapter 3
of \cite{bahouri2011fourier} and the references therein.
\par 
Let $\Omega \subset \mathbb{R}^{d}$, $d\geq 1$ be a bounded open set,  $\Gamma$ denote the boundary of $\Omega$, $\mathbf{n}$ represent the outward unit normal field on $\Gamma$ and $\omega \subset \Omega$ be a nonempty open subset. 
Throughout this paper, the following notation will be consistently employed:
\begin{eqnarray*}
	\Omega_{T}:=\Omega\times (0,T), \quad \omega_{T}:=\omega\times (0,T) \quad \text{and} \quad \Gamma_{T}:=\Gamma\times (0,T),
\end{eqnarray*}
where $T>0$ is the control time. The main goal of this work is to study the cost of controllability of the following parabolic–transport equation with Neumann boundary conditions:
\begin{equation}\label{s1}
	\left\{
	\begin{aligned}
		\partial_t y-\varepsilon\Delta y+\mathfrak{B}(x,t)\cdot\nabla y &=u(x,t)\mathds{1}_{\omega} & & \text {in}\;\Omega_{T}, \\
		\partial_{\mathbf{n}} y &=0 & & \text {on}\;\Gamma_{T}, \\
		y(x,0)&=y_{0}(x) & & \text {in } \Omega,  \\	
	\end{aligned}
	\right.
\end{equation}
where $\varepsilon>0$ is the viscosity (diffusion coefficient) and $\mathfrak{B}$ is the velocity (speed), that satisfies:
\begin{equation*}
	\mathfrak{B}(x,t)=\nabla f(x,t)\quad \mbox{with }\quad f\in W^{2,\infty}(\mathbb{R}^{d}\times (0,\infty))
\end{equation*}
in case of uniform controllability in $\varepsilon$. The reason for this spatial extension (a regular open strictly containing $\Omega$ is sufficient) is to define geometric conditions for the trajectories of the vector field $\mathfrak{B}$, and the extension of time, to have norms of $\mathfrak{B}$ that do not depend on $T$ as considered in \cite{guerrero2007singular}. In the case of the controllability cost explosion, we only assume that  
\begin{equation*}
	\mathfrak{B}\in L^{\infty}(0,T; W^{1,\infty}(\Omega)^{d}),
\end{equation*}
where $\mathfrak{B}$ is a general velocity. In system \eqref{s1}, $y=y(x,t)$ represents the state, $u\in L^{2}(\omega_{T})$ is the control function that only affects the system through $\omega_{T}$, $\mathds{1}_{\omega}$ denotes the characteristic function of $\omega$ and $y_{0}\in L^{2}(\Omega)$ is the initial state. Let us start with the definition of null controllability:
\begin{definition}
	We say that system \eqref{s1} is null controllable at time $T$ if for every state $y_{0}\in L^{2}(\Omega)$, there exists a control $u\in L^{2}(\omega_{T})$ such that, the solution of \eqref{s1} satisfies $y(\cdot,T)=0$.
\end{definition}
\par
It is well known that system \eqref{s1} is null controllable for any control time $T$ and any control region $\omega$ when $\mathfrak{B}\in L^{\infty}(\Omega_{T})$. Specifically, in the case where $\varepsilon=1$, we can refer to \cite{fursikov1996controllability} and\cite{fernandez2006null} for further details. Furthermore, by scaling the time variable using the transformation $t\mapsto\frac{t}{\varepsilon}$, we can effectively reduce the problem to this specific case, as shown in \cite[Proposition 3.1]{et2024asymptotic}. It is a well-known fact that the controls depend continuously on the initial data. In other words, there exists a constant $C:=C(T,\varepsilon)>0$ such that
\begin{equation}
	\|u\|_{L^{2}(\omega_{T})}\leqslant C\|y_{0}\|_{L^{2}(\Omega)} \label{controllability}
\end{equation}
and the null controllability 
is equivalent to the following observability inequality:
\begin{eqnarray}
	\exists\,\mathcal{C}:=\mathcal{C}(T,\varepsilon)>0, \; \forall \varphi_{T}\in L^{2}(\Omega),\; \|\varphi(\cdot, 0)\|_{L^{2}(\Omega)}\leq \mathcal{C} \|\varphi\|_{L^{2}(\omega_{T})}, \label{observability}
\end{eqnarray}
where $\varphi$ is the solution of the adjoint system of \eqref{s1}:
\begin{equation}
	\left\{
	\begin{aligned}
		\partial_t \varphi+\varepsilon\Delta \varphi+\nabla \cdot\left(\varphi \mathfrak{B}(x,t)\right)  &=0 & & \text { in } \Omega_{T}, \\
		\left(\varepsilon \nabla\varphi +\varphi \mathfrak{B}(t,x)\right)\cdot\mathbf{n}(x) &=0 & & \text { on } \Gamma_{T}, \\
		\varphi(x,T)&=\varphi_{T}(x) & & \text { in } \Omega.
		\label{s2}
	\end{aligned}
	\right.
\end{equation}
\par
By employing the Hilbert Uniqueness Method \cite{lions1988controlabilite,russell1978controllability}, it can be shown that the optimal constants satisfying \eqref{controllability} and \eqref{observability} are equal, that is,
\begin{equation*}
	\sup_{y_{0}\in L^{2}(\Omega)\setminus\{0\}}\inf_{u\in \mathbb{A}(y_{0})}\frac{\| u\|_{L^{2}(\omega_{T})}}{\| y_{0}\|_{L^{2}(\Omega)}}=\sup_{\varphi_{T}\in L^{2}(\Omega)\setminus\{0\}}\frac{\| \varphi(\cdot, 0)\|_{L^{2}(\Omega)}}{\| \varphi\|_{L^{2}(\omega_{T})}},
\end{equation*}
where $\varphi$ is the solution of the adjoint system \eqref{s2} and $$\mathbb{A}(y_{0}):=\left\{u\in L^{2}(\omega_{T}): \;\text{the solution of}\;\eqref{s1}\;\text{satisfies}\;y(\cdot,T)=0 \right\}.$$
\par
In the sequel, we will adopt the following definition:
\begin{definition} 
	We define the cost of null controllability of system \eqref{s1} by the following quantity: 
	\begin{eqnarray}
		\mathcal{K}(\varepsilon,T,\Omega,\omega):=\sup_{\varphi_{T}\in L^{2}(\Omega)\setminus\{0\}}\frac{\| \varphi(\cdot, 0)\|_{L^{2}(\Omega)}}{\| \varphi\|_{L^{2}(\omega_{T})}}. \label{cost of control}
	\end{eqnarray}
\end{definition}
The main objective of this paper is to investigate the asymptotic behavior of the cost of null controllability of system \eqref{s1} when the viscosity is small enough. To elucidate the main findings of this paper, let us examine the trajectories of the vector field $\mathfrak{B}$ given by the mapping $t\mapsto \varPhi(t,t_{0},x_{0})$:
\begin{equation} \label{OD}
	\left\{
	\begin{aligned}
		\frac{d}{dt}\varPhi(t,t_{0},x_{0}) &=\mathfrak{B}(\varPhi(t,t_{0},x_{0}),t) & & t\in (0,T), \\
		\varPhi(t_{0},t_{0},x_{0}) &=x_{0},
	\end{aligned}
	\right.
\end{equation}
for each $(x_{0},t_{0})\in \mathbb{R}^{d}\times [0,T]$. The solutions of the ordinary differential equation \eqref{OD} encompass all relevant information regarding the trajectories of a particle moving with velocity $\mathfrak{B}$.
The following notations will be useful in the sequel.\\
\textbf{Notations.}
\begin{enumerate}[label=(N\arabic*), ref=(N\arabic*)]
	\item  The canonical Euclidean scalar product of $\mathbb{R}^{d}$ is denoted by $\cdot$ and $\left|\cdot\right|$ stands for the
	associated canonical Euclidean norm.
	\item For all $x,y\in\mathbb{R}^{d}$ and all $r>0$, $B(x,r)$  and $\overline{B}(x,r)$ are the open and closed balls of center $x$ and radius $r$, respectively.
	\item For 
    $A,B\subset\mathbb{R}^{d}$,    $\mbox{dist}(x,A)$ and $\mbox{dist}(A,B)$ designate the distance from $x$ to $A$ and the distance between $A$ and $B$, respectively.
\end{enumerate}
The following definition serves a specific purpose that is essential to prove our first main result.
\begin{definition} \label{definition}
	Let $T_{0}\in (0,T)$, $r_{0}>0$ and $\mathcal{O}\subset\mathbb{R}^{d}$ nonempty open. We say that $(T,T_{0},r_{0},\mathfrak{B},\Omega)$ satisfies the flushing condition \eqref{Flushing Condition} for $\mathcal{O}$ if 
	\begin{equation}\tag{$\mathcal{F}\mathcal{C}$} \label{Flushing Condition}
		\forall x_{0}\in\overline{\Omega},\;\forall t_{0}\in [T_{0},T],\; \exists\, t\in (t_{0}-T_{0},t_{0}),\;\forall x\in\overline{B}(x_{0},r_{0}), \varPhi(t,t_{0},x_{0})\in \mathcal{O}. 
	\end{equation}
    
\end{definition}
\begin{remark} \label{Remark 1}
 When the ODE \eqref{OD} is autonomous, i.e., $\mathfrak{B}=\mathfrak{B}(x)$, we can characterize \eqref{Flushing Condition} by any backward trajectories of $\mathfrak{B}$ originating from $\overline{\Omega}$ at time $0$ enter the open set $\mathcal{O}$. This characterization is discussed in Proposition \ref{P5} of Section \ref{Section 2}.
\end{remark}
We will show that if every backward trajectory of $\mathfrak{B}(x,t)$ starting from $\overline{\Omega}$ enters the control region within a time that does not surpass a fixed time barrier, then the cost of null controllability of \eqref{s1} remains uniformly bounded with respect to $\varepsilon$ when it is small enough and the control time is sufficiently large. To be more precise, we establish the following theorem.
\begin{theorem}
	Under the following conditions:
	\begin{enumerate}[label=(\arabic*), ref=(\arabic*)]
		\item $\Omega\subset \mathbb{R}^{d}$ is a $\mathcal{C}^{2}$ domain, $d\geq 1$ and $\omega\subset\Omega$ is a nonempty open subset, 
		\item $\mathfrak{B}=\nabla f$ with $f\in W^{2,\infty}(\mathbb{R}^{d}\times (0,\infty))$,
		\item \label{cc3} there exist $T_{0}\in (0,T)$ and $r_{0}>0$ such that, $(T,T_{0},r_{0},\mathfrak{B},\Omega)$ satisfies \eqref{Flushing Condition} for $\omega$, 
		\item $\forall (x,t)\in\Gamma\times (0,\infty),\; \partial_{\n}f(x,t)\geq c$ for some $c>0$.
	\end{enumerate}
	There exists a constant $\rho_{0}\geq 1$ depending only on $T_{0}$, $r_{0}$ and $f$ such that, if $T\geq \rho_{0} T_{0}$, there exists a 
	constant $C>0$ independent of $\varepsilon$ that satisfies the following estimate:
	\begin{equation} \label{cost1}
		\mathcal{K}(\varepsilon,T,\Omega,\omega)\leqslant C,\quad\text{for}\;\varepsilon\;\text{small enough,}
	\end{equation}
	where $\mathcal{K}$ is the cost of null controllability of \eqref{s1}. \label{m1}
\end{theorem} 
\par 
The proof of Theorem \ref{m1} will be in Appendix A.
Now, we present an example that illustrate the conditions of Theorem \ref{m1}.
\begin{eexample} In any dimension $d\geq 1$, we consider $\Omega=B(0,1)$, $\omega\subset \Omega$ an open subset contains $0$ and $f(x)=\frac{|x|^{2}}{2}$ for $x\in\Omega$ and has a compact support in $\mathbb{R}^{d}$ (to obtain $f\in W^{2,\infty}(\mathbb{R}^{d})$). Indeed, in this case $\n(x)=x$, then $\partial_{\n}f(x)=1$ for all $\left|x\right|=1$ and for any $x_{0}\in \overline{B}(0,1)$, the solution of the system:
		\begin{equation*} 
			\left\{
			\begin{aligned}
				\frac{d}{dt}\Psi(t,0,x_{0}) &=-\nabla f (\Psi(t,0,x_{0})) & & t\geq 0, \\
				\Psi(0,0,x_{0}) &=x_{0}
			\end{aligned}
			\right.
		\end{equation*}
  is given by $\Psi(t,0,x_{0})=e^{-t}x_{0},\; t\geq 0$. Hence
		$$\forall x_{0}\in \overline{B}(0,1),\;\;\exists t\in (-\infty,0),\;\; \Phi(t,0,x_{0})=\Psi(-t,0,x_{0})\in \omega.$$
Thus, condition \ref{cc3} of Theorem \ref{m1} is satisfied (see Remark \ref{Remark 1}).
\end{eexample} 
The literature investigates two methods: the spectral approach, illustrated in \cite{barcena2021cost}, and the Agmon inequalities-based approach, illustrated in \cite{guerrero2007singular}. Both employ Carleman and dissipation estimates. In this work, we will use the second approach because the transport term depends on the time variable.
\begin{remark}
The problem when $\mathfrak B$ is not
a gradient field is open. We now explain what we can ensure and what remains open:
\begin{itemize}
    \item A dissipation estimate is satisfied by Agmon inequality for a general transport $\mathfrak{B}(x,t)$ such that $\mathfrak{B}(x,t)\cdot\n(x)\geq 0$ and satisfies the flushing condition \eqref{Flushing Condition}. We prove those results in Subsection \ref{Dissipartion results}.
    \item The difficulty arises in the proof of Carleman estimate for the solutions of system \eqref{s2} which leads to observability constant of the form $\exp\left(\frac{C}{\varepsilon}\left(1+\frac{1}{T}\right)\right)$
for a constant $C>0$ independent
of $\varepsilon$ and $T$. Indeed, in the computations of Carleman estimate, we find the term \eqref{Tcar} in the right-hand side, which is difficult to absorb.
    \begin{eqnarray}
        \varepsilon \int_{\Gamma_{T}}(\mathfrak{B}(x,t)\cdot\n(x))\exp(-2s\alpha)|\nabla\varphi|^{2}\d \sigma\d t+2\int_{\Gamma_{T}}\exp(-2s\alpha)(\mathfrak{B}(x,t)\cdot\nabla\varphi)(\mathfrak{B}(x,t)\cdot\n(x))\varphi\mathrm{~d}\sigma\mathrm{~d}t. \label{Tcar}
    \end{eqnarray}
 This is the reason why we have imposed on $\mathfrak{B}(x,t)$ to be a gradient field $\nabla f(x,t)$, which allows us to transform system \eqref{s2} to system \eqref{S3} without a transport term. A Carleman estimate is proved for the solutions of the new system because it avoids terms of type \eqref{Tcar}.The assumption $\partial_{\n}f\geq c$ allows us to absorb a boundary term that remains in the right-hand side of the Carleman estimate.
    \item In the case of general transport with $\mathfrak{B}(x,t)\cdot \n(x)=0$, a Carleman estimate can be proved (the term \eqref{Tcar} is null), but the flushing condition \eqref{Flushing Condition} is never satisfied by $\omega\subset\Omega$, due to $\partial\Omega$ being a
periodic integral curve of $\mathfrak{B}$ that never reaches $\omega$. Under that hypothesis,
we can satisfy the flushing condition \eqref{Flushing Condition} for any trajectory that starts in $\Omega$ (but not in $\partial\Omega$), as shown in the following example: $d=2$, $\Omega=B((0,0),1)$, $\omega$ an open subset contains $(0,0)$ and $\mathfrak{B}(x,y)=(-x+y+x(x^{2}+y^{2}), -x-y+y(x^{2}+y^{2}))$. Using the Lyapunov function $V(x,y)=x^{2}+y^{2}$, we can show that $(0,0)$ is not globally asymptotically stable and its basin of attraction is $B((0,0),1)$, as shown in Figure \ref{Fig1}. In this case, the asymptotic behavior of the controllability cost is an open problem, even in the case of Dirichlet boundary conditions.
\begin{figure}[ht]
\begin{center} 
\includegraphics[width=1.8 in]{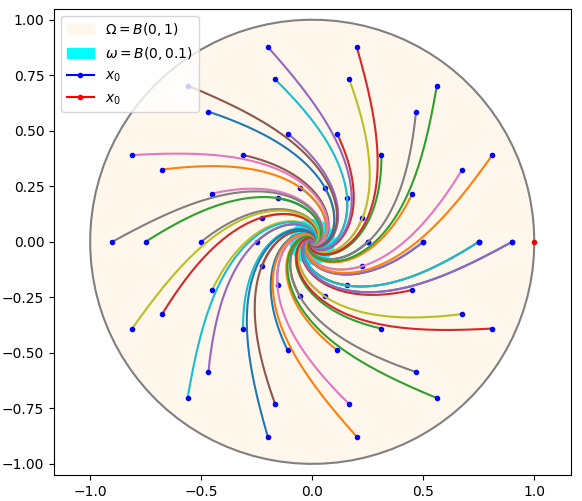} 
\end{center}
\caption{Backward trajectories $t\mapsto\Phi(t,0,x_{0})$ starting in $\Omega$ enter $\omega$ and those starting in $\Gamma$ remain in $\Gamma$.}
\label{Fig1}
\end{figure}
\end{itemize}
\end{remark}
The following remark concerns the minimal time for uniform controllability.
\begin{remark}
The minimal time needed to achieve uniform controllability, i.e.
\begin{eqnarray*}
    T_M:=\min\{T>0\;:\; \text{there is}\; C>0\;\text{such that}\; \mathcal{K}(\varepsilon, T, \Omega, \omega)\leq C,\;\text{for all}\;\varepsilon\;\text{small enough}\}
\end{eqnarray*}
has been studied and several upper and lower estimates have been obtained in the literature. Notably, they have been obtained in the one-dimensional case with constant transport coefficient and a boundary control in \cite{coron2005singular, glass2010complex, lissy2015explicit} with different
methods, and more recently in \cite{laurent2022uniform} for space-dependent transport in gradient form. It is important to remark that a formula for $T_M$ is not known even in the one dimensional heat equation with constant coefficients, which shows the toughness of the problem. 
In this paper, no estimate of the minimal time is provided, as in \cite{guerrero2007singular} for Dirichlet boundary conditions. Based on our analysis, we simply have $T_M\leq \rho_0 T_0$, where $T_0$ is the optimal time for the flushing condition and  the coefficient $\rho_0$ depends on the dissipation and observability constants. The challenge lies in the difficulty of explicitly identifying this coefficient.
One might conjecture that the minimal time for uniform controllability corresponds, up to a factor, to the optimal time required to reach the control region. However, this conjecture is far from obvious, as it would require thorough analysis to confirm, and it appears challenging to address.
\end{remark}
The second main result of this paper is to show that if there is a backward trajectory of $\mathfrak{B}$ that starts in $\Omega$ at time $T$ stays in $\Omega$ and does not enter the region $\overline{\omega}$ during time $[0,T]$, then for a small time control, the control cost explodes exponentially when the viscosity vanishes. To be more precise, we provide a proof of the following theorem in Appendix B.
\begin{theorem} \label{m2}
	We assume that:
	\begin{enumerate}[label=(\arabic*), ref=(\arabic*)]
		\item $\Omega\subset\mathbb{R}^{d}$ is a domain with Lipschitz boundary $\Gamma$, $d\geq 1$ and $\omega\subset\Omega$ is a nonempty open subset, 
		\item $\mathfrak{B}\in L^{\infty}(0,T; W^{1,\infty}(\Omega)^{d})$,
		\item \label{C3m2} $\exists x_{0}\in\Omega$ such that, for all $t\in [0,T]$, $\Phi(t,T,x_{0})\in \Omega\setminus\overline{\omega}$.
	\end{enumerate}
	Then there exists a constant $C>0$ independent of $\varepsilon$ such that we have the following estimate:
	\begin{equation} \label{cost2}
		\mathcal{K}(\varepsilon,T,\Omega,\omega)\geq\exp\left(\frac{C}{\varepsilon}\right),\quad\text{for}\;\varepsilon\;\text{small enough,}
	\end{equation}
	where $\mathcal{K}$ is the cost of the null controllability of \eqref{s1}.
\end{theorem}
Generally, condition \ref{C3m2} of Theorem \ref{m2} holds true, provided that $T$ is sufficiently small, using a continuity argument. For instance:
\begin{eexample} Assume that $\omega\subset\subset \Omega$ and $\mathfrak{B}\in L^{\infty}(\Omega\times (0,\infty))^{d}$.
	Let $x_{0}\in \Omega\setminus\overline{\omega}$, $r=\mbox{dist}(x_{0},\omega)$ and $b=\|\mathfrak{B}\|_{L^{\infty}(\Omega\times (0,\infty))}$. One has
	\begin{eqnarray*}
		\left|\Phi(t,T,x_{0})-x_{0}\right| &=&\left|\int_{T}^{t}\mathfrak{B}(\Phi(s,T,x_{0}),s)\d s\right|\leq Tb,\; \mbox{for all}\; t\in [0,T].
	\end{eqnarray*}
	Taking $T>0$ such that $0< Tb <r$, we obtain $\Phi(t,T,x_{0})\in B(x_{0},r)\subset \Omega\setminus\overline{\omega}$.
	\par 
	We can have this property for all $T>0$, as the following example shows: in 2-D, let $\mathfrak{B}(x,y)=(y,-x)$. In this case, the matrix associated with the equation \eqref{OD} is skew-symmetric. As a result, we have $\left|\Phi(t,T,x_{0})\right|=\left|x_{0}\right|$ for any $(x_{0},t)\in\mathbb{R}^{2}\times \mathbb{R}$, as shown in Figure \ref{Fig2}:
	\begin{figure}[ht] 
		\begin{center}
			\includegraphics[width=1.5 in]{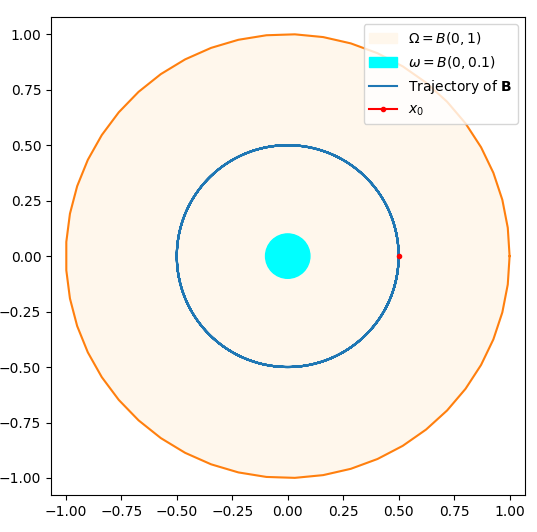}
		\end{center}
		\caption{A trajectory of $\mathfrak{B}$ in $\Omega\setminus\overline{\omega}$.} 
               \label{Fig2}
	\end{figure}
\end{eexample}
\begin{remark}
	Theorem \ref{m2} is a generalization of \cite[Theorem 2.8]{barcena2021cost} for a small time control. In fact, the conditions specified in \cite[Theorem 2.8]{barcena2021cost} lead directly to the realization of condition \ref{C3m2} of Theorem \ref{m2} for $T<h$ and $x_{0}\in (p_{l}+T,p_{l}+h)$.
\end{remark} 
\par 
In the context of the problem under study, Sergio Guerrero and Gilles Lebeau established analog results in their work, specifically in \cite[Theorem 1 and Theorem 3]{guerrero2007singular}, with a transport flow belonging to $W^{1,\infty}(\mathbb{R}^{d}\times (0,+\infty))$ and employing Dirichlet conditions. Jon Asier B\'arcena-Petisco, in \cite[Theorem 2.7 and Theorem 2.8]{barcena2021cost}, demonstrated the same results for the case of the first vector of the canonical basis of $\mathbb{R}^{d}$ as a velocity and autonomous Robin (or Fourier) conditions. The generalization of these results to a velocity field expressed as a gradient field belonging to $W^{1,\infty}(\Omega)^{d}$ was accomplished by \cite{et2024asymptotic}, extending the findings of \cite[Theorem 2.7 and Theorem 2.8]{barcena2021cost}. Additionally, in \cite{laurent2021uniform} Camille Laurent and Matthieu Léautaud investigated uniform controllability and the corresponding optimal time for homogeneous Dirichlet conditions on a smooth, connected, compact manifold, while in \cite{laurent2022uniform} they considered the analogous scenario for 1-D systems.
\par 
Our contribution is to answer the open questions presented in \cite[Remark 3]{guerrero2007singular} and \cite{barcena2021cost}. The objective then is to study the null controllability cost of a transport-diffusion equation with Neumann conditions and velocity $\mathfrak{B}(x,t)$ in the form of a gradient field also depends on the time variable in the case of uniform controllability.
The main difficulty encountered initially was to establish an Agmon inequality. However, this obstacle has been overcome with the help of estimate \eqref{Trace estimate} while still using the tools presented in \cite{guerrero2007singular}, which allowed for the problem's resolution. Subsequently, the second challenge was to establish a new Carleman estimate, yielding an observability constant of the form $e^{C/\varepsilon}$. 
Furthermore, it should be noted that the vector $\mathfrak{B}$ depends on both $x$ and $t$, which prevents the use of the spectral approach, as is possible in cases \cite{barcena2021cost,et2024asymptotic}. In the proof of Theorem \ref{m1}, we use a new decomposition (see the system introduced in Section \ref{Section 3}) of the adjoint system \eqref{s2} to prove estimates within the control region and outside, using the uniform Agmon inequality in $\varepsilon$ and the Carleman estimate.  In the proof of Theorem \ref{m2}, we construct a solution of the adjoint system \eqref{s2}, considering smooth initial data and exploiting Agmon inequality.
\par 
From a historical standpoint, the problem under investigation has its roots in the field of controllability problems within singular limits, which were initially introduced by Jacques-Louis Lions. The articles \cite{lopez2000null} of Antonio L\'opez, Xu Zhang, and Enrique Zuazua and \cite{phung2002null} of Kim Dang Phung provide an illustration of how the null controllability of the heat equation emerges as a singular limit from the exact controllability of dissipative wave equations. Subsequently, the specific focus of the study revolves around the evanescent viscosity limit, which was first introduced by Jean-Michel Coron and Sergio Guerrero in their work \cite{coron2005singular}. Initially explored in the context of 1-D transport equations, this problem was subsequently extended to higher dimensions by Sergio Guerrero and Gilles Lebeau in their work cited as \cite{guerrero2007singular}. To explore other references related to optimal time of the controllability with
vanishing viscosity limit of the heat equation see \cite{glass2010complex} of Olivier Glass and \cite{lissy2012link,lissy2014application,lissy2015explicit} of Pierre Lissy and the reference therein.  For the motivation behind the problem and other applications, we invite you to refer to the introductions of the mentioned works \cite{laurent2021uniform,barcena2021cost,et2024asymptotic}. The last interesting paper involving parabolic equations, specifically the fourth order equation, is \cite{carreno2016cost} of Nicolas Carreño and Patricio Gúzman.
\par 
Our paper is organized as follows. In Section \ref{Section 2}, we present several results related to property \eqref{Flushing Condition}. Moving on to Section \ref{Section 3}, we examine the existence and uniqueness of both strong and weak solutions for a parabolic system that includes the adjoint system \eqref{s2}. In Section \ref{Section 4}, we will prove some new Agmon inequalities and significant dissipation results and we present a new Carleman estimate for the adjoint system \eqref{s2} which will be shown in Appendix C while Appendices A and B are devoted to the proof of our main results, Theorem \ref{m1} and Theorem \ref{m2}. In the last section, we present a conclusion to this work.
\section{Some preliminary results of flushing condition} \label{Section 2}
In this section, we present two relevant results concerning property \eqref{Flushing Condition}. The first result offers a characterization of this property specifically in the autonomous case, while the second result introduces a refinement of the regions associated with this property, which will be used later in our analyses.\\
\par 
The following lemma ensures the existence and uniqueness of differentiable solutions of the ordinary differential equation \eqref{OD}.
\begin{lemma}  Let $\mathfrak{B}\in L^{\infty}(0,T;W^{1,\infty}(\mathbb{R}^{d})^{d})$.
	For all $(x_{0},t_{0})\in \mathbb{R}^{d}\times [0,T]$, the ordinary differential equation \eqref{OD} admits a unique global differentiable solution $\Phi$. Moreover for all $t\in [0,T]$ and all $(x_{0},t_{0}), (y_{0},s_{0})\in\mathbb{R}^{d}\times [0,T]$
	\begin{eqnarray}
		&&|\Phi(t,t_{0},x_{0})-\Phi(t,s_{0},y_{0})|\leq \exp\left(\|\nabla\mathfrak{B}\|_{L^{\infty}(\mathbb{R}^{d}\times (0,T))}T\right)\left(\|\mathfrak{B}\|_{L^{\infty}(\mathbb{R}^{d}\times (0,T))}|t_{0}-s_{0}|+|x_{0}-y_{0}|\right). \label{f1}
	\end{eqnarray} 
\end{lemma}
\begin{proof}
	For all $t\in (0,T)$, $\mathfrak{B}(\cdot,t)\in W^{1,\infty}(\mathbb{R}^{d})^{d}$ and $\mathbb{R}^{d}$ (is convex),  from \cite{brezis1983analyse} for scalar-valued functions, we deduce that
	\begin{eqnarray}
		|\mathfrak{B}(x,t)-\mathfrak{B}(y,t)|
		&\leqslant & \|\nabla\mathfrak{B}(\cdot,t)\|_{L^{\infty}(\mathbb{R}^{d})}|x-y| \nonumber\\
		&\leqslant & \|\nabla\mathfrak{B}\|_{L^{\infty}(\mathbb{R}^{d}\times (0,T))}|x-y|.\label{odf1}
	\end{eqnarray}
	where $\mathfrak{B}:=(\mathfrak{B}_{1},\cdots,\mathfrak{B}_{d})$ and $\|\nabla\mathfrak{B}(\cdot,t)\|^{2}_{L^{\infty}(\mathbb{R}^{d})}:=\displaystyle\sum_{1\leq i,j\leq d}\|\partial_{x_j}\mathfrak{B}_{i}(\cdot,t)\|^{2}_{L^{\infty}(\mathbb{R}^{d})}$.
	The Cauchy-Lipschitz theorem affirms that for all $(x_{0},t_{0})\in \mathbb{R}^{d}\times [0,T]$, \eqref{OD} has a unique global solution given by 
	\begin{eqnarray}
		\Phi(t,t_{0},x_{0})=x_{0}+\int_{t_{0}}^{t}\mathfrak{B}(\Phi(s,t_{0},x_{0}),s)\mathrm{~d} s,\quad 0\leqslant t\leqslant T. \label{odf2}
	\end{eqnarray}
	Let $(x_{0},t_{0}), (y_{0},s_{0})\in \mathbb{R}^{d}\times [0,T]$ with $t_{0}\leqslant s_{0}$, from \eqref{odf1} and \eqref{odf2}, we obtain
	\begin{eqnarray}
		&&\left|\Phi(t,t_{0},x_{0})-\Phi(t,s_{0},y_{0})\right| \leqslant |x_{0}-y_{0}|+ \|\mathfrak{B}\|_{L^{\infty}(\mathbb{R}^{d}\times (0,T))}|s_{0}-t_{0}| + \|\nabla\mathfrak{B}\|_{L^{\infty}(\mathbb{R}^{d}\times (0,T))}\int_{\min(s_{0},t)}^{\max(s_{0},t)}\left|\Phi(s,t_{0},x_{0})-\Phi(s,s_{0},y_{0})\right|\mathrm{~d} s.\nonumber
	\end{eqnarray}
	Applying Grönwall's lemma to this last inequality, we obtain \eqref{f1}. 
\end{proof}
\par 
In the autonomous case, we can characterize condition \eqref{Flushing Condition} as follows.
\begin{proposition} \label{P5}
	Let $\mathcal{O}$ a nonempty open set of $\mathbb{R}^{d}$ and $\mathfrak{B}\in W^{1,\infty}(\mathbb{R}^{d})^{d}$.\\
    Assume that 
	$$\forall x_{0}\in\overline{\Omega},\exists\, t\in (-\infty,0),\; \Phi(t,0,x_{0})\in\mathcal{O}.$$
	Then there exist $T_{0}>0$ and $r_{0}>0$ such that, for all $T>T_{0}$, $(T,T_{0},r_{0},\mathfrak{B},\Omega)$ satisfies condition \eqref{Flushing Condition} for $\mathcal{O}$.
\end{proposition}
\begin{proof}
	For all $x_{0}\in\overline{\Omega}$ there exists $t:=t(x_{0})\in (-\infty,0)$ such that, $ \Phi(t(x_{0}),0,x_{0})\in\mathcal{O}.$ By the continuity of the flow, there exists $r:=r(x_{0})>0$ such that,
	\begin{eqnarray}
		|x-x_{0}|\leq 2r(x_{0})\Longrightarrow \Phi(t(x_{0}),0,x)\in\mathcal{O}. \label{od4}
	\end{eqnarray}
	By compactness of $\overline{\Omega}$, there exist $x^{1}_{0},\cdots,x^{I}_{0}\in \overline{\Omega}$ such that, $\overline{\Omega}\subset\displaystyle\bigcup_{i=1,\cdots, I}B(x^{i}_{0},r(x^{i}_{0}))$. We put $0<T_{0}<-\displaystyle\min_{i=1,\cdots, I}t(x^{i}_{0})$ and $r_{0}:= \displaystyle\min_{i=1,\cdots, I}r(x^{i}_{0})$. From \eqref{od4}, we obtain
	\begin{eqnarray}
		\forall x_{0}\in\overline{\Omega},\exists\, t\in (-T_{0},0),\;|x-x_{0}|\leq r_{0}\Longrightarrow \Phi(t,0,x)\in\mathcal{O}. \label{od5}
	\end{eqnarray}
	Let $T>T_{0}$, $x_{0}\in\overline{\Omega}$ and $t_{0}\in [T_{0},T]$. Since $\mathfrak{B}$ is independent of $t$, then
	\begin{eqnarray}
		\Phi(t,t_{0},x)=\Phi(t-t_{0},0,x),\quad  (x,t)\in \mathbb{R}^{d}\times\mathbb{R}. \label{od6}
	\end{eqnarray}
	From \eqref{od5} and \eqref{od6}, we conclude that $(T,T_{0},r_{0},\mathfrak{B},\Omega)$ satisfies condition \eqref{Flushing Condition} for $\mathcal{O}$. 
\end{proof}
The following proposition guarantees that condition \eqref{Flushing Condition} remains true for small regions.
\begin{proposition} \label{P1}
	Let $\mathcal{O}\subset\mathbb{R}^{d}$ an open nonempty with bounded boundary and assume that $(T,T_{0},r_{0},\mathfrak{B},\Omega)$ satisfies condition \eqref{Flushing Condition} for $\mathcal{O}$. Then there exists $\mathcal{O}_{0}\subset \subset \mathcal{O}$ an open such that $\left(T,T_{0},\frac{r_{0}}{2},\mathfrak{B},\Omega\right)$ satisfies condition \eqref{Flushing Condition} for $\mathcal{O}_{0}$.
\end{proposition}
\begin{proof}
	For all $(x_{0},t_{0})\in \overline{\Omega}\times [T_{0},T]$, there exists $t:=t(x_{0},t_{0})\in (t_{0}-T_{0},t_{0})$ such that, for all $x\in \overline{B}(x_{0},r_{0})$ we have $\Phi(t(x_{0},t_{0}),t_{0},x)\in\mathcal{O}$. We set
	\begin{eqnarray}
		d(x_{0},t_{0}):=\text{dist}\left(\left\{\Phi(t(x_{0},t_{0}),t_{0},x):\;x\in \overline{B}\left(x_{0},\frac{r_{0}}{2}\right)\right\},\partial\mathcal{O}\right)>0, \label{od7}
	\end{eqnarray}
	since $\left\{\Phi(t(x_{0},t_{0}),t_{0},x):\;x\in \overline{B}\left(x_{0},\frac{r_{0}}{2}\right)\right\}$ is a closed thanks to \eqref{f1} and $\partial\mathcal{O}$ is a compact.  The continuity of the flow in \eqref{f1}, asserts that it exists $r:=r(x_{0},t_{0})>0$ such that, for all $|s-t_{0}|<r(x_{0},t_{0})$, we have
	\begin{eqnarray}
		\forall x\in \overline{B}(x_{0},r_{0}),\; |\Phi(t(x_{0},t_{0}),s,x)-\Phi(t(x_{0},t_{0}),t_{0},x)|\leq \frac{d(x_{0},t_{0})}{2}. \label{od8}
	\end{eqnarray}
	We then consider $\mathcal{U}_{x_{0},t_{0}}$ the set of couples $(x,t)$ satisfying 
	$$x\in B\left(x_{0},\frac{r_{0}}{2}\right),\quad t-T_{0}<t(x_{0},t_{0})<t\quad \text{and}\quad t\in (t_{0}-r(x_{0},t_{0}),t_{0}+r(x_{0},t_{0})).$$
	Since $\mathcal{U}_{x_{0},t_{0}}$ is on open containing $(x_{0},t_{0})$ and $\overline{\Omega}\times [T_{0},T]$ is compact, then it admits a finite covering by $\mathcal{U}_{x^{i}_{0},t^{i}_{0}}$, $i=1,\cdots,I$. Taking $d_{0}=\displaystyle\min_{i=1,\cdots,I}d(x^{i}_{0},t^{i}_{0})$ and using \eqref{od7} and \eqref{od8}, we can show that  $\left(T,T_{0},\frac{r_{0}}{2},\mathfrak{B}, \Omega\right)$ satisfies condition \eqref{Flushing Condition} for all open $\mathcal{O}_{0}$ such that $$\left\lbrace x\in\mathcal{O}:\;\text{dist}(x,\partial\mathcal{O})\geq \frac{d_{0}}{2}\right\rbrace \subset \mathcal{O}_{0}.$$ 
\end{proof}
\section{Wellposedness and results of a parabolic equation including adjoint system} \label{Section 3}
In this section, we will establish the well-posedness and regularity properties of solutions for the following backward, inhomogeneous linear transport-diffusion equation, accompanied by mixed boundary conditions Dirichlet and non-autonomous Robin conditions:
\begin{equation}  \label{s3}
	\left\{
	\begin{aligned}
		\partial_t \varphi+\varepsilon\Delta \varphi+\nabla \cdot\left(\varphi \mathfrak{B}(x,t)\right) &=F(x,t) & & \text { in } \mathcal{U}\times (t_{1},t_{2}), \\
		\left(\varepsilon \nabla\varphi +\varphi \mathfrak{B}(x,t)\right)\cdot \mathbf{n}(x)\mathds{1}_{\Gamma}(x)+\varphi\mathds{1}_{\Gamma_{0}}(x) &=0 & & \text { on } \partial\mathcal{U}\times (t_{1},t_{2}), \\
		\varphi(x,t_{2}) &=G(x) & & \text { in } \mathcal{U},\\
	\end{aligned} 
	\right.
\end{equation}
where $0\leqslant t_{1}<t_{2}\leqslant T$, and $\Omega_{0}\subset\subset\Omega$ a regular open, $\mathcal{U}:=\Omega\setminus\overline{\Omega_{0}}$, $\Gamma=\partial\Omega$, $\Gamma_{0}:=\partial\Omega_{0}$, $F\in L^{2}(\mathcal{U}\times (t_{1},t_{2}))$ and $G\in L^{2}(\mathcal{U})$. The figure \ref{Fig3} illustrates the geometric domains.\\
\begin{figure}[ht] 
		\begin{center}
			\includegraphics[width=1.4 in]{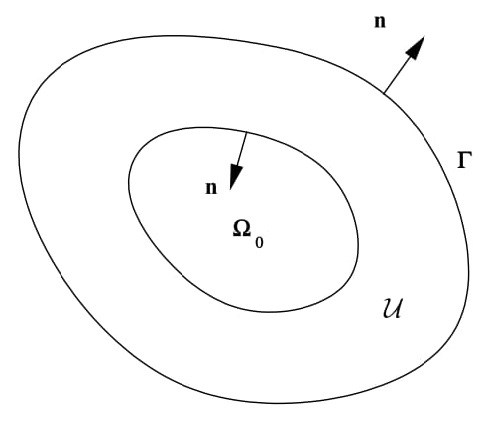}
		\end{center}
		\caption{Representation of $\mathcal{U}$.} \label{Fig3}
 \end{figure}
\textbf{Notation.}
In the following, $\mathcal{S}(\Omega_{0},t_{1},t_{2},F,G,\varepsilon,\mathfrak{B})$ will refer to system \eqref{s3}.
Note that in case $\Omega_{0}=\varnothing$, system $\mathcal{S}(\varnothing,0,T,0,\varphi_{T},\varepsilon,\mathfrak{B})$ is the adjoint system \eqref{s2}.
\subsection{Notations and function spaces}
Let $\Omega\subset\mathbb{R}^{d}$, $d\geq 1$ a domain with Lipschitz boundary $\Gamma$, $L^{2}(\Omega)$ and $L^{2}(\Gamma)$ are the classical Hilbert spaces over $\mathbb{R}$ with respect to the Lebesgue measure $\mathrm{~d} x$ on $\Omega$ and the $(d-1)$-dimensional Hausdorff measure $\mathrm{~d}\sigma$ on $\Gamma$, and $(\cdot,\cdot)$ is the canonical scalar product of $L^{2}(\Omega)$. We consider $H^{1}(\Omega)$ and $W^{k,\infty}(\Omega)$, $k=1,2$ the usual $L^2$ and $L^{\infty}$-based Sobolev spaces over $\Omega$, respectively, and $D(\Omega)$ the space of the test functions on $\Omega$. We recall that there exists a unique linear bounded operator $\gamma_{0} : H^{1}(\Omega)\longrightarrow L^{2}(\Gamma)$ such that $\gamma_{0}(u)=u_{|_{\Gamma}}$ (the restriction of $u$ on $\Gamma$) if $u\in H^{1}(\Omega)\cap \mathcal{C}(\overline{\Omega})$, see \cite{arendt2011dirichlet}. The quantity $\gamma_{0}(u)$ is called the trace of $u$ and one can also use the notation $u_{|_{\Gamma}}$ for $u\in H^{1}(\Omega)$ (to simplify, we note $u$ instead of $u_{|_{\Gamma}}$). In the sequel, we will employ the following $H^{1}(\Omega)-$trace estimate:
\begin{eqnarray}
	\int_{\Gamma}|u|^{2}d\sigma\leqslant C \|u\|_{H^{1}(\Omega)}\|u\|_{L^{2}(\Omega)}, \label{Trace estimate}
\end{eqnarray}
where $C>0$ depending only on $\Omega$.
For the proof of the inequality \eqref{Trace estimate}, we refer to \cite[Theorem. 1.5.1.10]{grisvard1985elliptic}.
\par
For any $\Omega_{0}\subset\subset\Omega$ a regular open, we set $\mathcal{U}=\Omega\setminus\overline{\Omega_{0}}$, $\mathcal{U}_{T}=\mathcal{U}\times (0,T)$, $\Gamma_{0}=\partial\Omega_{0}$, $\Gamma=\partial\Omega$ and 
we introduce $H^{1}_{\Gamma_{0}}(\mathcal{U})$ the space of all those functions in $H^{1}(\mathcal{U})$ whose trace vanishes on $\Gamma_{0}$:
\begin{eqnarray*}
	H^{1}_{\Gamma_{0}}(\mathcal{U}):= \begin{cases}
		\left\{u\in H^{1}(\mathcal{U}):\;\;u=0\;\text{on}\;\Gamma_{0} \right\},\quad &\text{if}\; \Omega_{0}\neq \varnothing,\\
		H^{1}(\Omega) ,\quad &\text{if}\; \Omega_{0}= \varnothing.
	\end{cases}
\end{eqnarray*}
We will keep this space the induced norm of $H^{1}(\mathcal{U})$.
We note by $H_{\Gamma_{0}}^{1}(\mathcal{U})^{\prime}$ the dual of $H_{\Gamma_{0}}^{1}(\mathcal{U})$ and the product duality is denoted by $\langle \cdot,\cdot\rangle_{H_{\Gamma_{0}}^{1}(\mathcal{U})^{\prime}, H_{\Gamma_{0}}^{1}(\mathcal{U})}$. Clearly $H^{1}_{\Gamma_{0}}(\mathcal{U})$ is dense in $L^{2}(\mathcal{U})$, as usual we can identify $L^{2}(\mathcal{U})$ with a
dense subspace of $H^{1}_{\Gamma_{0}}(\mathcal{U})^{\prime}$. Here, we use the following weak definition of normal derivative.
Let $u\in H^{1}_{\Gamma_{0}}(\mathcal{U})$ which satisfies $\Delta u \in L^2(\mathcal{U})$ and for $h \in L^2(\Gamma)$ we can define the equality $\partial_{\n}u|_{\Gamma}:=h$ in a weak sense by
\begin{equation}
	\int_{\mathcal{U}} \Delta u\,v \mathrm{~d} x+\int_{\mathcal{U}} \nabla u \cdot \nabla v \mathrm{~d}x=\int_{\Gamma} h\,v \mathrm{~d}\sigma \quad \forall v \in H^{1}_{\Gamma_{0}}(\mathcal{U}). \label{dn}
\end{equation}
In this case, the function $h \in L^2(\Gamma)$ verifying \eqref{dn} is unique, for further details, see \cite{tucsnak2009observation}. This means that we define the normal derivative $\partial_{\n}u|_{\Gamma}$ of $u$ on $\Gamma$ by the validity of Green's formula.
\subsection{Weak solutions of system $\mathcal{S}(\Omega_{0},t_{1},t_{2}, F,G,\varepsilon,\mathfrak{B})$}
The Lions' theorem \cite{Lions,showalter2013monotone} provides a significant framework for establishing the existence and uniqueness of weak solutions for  $\mathcal{S}(\Omega_{0},t_{1},t_{2},F,G,\varepsilon,\mathfrak{B})$. Considering  
\begin{eqnarray}
	\hat{\varphi}(\cdot,t)=\varphi(\cdot,t_{2}-t),\;\; \hat{F}(\cdot,t)=-F(\cdot ,t_{2}-t),\;\; \hat{\mathfrak{B}}(\cdot,t)=-\mathfrak{B}(\cdot,t_{2}-t)\;\;\text{and}\;\; \tau=t_{2}-t_{1}. \label{data}
\end{eqnarray}
Then $\varphi$ is a solution of $\mathcal{S}(\Omega_{0},t_{1},t_{2},F,G,\varepsilon,\mathfrak{B})$ if and only if $\hat{\varphi}$ is a solution of the following forward system:  
\begin{equation} \label{s6w}
	\left\{
	\begin{aligned}
		\partial_t \hat{\varphi}-\varepsilon\Delta \hat{\varphi} +\nabla\cdot (\hat{\varphi}\hat{\mathfrak{B}}(x,t)) &=\hat{F}(x,t)\quad &&\text { in } \mathcal{U}\times (0,\tau), \\
		(\varepsilon \nabla\hat{\varphi} -\hat{\varphi} \hat{\mathfrak{B}}(x,t))\cdot \mathbf{n}(x)\mathds{1}_{\Gamma}(x)+\hat{\varphi}\mathds{1}_{\Gamma_{0}}(x) &=0 \quad && \text { on } \partial\mathcal{U}\times (0,\tau), \\
		\hat{\varphi}(x,0) &=G(x) \quad &&\text { in } \mathcal{U}.
	\end{aligned}
	\right.
\end{equation}
Let us consider the bilinear form defined on $[0,\tau]\times H^{1}_{\Gamma_{0}}(\mathcal{U})\times H^{1}_{\Gamma_{0}}(\mathcal{U})$ by
\begin{eqnarray}
	\mathfrak{a}_{w}(t,u,v):=\varepsilon\int_{\mathcal{U}}\nabla u\cdot\nabla v\mathrm{~d}x-\int_{\mathcal{U}}u\hat{\mathfrak{B}}(x,t)\cdot\nabla v\mathrm{~d}x. \label{fw}
\end{eqnarray} 
\begin{definition}
	Let  $F\in L^{2}(t_{1},t_{2}; H_{\Gamma_{0}}^{1}(\mathcal{U})^{\prime})$ and $G\in L^{2}(\mathcal{U})$. A weak solution of \eqref{s3} is a function $u\in L^{2}(t_{1},t_{2};H_{\Gamma_{0}}^{1}(\mathcal{U}))\cap H^{1}(t_{1},t_{2}; H_{\Gamma_{0}}^{1}(\mathcal{U})^{\prime})$ such that
	\begin{eqnarray} \label{weak solution}
		&&-\int_{t_{1}}^{t_{2}}(u(t),v^{\prime}(t)) \d t-\int_{t_{1}}^{t_{2}}\mathfrak{a}_{w}(t_{2}-t,u(t),v(t))\d t=\int_{t_{1}}^{t_{2}}\langle F(\cdot,t),v(t) \rangle_{ H_{\Gamma_{0}}^{1}(\mathcal{U})^{\prime},H_{\Gamma_{0}}^{1}(\mathcal{U})}\d t - (G,v(t_{2})), 
	\end{eqnarray}
	for all $v\in H^{1}(t_{1},t_{2};L^{2}(\mathcal{U}))\cap L^{2}(t_{1},t_{2};H^{1}_{\Gamma_{0}}(\mathcal{U}))$ such that $v(t_{1})=0$.
\end{definition}
\begin{proposition} \label{pw} Let $\mathfrak{B}\in L^{\infty}(\mathcal{U}_{T})^{d}$, then for all $0\leqslant t_{1}< t_{2}\leqslant T$, system \eqref{s6w}, and hence system $\mathcal{S}(\Omega_{0},t_{1},t_{2}, F, G,\varepsilon,\mathfrak{B})$ has a unique weak solution. Moreover there exists a constant $C>0$ independent of $\varepsilon$ such that, the weak solution of $\mathcal{S}(\Omega_{0},t_{1},t_{2}, F, G,\varepsilon,\mathfrak{B})$ verifies 
	\begin{eqnarray}\label{dissw}
		&&\|\varphi\|_{\mathcal{C}([t_{1},t_{2}];L^{2}(\mathcal{U}))}	+\sqrt{\varepsilon}\|\varphi\|_{L^{2}(t_{1},t_{2};H^{1}(\mathcal{U}))} \leqslant C\exp\left(C(t_{2}-t_{1})C(\varepsilon,\mathfrak{B})\right)\left(\|F\|_{L^{2}(t_{1},t_{2};L^{2}(\mathcal{U}))}+\|G\|_{L^{2}(\mathcal{U})}\right),
	\end{eqnarray}
	where $C(\varepsilon,\mathfrak{B}):=\frac{\|\mathfrak{B}\|^{2}_{L^{\infty}(\mathcal{U}_{T})}}{\varepsilon}+\varepsilon+1$.
\end{proposition}
\begin{proof}
	To prove the existence and uniqueness of a weak solution of system \eqref{s6w}, we apply Lions' theorem, so it suffices to prove that the form $\mathfrak{a}_{w}$ defined in \eqref{fw} satisfies:
	\begin{itemize}
		\item $t\longmapsto \mathfrak{a}_{w}(t,u,v)$ is measurable for all $u,v\in H_{\Gamma_{0}}^{1}(\mathcal{U})$;
		\item $\mathfrak{a}_{w}$ is $H_{\Gamma_{0}}^{1}(\mathcal{U})$-bounded;
		\item $\mathfrak{a}_{w}$ is quasi-coercive; i.e., there exist $\alpha> 0$ and $\kappa\geq 0$ such that
		\begin{eqnarray}
			\mathfrak{a}_{w}(t,u,u) + \kappa \|u\|_{L^{2}(\mathcal{U})}\geq  \alpha\|u\|^{2}_{H^{1}(\mathcal{U})},\quad \mbox{for all}\;(u,t)\in H^{1}_{\Gamma_{0}}(\mathcal{U})\times [0,\tau]. \label{quasi-coercive}
		\end{eqnarray}
	\end{itemize}
	Using the boundedness of $\mathfrak{B}$, we obtain $(x,t)\longmapsto \varepsilon\nabla u(x)\cdot\nabla v(x)-u\hat{\mathfrak{B}}(x,t)\cdot\nabla v (x)$ is integrable on $\mathcal{U}\times [0,T]$ for all $u,v\in H^{1}(\mathcal{U})$, then, in particular from Fubini's theorem, we obtain 
	$t\longmapsto \mathfrak{a}_{w}(t,u,v)$ is measurable for all $u,v\in H_{\Gamma_{0}}^{1}(\mathcal{U})$.
	On the other hand 
	\begin{eqnarray}
		|\mathfrak{a}_{w}(t,u,v)| &\leqslant & \varepsilon \|\nabla u\|_{L^{2}(\mathcal{U})}\|\nabla v\|_{L^{2}(\mathcal{U})}+\|\mathfrak{B}\|_{L^\infty(\mathcal{U}_{T})}\| u\|_{L^{2}(\mathcal{U})}\|\nabla v\|_{L^{2}(\mathcal{U})} \nonumber\\
		&\leqslant & \left(\varepsilon+\|\mathfrak{B}\|_{L^\infty(\mathcal{U}_{T})}\right)\| u\|_{H^{1}(\mathcal{U})}\| v\|_{H^{1}(\mathcal{U})}. \nonumber
	\end{eqnarray}
	Hence, the form $\mathfrak{a}_{w}$ is $H_{\Gamma_{0}}^{1}(\mathcal{U})$-bounded. We claim that $\mathfrak{a}_{w}$ is quasi-coercive. By Hölder's inequality, we get
	\begin{eqnarray*}
		\left|\int_{\mathcal{U}}u\hat{\mathfrak{B}}(x,t)\cdot\nabla u\mathrm{~d}x\right| &\leqslant & \|\mathfrak{B}\|_{L^\infty(\mathcal{U}_{T})}\| u\|_{L^{2}(\mathcal{U})}\|\nabla u\|_{L^{2}(\mathcal{U})}\\
		&\leqslant & \frac{\varepsilon}{2}\|\nabla u\|^{2}_{L^{2}(\mathcal{U})}+\frac{\|\mathfrak{B}\|^{2}_{L^\infty(\mathcal{U}_{T})}}{2\varepsilon}\| u\|^{2}_{L^{2}(\mathcal{U})}.
	\end{eqnarray*}
	Then 
	\begin{eqnarray*}
		\mathfrak{a}_{w}(t,u,u) &\geq &  \frac{\varepsilon}{2}\| u\|^{2}_{H^{1}(\mathcal{U})}-\left(\frac{\varepsilon}{2}
		+\frac{\|\mathfrak{B}\|^{2}_{L^\infty(\mathcal{U}_{T})}}{2\varepsilon}\right)\| u\|^{2}_{L^{2}(\mathcal{U})}.
	\end{eqnarray*}
	Lions' theorem and \cite[Proposition III.2.1]{showalter2013monotone} yield the result that \eqref{s6w} has a unique weak solution. Consequently system \eqref{s3} also admits a unique weak solution. 
	\par 
	Let $\varphi$ the unique weak solution of \eqref{s3}. 
	From \cite[Proposition III.1.2]{showalter2013monotone}, we have $\|\varphi(\cdot)\|^{2}_{L^{2}(\mathcal{U})}$ is absolutely continuous on $[0,T]$ and the following standard energy identity is satisfied. 
	\begin{eqnarray}
		\frac{1}{2}\frac{\d}{\d t}\|\varphi\|^{2}_{L^{2}(\mathcal{U})}=\langle \partial_{t}\varphi,\varphi \rangle_{H_{\Gamma_{0}}^{1}(\mathcal{U})^{\prime},H_{\Gamma_{0}}^{1}(\mathcal{U})}\;\;a.e\; t\in [0,T]. \label{energy identity}
	\end{eqnarray}
	Using \eqref{energy identity}
	and integrating by parts, we obtain
	\begin{eqnarray*}
		-\frac{1}{2}\frac{d}{\mathrm{~d}t}\int_{\mathcal{U}}|\varphi|^{2}\mathrm{~d}x	+\varepsilon\int_{\mathcal{U}}|\nabla\varphi|^{2}\mathrm{~d}x &=&- \int_{\mathcal{U}}\varphi \mathfrak{B}(x,t)\cdot\nabla\varphi\mathrm{~d}x-\int_{\mathcal{U}}F\varphi\mathrm{~d}x\\
        &\leqslant&  \|\mathfrak{B}\|_{L^{\infty}(\mathcal{U}_{T})}\|\varphi\|_{L^{2}(\mathcal{U})}\|\nabla\varphi\|_{L^{2}(\mathcal{U})}-\int_{\mathcal{U}}F\varphi\mathrm{~d}x.
	\end{eqnarray*}
	By Young's inequality, we get
	\begin{eqnarray*}
		-\frac{1}{2}\frac{d}{\mathrm{~d}t}\int_{\mathcal{U}}|\varphi|^{2}\mathrm{~d}x	+\varepsilon\int_{\mathcal{U}}|\nabla\varphi|^{2}\mathrm{~d}x\leqslant\frac{\|\mathfrak{B}\|^{2}_{L^{\infty}(\mathcal{U}_{T})}\|\varphi\|^{2}_{L^{2}(\mathcal{U})}}{2\varepsilon}+\frac{\varepsilon}{2}\|\nabla\varphi\|^{2}_{L^{2}(\mathcal{U})} -\int_{\mathcal{U}}F\varphi\mathrm{~d}x.
	\end{eqnarray*}
	Adding $\frac{\varepsilon}{2}\|\varphi\|^{2}_{L^{2}(\mathcal{U})}$ in both side and by Young's inequality, we deduce that 
	\begin{eqnarray*}
		&&-\frac{1}{2}\frac{d}{\mathrm{~d}t}\int_{\mathcal{U}}|\varphi|^{2}\mathrm{~d}x	+\frac{\varepsilon}{2}\|\varphi\|^{2}_{H^{1}(\mathcal{U})}\leqslant \frac{C(\varepsilon,\mathfrak{B})}{2}\|\varphi\|^{2}_{L^{2}(\mathcal{U})}+\frac{1}{2}\int_{\mathcal{U}}|F|^{2}\mathrm{~d}x, 
	\end{eqnarray*}
	where $C(\varepsilon,\mathfrak{B}):=\frac{\|\mathfrak{B}\|^{2}_{L^{\infty}(\mathcal{U}_{T})}}{\varepsilon}+\varepsilon+1$.
	Integrating this inequality in $[t,t_{2}]$, we obtain
	\begin{eqnarray*}
		\int_{\mathcal{U}}|\varphi(x,t)|^{2}\mathrm{~d}x	+\varepsilon\|\varphi\|^{2}_{L^{2}(t,t_{2};H^{1}(\mathcal{U}))}&\leqslant& C(\varepsilon,\mathfrak{B})\int_{t}^{t_{2}}\int_{\mathcal{U}}|\varphi(x,s )|^{2}\mathrm{~d}x\mathrm{~d}s +\left(\|F\|^{2}_{L^{2}(t_{1},t_{2};L^{2}(\mathcal{U}))}+\|G\|^{2}_{L^{2}(\mathcal{U})}\right).
	\end{eqnarray*}
	The Grönwall's lemma gives the desired result. 
\end{proof}
The following result gives an important estimate of the solutions of system $\mathcal{S}(\Omega_{0},t_{1},t_{2},F,0,\varepsilon,\mathfrak{B})$ for a particular source term $F$.
\begin{proposition}\label{P0}
   Let $f_{0}\in L^{2}(0,T;L^{2}(\mathcal{U}))$,  $f_{1},\cdots,f_{d}\in L^{2}(0,T;H^{1}_{0}(\mathcal{U}))$, $F=f_{0}+\varepsilon\displaystyle\sum_{i=1}^{d}\partial_{x_{i}}f_{i}$, $\varepsilon\in (0,1)$ and assume that $\mathfrak{B}\in W^{1,\infty}(\mathcal{U}_{T})^{d}$ such that $\mathfrak{B}(x,t)\cdot\mathbf{n}(x)\geq 0$ on $\Gamma_{T}$. There exists $C>0$ depending on $\mathfrak{B}$, $T$, $d$ and $\Omega$ such that, for any $0\leqslant t_1<t_2 \leqslant T$ and any $\varepsilon\in (0,1)$, the weak solution $\varphi$ of $\mathcal{S} (\Omega_{0},t_{1},t_{2}, F, 0,\varepsilon,\mathfrak{B})$ satisfies
	\begin{eqnarray*}
		\|\varphi(\cdot,t_{1})\|^{2}_{L^{2}(\mathcal{U})}\leqslant C\sum_{i=0}^{d}\|f_{i}\|^{2}_{L^{2}(t_{1},t_{2};L^{2}(\mathcal{U}))}.
	\end{eqnarray*}
\end{proposition}
\begin{proof}
	Using the energy identity \eqref{energy identity}, $\varepsilon\partial_{\n}\varphi(x,t)+\varphi\mathfrak{B}(x,t)\cdot\mathbf{n}(x)=0$ on $\Gamma_{T}$, $\varphi(\cdot,t)=0$ on $\Gamma_{0}$ and integration by parts, we obtain 
		\begin{eqnarray}
			&&\frac{1}{2}\frac{d}{dt}\|\varphi(\cdot,t)\|^{2}_{L^{2}(\mathcal{U})}=\varepsilon\int_{\mathcal{U}}|\nabla\varphi(x,t)|^{2}\mathrm{~d}x-\frac{1}{2}\int_{\mathcal{U}}\nabla\cdot\mathfrak{B}(x,t)|\varphi(x,t)|^{2}\mathrm{~d}x \nonumber\\
			&& + \frac{1}{2}\int_{\Gamma}|\varphi(x,t)|^{2}\mathfrak{B}(x,t)\cdot\n(x)\d x+ \int_{\mathcal{U}}f_{0}(x,t)\varphi(x,t)\mathrm{~d}x +\varepsilon\sum_{i=1}^{d}\int_{\mathcal{U}}\partial_{x_{i}}f_{i}(x,t)\varphi(x,t)\mathrm{~d}x. \label{i1}
		\end{eqnarray}
	On the other hand, since $f_{i}(\cdot,t)\in H^{1}_{0}(\mathcal{U})$ for $i=1,\cdots,d$, by integration by parts, we have
	\begin{eqnarray*}
		\int_{\mathcal{U}}\partial_{x_{i}}f_{i}(x,t)\varphi(x,t)\mathrm{~d}x=-\int_{\mathcal{U}}f_{i}(x,t)\partial_{x_{i}}\varphi(x,t)\mathrm{~d}x.
	\end{eqnarray*}
	Thus, by Cauchy-Schwarz and Hölder inequalities, we get 
	\begin{eqnarray}
		\varepsilon\left|\int_{\mathcal{U}}\partial_{x_{i}}f_{i}(x,t)\varphi(x,t)\mathrm{~d}x\right|
		&\leqslant & \frac{\varepsilon}{2d}\|\nabla\varphi (\cdot,t)\|^{2}_{L^{2}(\mathcal{U})}+\frac{d}{2}\|f_{i} (\cdot,t)\|^{2}_{L^{2}(\mathcal{U})} \label{i2}
	\end{eqnarray}
	and 
	\begin{eqnarray}
		\left|\int_{\mathcal{U}}f_{0}(x,t)\varphi(x,t)\mathrm{~d}x\right| 
		&\leqslant & \frac{d}{2}\|f_{0} (\cdot,t)\|^{2}_{L^{2}(\mathcal{U})}+\frac{1}{2}\|\varphi (\cdot,t)\|^{2}_{L^{2}(\mathcal{U})}. \label{i3}
	\end{eqnarray}
	Using \eqref{i1}-\eqref{i3} and $\mathfrak{B}(x,t)\cdot\n(x)\geq 0$ on $\Gamma_{T}$, we obtain 
	\begin{eqnarray}
		\frac{1}{2}\frac{d}{dt}\|\varphi(\cdot,t)\|^{2}_{L^{2}(\mathcal{U})} &\geq& \frac{\varepsilon}{2}\int_{\mathcal{U}}|\nabla\varphi(x,t)|^{2}\mathrm{~d}x-\frac{d}{2}\sum_{i=0}^{d}\|f_{i}(\cdot,t)\|^{2}_{L^{2}(\mathcal{U})}  -C\|\varphi(\cdot,t)\|^{2}_{L^{2}(\mathcal{U})}\nonumber\\
		&\geq&-\frac{d}{2}\sum_{i=0}^{d}\|f_{i}(\cdot,t)\|^{2}_{L^{2}(\mathcal{U})} -C\|\varphi(\cdot,t)\|^{2}_{L^{2}(\mathcal{U})}, \label{i4}
	\end{eqnarray}
	where $C:=\|\nabla\cdot\mathfrak{B}\|_{L^{\infty}(\mathcal{U}_{T})}+1$. Integrating \eqref{i4} in $(t,t_{2})$ for $t_{1}\leqslant t < t_{2}$, we have
	\begin{eqnarray}
		\|\varphi(\cdot,t)\|^{2}_{L^{2}(\mathcal{U})}&\leqslant&  d\sum_{i=0}^{d}\int_{t}^{t_{2}}\|f_{i}(\cdot,s)\|^{2}_{L^{2}(\mathcal{U})}\d s +2C\int_{t}^{t_{2}}\|\varphi(\cdot,s)\|^{2}_{L^{2}(\mathcal{U})}\mathrm{~d}s. \nonumber  
	\end{eqnarray}
	Applying Grönwall's lemma, we obtain the desired result. 
\end{proof}
\subsection{Strong solutions of system $\mathcal{S}(\Omega_{0},t_{1},t_{2}, F, G,\varepsilon,\mathfrak{B})$}
The existence and uniqueness of strong solutions for system \eqref{s3} is derived mainly from the reference \cite{arendt2014maximal}. In this section, we will assume that $\mathfrak{B}\in W^{1,\infty}(\mathcal{U}_{T})^{d}$, allowing us to write \eqref{s3} and \eqref{s6w} respectively as follows:
\begin{equation} \label{s5}
	\left\{
	\begin{aligned}
		\partial_t \varphi+\varepsilon\Delta \varphi + \mathfrak{B}(x,t)\cdot\nabla\varphi+\left(\nabla\cdot\mathfrak{B}(x,t)\right)\varphi &=F(x,t)\quad &&\text { in } \mathcal{U}\times (t_{1},t_{2}), \\
		\left(\varepsilon \nabla\varphi +\varphi \mathfrak{B} (x,t)\right)\cdot \mathbf{n}(x)\mathds{1}_{\Gamma}(x)+\varphi\mathds{1}_{\Gamma_{0}}(x) &=0 \quad && \text { on } \partial\mathcal{U}\times (t_{1},t_{2}), \\
		\varphi(x,t_{2}) &=G(x) \quad &&\text { in } \mathcal{U}
	\end{aligned}
	\right.
\end{equation}
and
\begin{equation} \label{s6}
	\left\{
	\begin{aligned}
		\partial_t \hat{\varphi}-\varepsilon\Delta \hat{\varphi} + \hat{\mathfrak{B}}(x,t)\cdot\nabla\hat{\varphi}+(\nabla\cdot\hat{\mathfrak{B}}(x,t))\hat{\varphi} &=\hat{F}(x,t)\quad &&\text { in } \mathcal{U}\times (0,\tau), \\
		(\varepsilon \nabla\hat{\varphi} -\hat{\varphi} \hat{\mathfrak{B}}(x,t))\cdot \mathbf{n}(x)\mathds{1}_{\Gamma}(x)+\hat{\varphi}\mathds{1}_{\Gamma_{0}}(x) &=0 \quad && \text { on } \partial\mathcal{U}\times (0,\tau), \\
		\hat{\varphi}(x,0) &=G(x) \quad &&\text { in } \mathcal{U}.
	\end{aligned}
	\right.
\end{equation}
We consider the bilinear form defined on $[0,\tau]\times H_{\Gamma_{0}}^{1}(\mathcal{U})\times H_{\Gamma_{0}}^{1}(\mathcal{U})$ by
\begin{eqnarray*}
	\mathfrak{a}(t,u,v)&:=&\varepsilon\int_{\mathcal{U}}\nabla u\cdot\nabla v\mathrm{~d}x-\int_{\Gamma}(\hat{\mathfrak{B}}(x,t)\cdot\n(x))u\; v\mathrm{~d}\sigma +
	\int_{\mathcal{U}}(\hat{\mathfrak{B}}(x,t)\cdot\nabla u)v\mathrm{~d}x
	+\int_{\mathcal{U}}(\nabla\cdot\hat{\mathfrak{B}}(x,t))u\,v\mathrm{~d}x
\end{eqnarray*}
and the following maximal regularity space
$$MR_{\mathfrak{a}}(t_{1},t_{2}):=\left\{u\in H^{1}(t_{1},t_{2};L^{2}(\mathcal{U}))\cap L^{2}(t_{1},t_{2};H_{\Gamma_{0}}^{1}(\mathcal{U})):\;\; \mathcal{A}(\cdot)u(\cdot)\in L^{2}(t_{1},t_{2};L^{2}(\mathcal{U}))\right\},$$
where $\mathcal{A}(t)\in \mathcal{L}\left(H_{\Gamma_{0}}^{1}(\mathcal{U}), H_{\Gamma_{0}}^{1}(\mathcal{U})^{\prime}\right)$ is the operator associated with $\mathfrak{a}(t,\cdot,\cdot)$ and defined by 
$$\langle \mathcal{A}(t)u,v\rangle_{H_{\Gamma_{0}}^{1}(\mathcal{U})^{\prime},H_{\Gamma_{0}}^{1}(\mathcal{U})} :=\mathfrak{a}(t,u,v).$$
It is a Hilbert space for the norm $\|\cdot\|_{MR_{\mathfrak{a}}(t_{1},t_{2})}$ defined by
$$\|u\|_{MR_{\mathfrak{a}}(t_{1},t_{2})}:=\left(\|u\|^{2}_{L^{2}(t_{1},t_{2};H^{1}(\mathcal{U}))}+\|\partial_{t}u \|^{2}_{L^{2}(t_{1},t_{2};L^{2}(\mathcal{U}))}+\|\mathcal{A}(\cdot)u(\cdot)\|^{2}_{L^{2}(t_{1},t_{2};L^{2}(\mathcal{U}))}\right)^{1/2}.$$
We have the following important result:
\begin{proposition} \label{embedded}
	The space $MR_{\mathfrak{a}}(t_{1},t_{2})$ embeds continuously into $\mathcal{C}\left([t_{1},t_{2}]; H_{\Gamma_{0}}^{1}(\mathcal{U})\right)$.
\end{proposition}
\begin{proof}
	For more details, we refer to
	\cite[Corollary 3.3]{arendt2014maximal}. 
\end{proof}
For all $t\in [0,\tau]$, we define the operators $A_{1}(t)$ and $A_{2}(t)$ by
\begin{eqnarray*}
	&&D(A_{1}(t)):=\left\{u\in H_{\Gamma_{0}}^{1}(\mathcal{U}):\;\Delta u\in L^{2}(\mathcal{U}),\; \varepsilon\partial_{\n}u|_{\Gamma} -\hat{\mathfrak{B}}(x,t)\cdot\n(x) u|_{\Gamma}=0 \right\},\\
	&&D(A_{2}(t)):= H_{\Gamma_{0}}^{1}(\mathcal{U})
\end{eqnarray*} 
and for all $(u,v)\in D(A_{1}(t))\times D(A_{2}(t))$ 
$$A_{1}(t)u:=-\varepsilon\Delta u\quad\mbox{and}\quad A_{2}(t) v:=\hat{\mathfrak{B}}(x,t)\cdot\nabla v+(\nabla\cdot\hat{\mathfrak{B}}(x,t))v.$$
System \eqref{s6} can be written equivalently as a Cauchy initial valued problem
\begin{equation}\label{cau1}
	\left\{
	\begin{aligned}
		Y^{\prime}+A(t)Y &=\hat{F}(\cdot, t) & & t\in [0,\tau], \\
		Y(0)&=G, 
	\end{aligned}
	\right.
\end{equation}
where $A(t)=A_{1}(t)+A_{2}(t)$, $D(A(t))=D(A_{1}(t))$ and $Y(t)=\hat{\varphi}(\cdot,t)$. We start with the definition of a strong solution of \eqref{s5}.
\begin{definition} 
	Let $F\in L^{2}(t_{1},t_{2};L^{2}(\mathcal{U}))$ and $G\in L^{2}(\mathcal{U})$. A strong solution of \eqref{s3} is a function $\varphi\in MR_{\mathfrak{a}}(t_{1},t_{2})$ fulfilling $\eqref{s5}_{1}$ in $L^{2}(t_{1},t_{2};L^{2}(\mathcal{U}))$, $\eqref{s5}_{2}$ in $L^{2}(t_{1},t_{2};L^{2}(\partial\mathcal{U}))$ and $\eqref{s5}_{3}$, where $\eqref{s5}_{j}$ is the j-th equation in system \eqref{s5}.
\end{definition}
Now we are in position to establish the following existence, uniqueness and regularity results.
\begin{proposition} \label{ps}
	Let $\mathfrak{B}\in W^{1,\infty}(\mathcal{U}_{T})^{d}$, $F\in L^{2}(t_{1},t_{2};L^{2}(\mathcal{U}))$ and $G\in H_{\Gamma_{0}}^{1}(\mathcal{U})$. Then the Cauchy problem \eqref{cau1}, and hence system \eqref{s3} has a unique strong 
	solution $\varphi\in MR_{\mathfrak{a}}(t_{1},t_{2})$. Moreover, $\varphi\in \mathcal{C}([t_{1},t_{2}];L^{2}(\mathcal{U}))$ and there exists a constant $C:=C(T,\varepsilon)> 0$ such that
	\begin{equation}
		\|\varphi\|_{MR_{\mathfrak{a}}(t_{1},t_{2})}\leqslant C\left(\|F\|_{L^{2}(t_{1},t_{2};L^{2}(\mathcal{U}))}+\|G\|_{H^{1}(\mathcal{U})}\right). \label{ee}
	\end{equation}
\end{proposition}
\begin{proof} 
	To prove the existence and uniqueness of a strong solution of \eqref{s1}, we apply Theorem 4.2 and Remark 4.6 of \cite{arendt2014maximal}, so 
	we consider the bilinear forms defined on $[0,\tau]\times H_{\Gamma_{0}}^{1}(\mathcal{U})\times H_{\Gamma_{0}}^{1}(\mathcal{U})$ by 
	\begin{eqnarray*}
		&&\mathfrak{a}_{1}(t,u,v):=\varepsilon\int_{\mathcal{U}}\nabla u\cdot\nabla v\mathrm{~d}x-\int_{\Gamma}(
		\hat{\mathfrak{B}}(x,t)\cdot\n(x))u\;v\mathrm{~d}\sigma, \\
		&&\mathfrak{a}_{2}(t,u,v):= \int_{\mathcal{U}}(\hat{\mathfrak{B}}(x,t)\cdot\nabla u)v\mathrm{~d}x+\int_{\mathcal{U}}(\nabla\cdot\hat{\mathfrak{B}}(x,t))u\;v\mathrm{~d}x.
	\end{eqnarray*}
	Clearly, we have 
	\begin{eqnarray*}
		\mathfrak{a}_{1}(t,u,v)+\mathfrak{a}_{2}(t,u,v)=\mathfrak{a}(t,u,v)
	\end{eqnarray*}
	and we claim that, $\mathfrak{a}_{1}$ and $\mathfrak{a}_{2}$ satisfies the conditions:
	\begin{itemize}
		\item $|\mathfrak{a}_{1}(t,u,v)|\leq M_{1}\|u\|_{H^{1}(\mathcal{U})}\|v\|_{H^{1}(\mathcal{U})},\;\mbox{for all}\; u,v \in H^{1}_{\Gamma_{0}}(\mathcal{U})$ and all $t\in [0,\tau]$;   
		\item $\mathfrak{a}_{1}$ is quasi-coercive, see \eqref{quasi-coercive};
		\item  $\mathfrak{a}_{1}$ satisfies the square root property; i.e., $R\left(A_{1}(t)^{-1/2}\right)=H_{\Gamma_{0}}^{1}(\mathcal{U})$;
		\item $\mathfrak{a}_{1}$ is Lipschitz-continuous; i.e., there exists a constant $C_{1}\geq 0$ such that, for all $u,v\in H^{1}_{\Gamma_{0}}(\mathcal{U})$ and all $s,t\in [0,\tau]$,
		\begin{eqnarray*}
			|\mathfrak{a}_{1}(t,u,v)-\mathfrak{a}_{1}(s,u,v)|\leq C_{1}|t-s|\|u\|_{H^{1}(\mathcal{U})}\|v\|_{H^{1}(\mathcal{U})};
		\end{eqnarray*}
		\item $|\mathfrak{a}_{2}(t,u,v)|\leq M_{2}\|u\|_{H^{1}(\mathcal{U})}\|v\|_{L^{2}(\mathcal{U})}\;\mbox{for all}\; (u,v)\in H^{1}_{\Gamma_{0}}(\mathcal{U})\times L^{2}(\mathcal{U})$ and all $t\in [0,\tau]$;
		\item $t\longmapsto \mathfrak{a}_{2}(t,u,v)$ is measurable for all $u, v\in H^{1}_{\Gamma_{0}}(\mathcal{U})$.
	\end{itemize}
	By the boundedness of $\mathfrak{B}$ and the continuity of the trace operator, the form $\mathfrak{a}_{1}$ is $H^{1}(\mathcal{U})$-bounded. Since $\mathfrak{a}_{1}$ is symmetric, then it satisfies the square root property, see \cite{kato2013perturbation}. Using $u=0$ on $\Gamma_{0}$, the trace estimate \eqref{Trace estimate} and Young's inequality, we obtain
	\begin{eqnarray*}
		\left|\int_{\Gamma}(\hat{\mathfrak{B}}(x,t)\cdot\n (x))|u|^{2}\mathrm{~d}\sigma\right|\leq \frac{\varepsilon}{2}\|\nabla u\|^{2}_{L^{2}(\mathcal{U})}+ \frac{C}{\varepsilon}\| u\|^{2}_{L^{2}(\mathcal{U})}.
	\end{eqnarray*}
	Hence $\mathfrak{a}_{1}$ is quasi-coercive. By the Lipschitz continuous of $\mathfrak{B}$, the form $\mathfrak{a}_{1}$ is also Lipschitz continuous. 
	The boundedness of $\mathfrak{B}$ implies the form $\mathfrak{a}_{2}: H_{\Gamma_{0}}^{1}(\mathcal{U})\times L^{2}(\mathcal{U})\longrightarrow\mathbb{R}$ is bounded for all fixed $t\in [0,\tau]$. We also have that $t\longmapsto \mathfrak{a}_{2}(t,u,v)$ is measurable for all $u,v\in H_{\Gamma_{0}}^{1}(\mathcal{U})$ as for the form $\mathfrak{a}_{w}$ above.\\
	Consequently, \cite[Theorem 4.2 and Remark 4.6]{arendt2014maximal} implies that the Cauchy problem
	\begin{equation}\label{ccal1}
		\left\{
		\begin{aligned}
			Y^{\prime}+\mathcal{A}(t)Y &=\hat{F}(\cdot,t) & & t\in [0,\tau], \\
			Y(0)&=G, 
		\end{aligned}
		\right.
	\end{equation}
	has a unique strong solution $Y\in MR_{\mathfrak{a}}(0,\tau)$. Furthermore 
	\begin{equation}
		\|Y\|_{MR_{\mathfrak{a}}(0,\tau)}\leqslant C\left(\|\hat{F}\|_{L^{2}(0,\tau;L^{2}(\mathcal{U}))}+\|G\|_{H^{1}(\mathcal{U})}\right). \label{eec}
	\end{equation}
	Let us then show that \eqref{cau1} has a unique strong solution $Y\in MR_{\mathfrak{a}}(0,\tau)$.
	That is, we will show if $Y\in MR_{\mathfrak{a}}(0,\tau)$ the strong solution of \eqref{ccal1}, then $Y(t)\in D(A(t))\;\text{and}\; A(t)(Y(t))=\mathcal{A}(t)(Y(t))$ for all $t\in [0,\tau]$.\\
	For all $v\in H_{\Gamma_{0}}^{1}(\mathcal{U})$, the strong solution of \eqref{ccal1} satisfies
	\begin{eqnarray*}
		\int_{\mathcal{U}}Y^{\prime}(t)v\mathrm{~d}x+\int_{\mathcal{U}}\mathcal{A}(t)Y(t)v\mathrm{~d}x=\int_{\mathcal{U}}\hat{F}v\mathrm{~d}x.
	\end{eqnarray*}
	Then,  
	\begin{eqnarray}
		&&\int_{\mathcal{U}}Y^{\prime}(t)v\mathrm{~d}x+\varepsilon\int_{\mathcal{U}}\nabla Y(t)\cdot\nabla v \mathrm{~d}x-\int_{\Gamma}\hat{\mathfrak{B}}(x,t)\cdot \n(x) Y(t)v\mathrm{~d}\sigma \label{cc4}\\ 
		&& +\int_{\mathcal{U}}(\hat{\mathfrak{B}}(x,t)\cdot\nabla Y(t))v\mathrm{~d}x +\int_{\mathcal{U}}\nabla\cdot \hat{\mathfrak{B}}(x,t) Y(t)v\mathrm{~d}x=\int_{\mathcal{U}}\hat{F}v\mathrm{~d}x. \nonumber
	\end{eqnarray}
	In particular for all $v\in \mathcal{D}(\mathcal{U})$, we obtain
	\begin{eqnarray*}
		\int_{\mathcal{U}}  \left[Y^{\prime}(t)+\hat{\mathfrak{B}}(x,t)\cdot\nabla Y(t)+\nabla\cdot \hat{\mathfrak{B}}(x,t) Y(t)-\hat{F}\right]v\mathrm{~d}x
		= -\varepsilon\int_{\mathcal{U}}\nabla Y(t)\cdot\nabla v \mathrm{~d}x,
	\end{eqnarray*}
	for all $v\in D(\mathcal{U})$, therefore $\Delta Y(t)\in L^{2}(\mathcal{U})$ and 
	\begin{eqnarray}\label{cc5}
		Y^{\prime}(t)=\varepsilon\Delta Y(t)-\hat{\mathfrak{B}}(x,t)\cdot\nabla Y(t)-\nabla\cdot \hat{\mathfrak{B}}(x,t) Y(t)+\hat{F}.
	\end{eqnarray}
	Substituting \eqref{cc5} into \eqref{cc4} gives
	\begin{eqnarray*}
		\forall v\in H_{\Gamma_{0}}^{1}(\mathcal{U}),\; \varepsilon\int_{\mathcal{U}}\Delta Y(t)v\mathrm{~d}x+\varepsilon\int_{\mathcal{U}}\nabla Y(t)\cdot\nabla v \mathrm{~d}x=\int_{\Gamma}\hat{\mathfrak{B}}(x,t)\cdot\n(x)Y(t) v\mathrm{~d}\sigma.
	\end{eqnarray*}
	Then $\partial_{\n}Y(t)|_{\Gamma}\in L^{2}(\Gamma)$ and $\varepsilon\partial_{\n}Y(t)|_{\Gamma}=\hat{\mathfrak{B}}(x,t)\cdot\n(x)Y(t)|_{\Gamma}.$ Consequently $Y(t)\in D(A(t))$. By a simple integration by parts, we have
	\begin{eqnarray*}
		\left( A(t)(Y(t)),v\right)=\mathfrak{a}(t,Y(t),v),
	\end{eqnarray*}
	for all  $v\in H_{\Gamma_{0}}^{1}(\mathcal{U})$, then $A(t)(Y(t))=\mathcal{A}(t)(Y(t))$. Finally, the Cauchy problem \eqref{cau1}, and hence system \eqref{s3} has a unique strong solution $\varphi\in MR_{\mathfrak{a}}(t_{1},t_{2})$.
    From Proposition \ref{embedded} and \eqref{eec}, we obtain $\varphi\in\mathcal{C}([t_{1},t_{2}];H_{\Gamma_{0}}^{1}(\mathcal{U}))$ and \eqref{ee}. 
\end{proof}
\begin{remark}
	Propositions \ref{pw} and \ref{ps} are valid if $\Omega_{0}=\varnothing$.
\end{remark}
\section{Agmon inequalities, Dissipation results and Carleman estimate} \label{Section 4}
In this section, we present important estimates that are key to proving essential results.
\subsection{Agmon inequalities}
In this subsection, we wil present some technical results and we will prove some new Agmon inequalities which will be the key to establish very interesting dissipativity estimates.
\par
Let us start with the following notation, which will be useful in what follows : \\
\textbf{Notation.}
Let $0\leqslant t_{1}\leqslant t_{2}\leqslant T$ and $x_{0}\in\mathbb{R}^{d}$. For $r>0$, we note $\mathcal{D}_{r}(t_{1},t_{2},x_{0})$ the union of
trajectories starting at $t_{2}$ in the ball $\overline{B}(x_{0},r)$:
$$\mathcal{D}_{r}(x_{0},t_{1},t_{2})=\left\lbrace (\Phi(t,t_{2},y),t)\;:\; y\in\overline{B}(x_{0},r)\;\;\mbox{and}\;\; t\in [t_{1},t_{2}]\right\rbrace,$$ 
where $t\longmapsto \Phi(t,\cdot,\cdot)$ are the trajectories of ordinary equation \eqref{OD}.
\par
The following Lemma asserts the existence of a Lipschitz function that verifies certain conditions associated with the trajectories of vector $\mathfrak{B}$, the construction of this function is based on the change of coordinates by the trajectories of the vector field $\mathfrak{B}$ and the use of radial functions. For more details, see \cite[Section 2.2]{guerrero2007singular}.

\begin{lemma} \label{lemma}
	Let $\mathfrak{B}\in L^{\infty}(0,T;W^{1,\infty}(\mathbb{R}^{d})^{d})$, then for all $0\leqslant t_{1}\leqslant t_{2}\leqslant T$, $x_{0}\in\mathbb{R}^{d}$ and all $r>0$
	there exists a nonnegative Lipschitz function $\theta$ on $\mathbb{R}^{d}\times [t_{1},t_{2}]$ such that 
	\begin{eqnarray*}
		\partial_{t}\theta-|\nabla \theta|^2 +\mathfrak{B}(x,t)\cdot \nabla \theta &\geq 0& \quad\quad \text {a.e in } \mathbb{R}^{d} \times[t_{1},t_{2}],\\
		\theta(x,t)&=&0\quad\quad \forall (x,t)\in \mathcal{D}_{r}(x_{0},t_{1},t_{2}),\\
		\theta(x,t)&\geq& c_{0}r^{2}\quad\forall (x,t)\notin \mathcal{D}_{2r}(x_{0},t_{1},t_{2}),
	\end{eqnarray*}
	where $c_{0}>0$ depends only on $t_{2}-t_{1}$ and $\displaystyle\int_{t_{1}}^{t_{2}}\|\nabla \mathfrak{B}(\cdot, s)\|_{\infty}\mathrm{~d}s$.\\
\end{lemma}
Now we are ready to present and prove some Agmon inequalities. 
\begin{proposition} \label{propagmon}
	Let $\Omega$ be a domain with Lipschitz boundary $\Gamma$, $\mathfrak{B}\in L^{\infty}(0,T;W^{1,\infty}(\Omega)^{d})$ and
let $\theta$ be a Lipschitz function on $\overline{\Omega} \times[0, T]$ such that
	\begin{equation*}
		\partial_{t}\theta -|\nabla \theta|^2 +\mathfrak{B}(x,t)\cdot \nabla \theta \geq 0,\;\;  \text { a.e in } \overline{\Omega} \times[0, T].
	\end{equation*}
	Then, we have the following estimates:
	\begin{enumerate}[label=(\arabic*), ref=(\arabic*)]
		\item There exists a constant $C>0$ (independent of $\varepsilon$) such that, for all $\varepsilon\in (0,1)$ and
		any solution $\varphi$ of system  $\mathcal{S}(\Omega_{0},t_{1},t_{2},0, G,\varepsilon,\mathfrak{B})$ with data $G\in L^{2}(\mathcal{U})$, the following Agmon-type inequality holds true for all $t \in[t_{1},t_{2}]$,
		\begin{eqnarray}
			&&\exp\left(-\frac{C}{\varepsilon}(t_{2}-t)\right)\int_{\mathcal{U}}|\psi(x, t)|^2 \mathrm{~d} x +\varepsilon \int_t^{t_{2}} \int_{\mathcal{U}} \exp\left(-\frac{C}{\varepsilon}(t_{2}-s)\right)|\nabla \psi(x, s)|^2 \mathrm{~d} x \mathrm{~d} s \leq \int_{\mathcal{U}}|\psi(x, t_{2})|^2 \mathrm{~d}x, \label{A1}
		\end{eqnarray}
		where $ \psi=\exp\left(\frac{\theta}{\varepsilon}\right)\varphi$ and $\mathcal{U}=\Omega\setminus\overline{\Omega_{0}}$.\\
		\item If moreover, $\mathfrak{B}(x,t)\cdot\mathbf{n}(x)\geq 0$ on $\Gamma_{T}$, then 
		\begin{eqnarray} 
			&&\exp\left(-C_{\mathfrak{B}}(t_{2}-t)\right)\int_{\mathcal{U}}|\psi(x, t)|^2 \mathrm{~d} x +2\varepsilon \int_t^{t_{2}} \int_{\mathcal{U}} \exp\left(-C_{\mathfrak{B}}(t_{2}-s)\right)|\nabla \psi(x, s)|^2 \mathrm{~d} x \mathrm{~d} s  \leq \int_{\mathcal{U}}|\psi(x, t_{2})|^2 \mathrm{~d}x, \label{A2}
		\end{eqnarray}
		where $C_{\mathfrak{B}}:=\|\nabla\cdot\mathfrak{B}\|_{L^{\infty}(\Omega_{T})}$.
	\end{enumerate}
\end{proposition}
\begin{proof}
	For all $t\in [t_{1},t_{2}]$, we consider the energy $E(t):=\dfrac{1}{2}\displaystyle\int_{\mathcal{U}}|\psi(x,t)|^{2}dx$. By several integrations by parts, one has 
	\begin{eqnarray}
		E^{\prime}(t)&=& \frac{1}{\varepsilon}\int_{\mathcal{U}}\left(\partial_{t}\theta-\lvert\nabla\theta\lvert^{2}+\mathfrak{B}(x,t)\cdot\nabla\theta\right)|\psi|^{2}\mathrm{~d}x+\varepsilon\int_{\mathcal{U}}\lvert \nabla \psi\lvert^{2}\mathrm{~d}x \nonumber \\
        &&-\frac{1}{2}\int_{\mathcal{U}}\nabla\cdot \mathfrak{B}(x,t)|\psi|^{2}\mathrm{~d}x  +\frac{1}{2}\int_{\Gamma}\mathfrak{B}(x,t)\cdot\n(x)|\psi|^{2} \mathrm{~d}\sigma. \label{en} 
	\end{eqnarray}
	\begin{enumerate}[label=(\arabic*), ref=(\arabic*)]
		\item By the hypothesis verified by the function $\theta$, we have  
		\begin{eqnarray}
			E^{\prime}(t)&\geq& \varepsilon\int_{\mathcal{U}}\lvert \nabla \psi\lvert^{2}\mathrm{~d}x-\|\nabla\cdot\mathfrak{B}\|_{L^{\infty}(\Omega_{T})}E(t)-\frac{\|\mathfrak{B}\|_{L^{\infty}(\Gamma_{T})}}{2}\int_{\Gamma}|\psi|^{2} \mathrm{~d}\sigma. \nonumber
		\end{eqnarray}
		By trace estimate \eqref{Trace estimate}
		and Young’s inequality, we obtain
		\begin{eqnarray}
			E^{\prime}(t)&\geq& \frac{\varepsilon}{2}\int_{\mathcal{U}}\lvert \nabla \psi\lvert^{2}\mathrm{~d}x-\left(\|\nabla\cdot \mathfrak{B}\|_{L^{\infty}(\Omega_{T})}+\varepsilon+\frac{C^{2}\|\mathfrak{B}\|^{2}_{L^{\infty}(\Gamma_{T})}}{4\varepsilon}\right)E(t). \nonumber
		\end{eqnarray}
		Thus, there exists a constant $C>0$ independent of $\varepsilon$ such that,
		\begin{eqnarray}
			E^{\prime}(t)&\geq& \frac{\varepsilon}{2}\int_{\mathcal{U}}\lvert \nabla \psi\lvert^{2}\mathrm{~d}x-\frac{C}{\varepsilon}E(t). \nonumber
		\end{eqnarray}	
		By applying Grönwall's lemma, we deduce inequality \eqref{A1}.
		\item If $\mathfrak{B}(x,t)\cdot\mathbf{n}(x)\geq 0$ on $\Gamma_{T}$, then \eqref{en} gives 
		\begin{eqnarray*}
			E^{\prime}(t)\geq \varepsilon\int_{\mathcal{U}}\lvert \nabla \psi\lvert^{2}\mathrm{~d}x-\frac{1}{2}\int_{\mathcal{U}}\nabla\cdot \mathfrak{B}(x,t)|\psi|^{2}\mathrm{~d}x.
		\end{eqnarray*}
		Hence
		\begin{eqnarray*}
			E^{\prime}(t)\geq \varepsilon\int_{\mathcal{U}}\lvert \nabla \psi\lvert^{2}\mathrm{~d}x-C_{\mathfrak{B}}E(t),
		\end{eqnarray*}
		where $C_{\mathfrak{B}}=\|\nabla\cdot\mathfrak{B}\|_{L^{\infty}(\Omega_{T})}$. The Grönwall lemma directly gives the inequality \eqref{A2}.
	\end{enumerate}
\end{proof}
Considering $\theta=0$ in the previous lemma, we obtain the following corollary:
\begin{corollary}
	Assume that $\mathfrak{B}\in L^{\infty}(0,T;W^{1,\infty}(\Omega)^{d})$ such that $\mathfrak{B}(x,t)\cdot\mathbf{n}(x)\geq 0$ on $\Gamma_{T}$.
	Then, any solution $\varphi$ of system  $\mathcal{S}(\Omega_{0},t_{1},t_{2},0,G,\varepsilon,\mathfrak{B})$ with data $G \in L^{2}(\mathcal{U})$  satisfies
	\begin{eqnarray} \label{A3}
		\int_{\mathcal{U}}|\varphi(x, t)|^2 \mathrm{~d} x\leq \exp\left(C_{\mathfrak{B}}(t_{2}-t)\right)\int_{\mathcal{U}}|\varphi(x, t_{2})|^2 \mathrm{~d}x,
	\end{eqnarray}
	where $C_{\mathfrak{B}}:=\|\nabla\cdot\mathfrak{B}\|_{L^{\infty}(\Omega_{T})}$.
\end{corollary}
\subsection{Dissipation results by Agmon inequality} \label{Dissipartion results}
In this subsection, we will assume that
\begin{enumerate}[label=(H\arabic*), ref=(H\arabic*)]
\item \label{A} $T>0$ and $\varepsilon\in (0,1)$,
\item \label{B} $\mathfrak{B}\in L^{\infty}(0,T;W^{1,\infty}(\Omega)^{d})$ such that $\mathfrak{B}(x,t)\cdot\mathbf{n}(x)\geq 0$ on $\Gamma_{T}$,
\item \label{C} $\exists T_{0}\in (0,T)$ and $r_{0}>0$ such that, $(T,T_{0},r_{0},\mathfrak{B},\Omega)$  satisfies condition \eqref{Flushing Condition} for the control region $\omega$
\end{enumerate}
and we will prove some very important dissipation results.
\par   
Applying Proposition \ref{P1} to Hypothesis \ref{C}, there exists $\omega_{0}\subset\subset\omega$ a regular open such that $(T,T_{0},\frac{r_{0}}{2},\mathfrak{B},\Omega)$ satisfies condition \eqref{Flushing Condition} for $\omega_{0}$. Hence,
\begin{eqnarray}
	&&\forall x_{0}\in\overline{\Omega},\; \forall t_{0}\in [T_{0},T],\; \exists t\in (t_{0}-T_{0},t_{0}),\; \forall x\in \overline{B}\left(x_{0},\frac{r_{0}}{2}\right)
	, \Phi(t,t_{0},x)\in \omega_{0}.\quad\quad \label{o4}
\end{eqnarray} 
In the following two subsections, $\mathcal{U}$ stands for the open:
\begin{equation*}
	\mathcal{U}:=\Omega\setminus\overline{\omega_{0}}.
\end{equation*} 
The assertion in \eqref{o4}, implies that 
\begin{eqnarray}
	&&\forall x_{0}\in\overline{\Omega},\; \forall t_{0}\in [T_{0},T],\; \exists t\in (t_{0}-T_{0},t_{0}),\; \forall x\in \overline{B}\left(x_{0},\frac{r_{0}}{2}\right),   \Phi(t,t_{0},x)\notin \overline{\mathcal{U}}. \quad\quad \label{o5}
\end{eqnarray}
In the next two parts of this subsection, we will show two dissipation results: the first applies outside the region 
$\omega_0$, while the second is global, with its proof relying on the first.
\subsubsection{Dissipation result outside the region $\omega_{0}$}
\begin{proposition} \label{P2}
	Under Hypotheses \ref{A}, \ref{B} and \ref{C}, there are constants $C_{0}>0$ dependent on $r_{0}, T_{0}$ and $\|\mathfrak{B}\|_{L^{\infty}(0,T;W^{1,\infty}(\Omega)^{d})}$ but independent of $\varepsilon$, and $C>0$ independent of $\varepsilon$ such that, for any $t_{0}\in [T_{0},T]$ and all weak solution $\varphi$ of $\mathcal{S}(\omega_{0},t_{0}-T_{0},t_{0},0, G,\varepsilon,\mathfrak{B})$ with data $G\in L^{2}(\mathcal{U})$ verify the following dissipation estimates:
	\begin{equation} \label{disspation 1}
		\|\varphi(\cdot, t_{0}-T_{0})\|^{2}_{L^{2}(\mathcal{U})}\leqslant C\exp\left(\frac{-C_{0}}{\varepsilon}\right)\|\varphi(\cdot, t_{0})\|^{2}_{L^{2}(\mathcal{U})}.
	\end{equation}
\end{proposition}
\begin{proof}
	Let $(x_{0},t_{0})\in\overline{\mathcal{U}}\times [T_{0},T]$. From \eqref{o5}, one has
	\begin{eqnarray}
		\exists t:=t(x_{0},t_{0})\in (t_{0}-T_{0},t_{0}),\; \forall x\in \overline{B}\left(x_{0},\frac{r_{0}}{2}\right),\; \Phi(t,t_{0},x)\notin \mathcal{\overline{\mathcal{U}}}. \label{d1}
	\end{eqnarray}
	Since $\overline{\mathcal{U}}$ is compact,then it admits a finite partition by the balls $B\left(x_{j},\frac{r_{0}}{2}\right)$, $j=1,\cdots,J$ and a partition of unity $\chi_{j}$ associated with this finite covering. For all $j=1,\cdots,J$, we consider $\theta_{j}$ the function that satisfies Lemma \ref{lemma} with the choice $x_{0}=x_{j}$, $t_{1}=t_{0}-T_{0}$, $t_{2}=t_{0}$ and $r=\frac{r_{0}}{4}$. Let $\varphi$ the weak solution of $\mathcal{S}(\omega_{0},t_{0}-T_{0},t_{0},0, G,\varepsilon,\mathfrak{B})$ and $\varphi_{j}$ the weak solution of $\mathcal{S}(\omega_{0},t_{0}-T_{0},t_{0},0,\chi_{j}G,\varepsilon,\mathfrak{B})$. By Agmon inequality \eqref{A2}, we obtain
	\begin{equation*}
		\| \psi_{j}(\cdot,t)\|_{L^{2}(\mathcal{U})}\leqslant \exp\left(\frac{C_{\mathfrak{B}}}{2}(t_{0}-t)\right)	\| \psi_{j}(\cdot,t_{0})\|_{L^{2}(\mathcal{U})},
	\end{equation*}
	for all $t\in [t_{0}-T_{0},t_{0}]$, where $\psi_{j}(\cdot, t)=\exp\left(\frac{\theta_{j}(\cdot, t)}{\varepsilon}\right)\varphi_{j}(\cdot,t)$. The properties of $\theta_{j}$ in Lemma \ref{lemma} and \eqref{d1} give$$	\begin{cases}
		\theta_{j}(x,t_{0})=0,\quad &\text{if}\;x\in \overline{B}\left(x_{j},\frac{r_{0}}{4}\right),\\
		\theta_{j}(x,t(x_{j},t_{0}))\geq \frac{c_{0}}{16}r_{0}^{2},\quad &\text{if}\; x\in\overline{\mathcal{U}}.
	\end{cases}$$
	Hence
	\begin{eqnarray}
		\| \varphi_{j}(\cdot,t(x_{j},t_{0}))\|_{L^{2}(\mathcal{U})}
		&\leqslant & \exp\left(\frac{-c_{0}r_{0}^{2}}{16\varepsilon}\right)\exp\left(\frac{C_{\mathfrak{B}}}{2}T_{0}\right)	\| \varphi_{j}(\cdot,t_{0})\|_{L^{2}(\mathcal{U})}. \label{d2}
	\end{eqnarray}
	Using Agmon inequality \eqref{A3}, we have 
	\begin{eqnarray}
		\| \varphi_{j}(\cdot,t_{0}-T_{0})\|_{L^{2}(\mathcal{U})}
		&\leqslant & \exp\left(\frac{C_{\mathfrak{B}}}{2}T_{0}\right)\| \varphi_{j}(\cdot,t(x_{j},t_{0}))\|_{L^{2}(\mathcal{U})}. \label{d3}
	\end{eqnarray}
	Based on \eqref{d2} and \eqref{d3}, we deduce that
	\begin{eqnarray} \label{d4}
		\| \varphi_{j}(\cdot,t_{0}-T_{0})\|_{L^{2}(\mathcal{U})}
		\leqslant \exp\left(\frac{-c_{0}r_{0}^{2}}{16\varepsilon}\right)\exp\left(C_{\mathfrak{B}}T_{0})\right)	\| \varphi_{j}(\cdot,t_{0})\|_{L^{2}(\mathcal{U})}.
	\end{eqnarray}
	The fact that the systems considered are linear and  $G=\displaystyle\sum_{j=1}^{J}\chi_{j}G$, we find
	\begin{equation*}
		\varphi(\cdot,t)=\sum_{j=1}^{J}\varphi_{j}(\cdot,t),\quad\text{for all}\;t_{0}-T_{0}\leqslant t\leqslant t_{0}.
	\end{equation*}
	Using this decomposition of $\varphi$ and \eqref{d4}, we then obtain the following:
	\begin{eqnarray*}
		\| \varphi(\cdot,t_{0}-T_{0})\|_{L^{2}(\mathcal{U})}
		\leqslant J\exp\left(\frac{-c_{0}r_{0}^{2}}{16\varepsilon}\right)\exp\left(C_{\mathfrak{B}}T_{0}\right)	\|\varphi(\cdot, t_{0})\|_{L^{2}(\mathcal{U})},
	\end{eqnarray*}
	which concludes estimate \eqref{disspation 1}. 
\end{proof}
\subsubsection{Global dissipation result}
\begin{proposition} \label{P3}
	Under Hypotheses \ref{A}, \ref{B} and \ref{C}, there is a constant $C>0$  independent of $\varepsilon$ such that, for any $\varphi$ solution of $\mathcal{S}(\varnothing,0,T,0, G,\varepsilon,\mathfrak{B})$ with data $G\in L^{2}(\Omega)$ satisfies the following dissipation estimates:
	\begin{enumerate}[label=(\arabic*), ref=(\arabic*)]
		\item For all $t_{0}\in [T_{0},T]$, we have
		\begin{equation} \label{disspation 2}
			\|\varphi(\cdot, t_{0}-T_{0})\|^{2}_{L^{2}(\Omega)}\leqslant C\left(\exp\left(\frac{-C_{0}}{\varepsilon}\right)\|\varphi(\cdot, t_{0})\|^{2}_{L^{2}(\Omega)}+\|\varphi\|^{2}_{L^{2}(t_{0}-T_{0},t_{0};L^{2}(\omega))}\right).
		\end{equation}
		\item For any integer $m$ such that $1\leqslant m\leqslant\frac{T}{T_{0}}$, there exists $C^{\prime}>0$ independent of $\varepsilon$ such that, for all $t\in [mT_{0},T]$, we have
		\begin{equation} \label{disspation 3}
			\|\varphi(\cdot, 0)\|^{2}_{L^{2}(\Omega)}\leqslant C^{\prime}\left(\exp\left(\frac{-mC_{0}}{\varepsilon}\right)\|\varphi(\cdot, t)\|^{2}_{L^{2}(\Omega)}+\|\varphi\|^{2}_{L^{2}(0,T;L^{2}(\omega))}\right),
		\end{equation}
    \end{enumerate}
    where $C_{0}>0$ is the constant of Proposition \ref{P2}.
\end{proposition}

\begin{proof}
	Throughout this proof $C\geq 1$ will be an independent constant of $\varepsilon$ which will be changed from one line to another, and $C_{0}, T_{0}$ are the constants of Proposition \ref{P2}.\\ 
	(1) The proof is based on the classic cut-off technique. Let us now consider $\vartheta\in \mathcal{C}^{\infty}(\mathbb{R}^{d})$ that check $\vartheta=1$ in a neighborhood of $\omega_{0}$ ($\omega_{0}$ is defined in the introduction to Subsection \ref{Dissipartion results}) and $\text{supp}(\vartheta)\subset\omega$. Define
	\begin{eqnarray*}
		\varphi_{1}(x,t)=\vartheta(x)\varphi(x,t)\quad\text{and}\quad \varphi_{2}(x,t)=(1-\vartheta(x))\varphi(x,t)\quad \text{on}\;\Omega_{T}.
	\end{eqnarray*}
	\textbf{Estimation of $\varphi_{1}$.}
	We will estimate $\varphi_{1}$ using  Proposition \ref{P0}. Firstly, one has
	\begin{eqnarray*}
		\partial_{t}\varphi_{1}+\varepsilon\Delta\varphi_{1}+\nabla\cdot(\varphi_{1}\mathfrak{B}(x,t))= F(x,t) \quad \text{on}\;\;\Omega_{T},
	\end{eqnarray*}
	where
	\begin{eqnarray*}
		F(x,t):= \varepsilon\displaystyle\sum_{i=1}^{d}\partial_{x_{i}}\left(2\varphi\partial_{x_{i}}\vartheta\right)-(\varepsilon\Delta\vartheta-\mathfrak{B}(x,t)\cdot\nabla\vartheta)\varphi.
	\end{eqnarray*}
	To apply Proposition \ref{P0}, we truncate $\varphi_{1}$ by $\psi\in\mathcal{C}^{\infty}(\mathbb{R})$ such that $\psi(t_{0}-T_{0})=1$ and $\psi(t_{0})=0$. Taking  $\varphi_{3}(x,t):=\psi(t)\varphi_{1}(x,t)$. Then 
	\begin{eqnarray*}
		\partial_{t}\varphi_{3}+\varepsilon\Delta\varphi_{3}+\nabla\cdot(\varphi_{3}\mathfrak{B}(x,t))
		&=& \psi^{\prime}(t)\vartheta(x)\varphi(x,t)+\psi(t)F(x,t):=H(x,t).
	\end{eqnarray*}
	Hence $\varphi_{3}$ is the solution of $\mathcal{S}(\varnothing,t_{0}-T_{0},t_{0}, H,0,\varepsilon,\mathfrak{B}).$ Let us apply Proposition \ref{P0} with $f_{i}=2\psi(t)\partial_{x_{i}}\vartheta\varphi,\; 1\leqslant i\leqslant d$ and $f_{0}=[\psi^{\prime}(t)\vartheta(x)-\psi(t)(\varepsilon\Delta\vartheta-\mathfrak{B}(x,t)\cdot\nabla\vartheta)]\varphi$, we obtain 
	\begin{eqnarray*}
		\|\varphi_{3}(\cdot,t_{0}-T_{0})\|^{2}_{L^{2}(\Omega)}\leqslant C\sum_{i=0}^{d}\|f_{i}\|^{2}_{L^{2}(t_{0}-T_{0},t_{0};L^{2}(\Omega))}.
	\end{eqnarray*}
	Since $\vartheta$ has support in $\omega$ and $\psi(t_{0}-T_{0})=1$, then 
	\begin{eqnarray}
		\|\varphi_{1}(\cdot,t_{0}-T_{0})\|^{2}_{L^{2}(\Omega)}\leqslant C\|\varphi\|^{2}_{L^{2}(t_{0}-T_{0},t_{0};L^{2}(\omega) )}. \label{dd10}
	\end{eqnarray}
	\textbf{Estimation of $\varphi_{2}$.} Now, we will estimate $\varphi_{2}$ by decomposing it into two solutions using Propositions \ref{P0} and \ref{P2}. Since $\varphi_{2}=\varphi-\varphi_{1}$, then 
	\begin{eqnarray*}
		\partial_{t}\varphi_{2}+\varepsilon\Delta\varphi_{2}+\nabla\cdot(\varphi_{2}\mathfrak{B}(x,t))= -F(x,t) \quad \text{on}\;\;\Omega_{T}.
	\end{eqnarray*}
	Therefore, we decompose $\varphi_{2}$ on $\mathcal{U}\times (t_{0}-T_{0},t_{0})$, 
	as follows 
	$$\begin{cases}
		\varphi_{2}=\varphi_{4}+\varphi_{5},\\
		\varphi_{4}\;\text{is the solution of}\; \mathcal{S}(\omega_{0},t_{0}-T_{0},t_{0},0,\varphi_{2}(\cdot,t_{0}),\varepsilon,\mathfrak{B}),\\
		\varphi_{5}\;\text{is the solution of}\; \mathcal{S}(\omega_{0},t_{0}-T_{0},t_{0},-F,0,\varepsilon,\mathfrak{B}).
	\end{cases}$$
	From Proposition \ref{P2} and $\varphi_{2}(x,t_{0})=(1-\vartheta(x))\varphi(x,t_{0})$, we obtain 
	\begin{equation} 
		\|\varphi_{4}(\cdot, t_{0}-T_{0})\|^{2}_{L^{2}(\mathcal{U})}\leqslant C\exp\left(\frac{-C_{0}}{\varepsilon}\right)\|\varphi(\cdot, t_{0})\|^{2}_{L^{2}(\mathcal{U})}. \label{d7}
	\end{equation}
	Concerning $\varphi_{5}$, by application of Proposition \ref{P0} with $f_{0}=(\varepsilon\Delta\vartheta-\mathfrak{B}(x,t)\cdot\nabla\vartheta)\varphi$ and  $f_{i}=-2\partial_{x_{i}}\vartheta\varphi,\;\; 1\leqslant i\leqslant d$, we get
	\begin{eqnarray*}
		\|\varphi_{5}(\cdot,t_{0}-T_{0})\|^{2}_{L^{2}(\mathcal{U})}\leqslant C\sum_{i=0}^{d}\|f_{i}\|^{2}_{L^{2}(t_{0}-T_{0},t_{0};L^{2}(\mathcal{U}))}.
	\end{eqnarray*}
	Since $\vartheta$ has support in $\omega$, then 
	\begin{eqnarray}
		\|\varphi_{5}(\;\cdot,t_{0}-T_{0})\|^{2}_{L^{2}(\mathcal{U})}
		\leqslant  C\|\varphi\|^{2}_{L^{2}(t_{0}-T_{0},t_{0};L^{2}(\omega))}. \label{d8}
	\end{eqnarray}
	The function $\vartheta=1$ in a neighborhood of $\omega_{0}$ implies that $\varphi_{2}$ has a support in $\mathcal{U}$, thus from \eqref{d7} and \eqref{d8}, we obtain
	\begin{eqnarray}
		\|\varphi_{2}(\cdot,t_{0}-T_{0})\|^{2}_{L^{2}(\Omega)}&=& \|\varphi_{2}(\cdot,t_{0}-T_{0})\|^{2}_{L^{2}(\mathcal{U})} \nonumber\\
		&\leqslant & 2\left( \|\varphi_{4}(\cdot,t_{0}-T_{0})\|^{2}_{L^{2}(\mathcal{U})}+\|\varphi_{5}(\cdot,t_{0}-T_{0})\|^{2}_{L^{2}(\mathcal{U})}\right) \nonumber\\
		&\leqslant & C\left(\exp\left(\frac{-C_{0}}{\varepsilon}\right)\|\varphi(\cdot, t_{0})\|^{2}_{L^{2}(\mathcal{U})}+\|\varphi\|^{2}_{L^{2}(t_{0}-T_{0},t_{0};L^{2}(\omega))}\right). \label{d9}
	\end{eqnarray}
	Finally, using \eqref{dd10} and \eqref{d9}, we obtain 
	\begin{eqnarray*}
		\|\varphi(\cdot,t_{0}-T_{0})\|^{2}_{L^{2}(\Omega)}&\leqslant & C\left(\exp\left(\frac{-C_{0}}{\varepsilon}\right)\|\varphi(\cdot, t_{0})\|^{2}_{L^{2}(\mathcal{U})}+\|\varphi\|^{2}_{L^{2}(t_{0}-T_{0},t_{0};L^{2}(\omega))}\right) \nonumber\\
		&\leqslant & C\left(\exp\left(\frac{-C_{0}}{\varepsilon}\right)\|\varphi(\cdot, t_{0})\|^{2}_{L^{2}(\Omega)}+\|\varphi\|^{2}_{L^{2}(t_{0}-T_{0},t_{0};L^{2}(\omega))}\right).\nonumber\\ \label{d11}
	\end{eqnarray*}
	(2)  Let $m$ be an integer such that $1\leqslant m\leqslant \frac{T}{T_{0}}$. From the first dissipation estimate \eqref{disspation 2}, we get 
	\begin{equation*} 
		\|\varphi(\cdot, (k-1)T_{0})\|^{2}_{L^{2}(\Omega)}\leqslant C\left(\exp\left(\frac{-C_{0}}{\varepsilon}\right)\|\varphi(\cdot, kT_{0})\|^{2}_{L^{2}(\Omega)}+\int_{(k-1)T_{0}}^{kT_{0}}\int_{\omega}|\varphi|^{2}\mathrm{~d}x\mathrm{~d}t\right),
	\end{equation*}
	for all $k=1,2,\cdots, m$. This last estimate gives
	\begin{eqnarray} \label{d5}
		&&C^{k-1}\exp\left(\frac{-(k-1)C_{0}}{\varepsilon}\right)\|\varphi(\cdot, (k-1)T_{0})\|^{2}_{L^{2}(\Omega)}-C^{k}\exp\left(\frac{-kC_{0}}{\varepsilon}\right)\|\varphi(\cdot, kT_{0})\|^{2}_{L^{2}(\Omega)}\leqslant C^{m}\int_{(k-1)T_{0}}^{kT_{0}}\int_{\omega}|\varphi|^{2}\mathrm{~d}x\mathrm{~d}t.
	\end{eqnarray}
	Summing \eqref{d5} from $1$ to $m,$ we obtain 
	\begin{eqnarray} \label{d10}
		\|\varphi(\cdot, 0)\|^{2}_{L^{2}(\Omega)}&\leqslant & C^{m}\left(\exp\left(\frac{-mC_{0}}{\varepsilon}\right)\|\varphi(\cdot, mT_{0})\|^{2}_{L^{2}(\Omega)}+\int_{0}^{mT_{0}}\int_{\omega}|\varphi|^{2}\mathrm{~d}x\mathrm{~d}t\right)\nonumber\\
		&\leqslant & C^{m}\left(\exp\left(\frac{-mC_{0}}{\varepsilon}\right)\|\varphi(\cdot, mT_{0})\|^{2}_{L^{2}(\Omega)}+\int_{0}^{T}\int_{\omega}|\varphi|^{2}\mathrm{~d}x\mathrm{~d}t\right).
	\end{eqnarray}
	Since $\mathfrak{B}(x,t)\cdot\mathbf{n}(x)\geq 0$ on $\Gamma_{T}$. By Agmon inequality \eqref{A3}, we have 
	\begin{eqnarray}
		\| \varphi(\cdot,mT_{0})\|^{2}_{L^{2}(\Omega)}
		&\leqslant & \exp\left(C_{\mathfrak{B}}(t-mT_{0})\right)\| \varphi(\cdot,t)\|^{2}_{L^{2}(\Omega)} \nonumber \\
		&\leqslant &  \exp\left(C_{\mathfrak{B}}T_{0}\right)\| \varphi(\cdot,t)\|^{2}_{L^{2}(\Omega)}. \label{d6}
	\end{eqnarray}
	From \eqref{d10} and \eqref{d6}, we obtain the dissipation estimate \eqref{disspation 3}. 
\end{proof}
\subsection{An observability inequality for the solutions of \eqref{s2}}
Let us consider $f\in W^{2,\infty}(\Omega_{T})$ such that, the transport field is of the time-dependent gradient form, i.e., $\mathfrak{B}(x,t)=\nabla f(x,t)$. In this case \eqref{s2} becomes 
\begin{equation} 
	\left\{
	\begin{aligned}
		\partial_t \varphi+\varepsilon\Delta \varphi+\nabla f\cdot\nabla\varphi+\Delta f\varphi &=0 & & \text { in } \Omega_{T}, \\
		\varepsilon\partial_{\n}\varphi + \varphi\partial_{\n}f&=0 & & \text { on } \Gamma_{T}, \\
		\varphi(\cdot,T)&=\varphi_{T} & & \text { in } \Omega.
		\label{2}
	\end{aligned}
	\right.
\end{equation}
In order to establish an observability inequality with an observability constant $\exp\left(\frac{C}{\varepsilon}\left(1+\frac{1}{T}\right)\right)$ for a constant $C>0$ independent of $\varepsilon$ and $T$, we have to show a Carleman estimate for the solutions of the adjoint system \eqref{2} while satisfying the constraint $s\geq \frac{C}{\varepsilon}$ and $\lambda\geq C$ (see parameters $s$ and $\lambda$ below). However, the presence of a transport field and the constraint pose challenges in this regard. To address this, we transform system \eqref{2} to a system without a transport term, using the transformation: 
\begin{eqnarray}
	\varphi(\cdot,t)\longmapsto \Phi(\cdot,t):=\exp\left(\frac{f(\cdot, t)}{2\varepsilon}\right)\varphi(\cdot,t). \label{chang2}
\end{eqnarray}
Then $\varphi$ is the solution of \eqref{2} if and only if $\Phi$ is the solution of the following system:
\begin{equation} \label{S3}
	\left\{
	\begin{aligned}
		-\partial_t\Phi-\varepsilon\Delta\Phi + a_{\varepsilon}(f)\Phi &=0 & & \text { in } \Omega_{T}, \\
		\varepsilon\partial_{\n}\Phi+b(f)\Phi  &=0 & & \text { on } \Gamma_{T}, \\
		\Phi(x,T)&=\Phi_{T}(x) & & \text{ in } \Omega,
	\end{aligned}
	\right.
\end{equation}
where $a_{\varepsilon}(f):=\frac{\mathcal{V}(f)}{\varepsilon}-\frac{\Delta f}{2}$, $\mathcal{V}(f):=\frac{|\nabla f|^{2}}{4}+\frac{\partial_{t}f}{2}$, $b(f):=\frac{\partial_{\n} f}{2}$ and $\Phi_{T}:=\exp\left(\frac{f(\cdot, T)}{2\varepsilon}\right)\varphi_{T}$.
The techniques employed are inspired by previous works such as \cite{barcena2021cost,fernandez2006null}. We introduce the following positive weight functions $\alpha_{\pm}$ and $\xi_{\pm}$ that depend only on $\Omega$ and $\omega$:
\begin{eqnarray}\label{def:alphaxi}
	\alpha_{\pm}(x, t):=\frac{\exp(6\lambda)-\exp(4\lambda\pm\lambda\eta(x))}{t(T-t)} \quad\mbox{and}\quad \xi_{\pm}(x, t):=\frac{\exp(4\lambda\pm\lambda\eta(x))}{t(T-t)},
\end{eqnarray}
where $\lambda \geq 1$ and $\eta=\eta(x)$ is a function in $\mathcal{C}^2(\overline{\Omega})$ satisfying
\begin{equation}
	\label{eta}
	\eta>0 \text { in } \Omega, \quad \eta=0 \text { on } \Gamma,\quad  \inf_{\Omega \backslash \omega^{\prime}}\left|\nabla \eta(x)\right|=\delta>0 \;\mbox{and}\; \|\eta\|_{\infty}=1,
\end{equation}
where $\omega^{\prime} \subset \subset \omega$ is a nonempty open set. If $\Omega$ is a domain with $\mathcal{C}^{2}$ smoothness, the paper \cite{fursikov1996controllability} provides a proof of the existence of $\eta$ that satisfies \eqref{eta}. The Carleman estimate we will use is as follows.
\begin{proposition} \label{P4}
	Let $T>0$, $\varepsilon\in (0,1)$, $\Omega$ is a $\mathcal{C}^{2}$ domain, $\omega \subset\Omega$ is a nonempty open set and assume that $f\in W^{2,\infty}(\Omega_{T})$ such that $\partial_{\n}f\geq c$ on $\Gamma_{T}$ for some $c>0$. Then there are constants $C>0$ and $\lambda_{1}, s_{1}\geq 1$ depend only on $\omega$ and $\Omega$ such that
	\begin{align}
		&s^{3}\lambda^{4}\int_{\Omega_{T}}\exp(-2s\alpha_{+})\xi_{+}^{3}|\Phi|^{2}\mathrm{~d}x\mathrm{~d}t +s\lambda^{2}\int_{\Omega_{T}}\exp(-2s\alpha_{+})\xi_{+}|\nabla\Phi|^{2}\mathrm{~d}x\mathrm{~d}t \nonumber\\
		&+ s\lambda^{2}
		\int_{\Gamma_{T}}
		\partial_{\n}f|\partial_{\n}\eta|^{2}\left(\xi+s\xi^{2}\right)\exp(-2s\alpha)|\Phi|^{2}\mathrm{~d}\sigma\mathrm{~d}t\leqslant Cs^{3}\lambda^{4}\int_{\omega_{T}}\exp(-2s\alpha_{+})\xi_{+}^{3}|\Phi|^{2}\mathrm{~d}x\mathrm{~d}t, 
		\label{Carleman}
	\end{align}
	for any $\Phi$ solution of \eqref{S3} with data $\Phi_{T}\in L^{2}(\Omega)$, $\lambda\geq \lambda_{1}$, $s\geq \frac{s_{1}}{\varepsilon}(T+T^{2})\mathcal{C}_{T}(f)$ with
	\begin{eqnarray}
		\mathcal{C}_{T}(f)&:=& 1+\|\nabla f\|_{\infty}+\|\nabla^{2}f\|_{\infty}+\|\nabla\partial_{t} f\|^{2/3}_{\infty}+\|\nabla f\|^{2/3}_{\infty}+ \|\partial_{t}^{2} f\|^{1/3}_{\infty}+\|\Delta f\|_{\infty}^{2/3} \nonumber\\
		&& +\|\partial_{t} f\|^{1/2}_{\infty}+\|\partial_{t}\partial_{\n}f\|_{L^{\infty}(\Gamma_{T})} +\|(\partial_{\n}f)^{-1}\|_{L^{\infty}(\Gamma_{T})}, \label{constant}
	\end{eqnarray}
 where $\|\cdot\|_{\infty}$ designates $\|\cdot\|_{L^{\infty}(\Omega_{T})}$.
\end{proposition}
\begin{proof}
	The proof of this result is given in Appendix C. 
\end{proof}
Under the same conditions of Proposition \ref{P4}, we have the following observability inequality:
\begin{corollary} 
	Let $f\in W^{2,\infty}(\Omega\times (0,+\infty))$ such that $\partial_{\n}f\geq c$ on $\Gamma\times (0,+\infty)$ for some $c>0$, and assume the same conditions in Proposition  \ref{P4}. Then, for all $0<\kappa<1$, there are two constants $C$ independent of $\varepsilon$ and $C_{1}>0$ independent of $\varepsilon$ and $T$ such that, for all $t\in [0,\kappa\,T]$, we have 
	\begin{equation}
 \int_{\Omega}|\varphi(x,t)|^{2}\mathrm{~d} x \leqslant C\exp\left(\frac{C_{1}}{\varepsilon}\left(1+\frac{1}{T}\right)\right)\int_{\omega_{T}}|\varphi(x,t)|^{2}\mathrm{~d} x \mathrm{~d}t.\label{io1}
	\end{equation}
\end{corollary}
\begin{proof}
    By Carleman estimate \eqref{Carleman}, we obtain
	\begin{equation}
		\begin{aligned}
			\int_{\Omega_{T}} \exp(-2 s \alpha_{+})\xi_{+}^3|\Phi|^2 \mathrm{d} x \mathrm{~d}t \leq C  \int_{\omega_{T}} \exp(-2 s \alpha_{+}) \xi_{+}^3|\Phi|^2 \mathrm{~d} x \mathrm{~d}t,
		\end{aligned}
	\end{equation}
	where $\lambda=\lambda_{1}$ and $s = \frac{s_{1}}{\varepsilon}(T^{2}+T)\mathcal{C}_{\infty}(f)$; see definition of $\mathcal{C}_{T}(f)$ in \eqref{constant}. Note that $\mathcal{C}_{\infty}(f)$ is well defined and independent of $T$, since $f\in W^{2,\infty}(\Omega\times (0,\infty))$ and $\partial_{\n}f\geq c$ on $\Gamma\times (0,+\infty)$. Taking lower and upper estimates with respect to $x$ of the weight functions, we get 
	\begin{equation}
		\begin{aligned}
			\int_{\Omega_{T}}\check{h}(t) |\Phi(x,t)|^2 \mathrm{d} x \mathrm{~d} t \leq C\int_{\omega_{T}}\hat{h}(t)|\Phi(x,t)|^2 \mathrm{~d} x \mathrm{~d}t,
		\end{aligned}
	\end{equation}
	where 
	$$\begin{cases*} \displaystyle
		\check{h}(t):=\exp \left(-2 s \max _{x \in \overline{\Omega}} \alpha_{+}(x,t)\right) \min _{x \in \overline{\Omega}} \xi^{3}_{+}(x,t),\\
		\displaystyle \hat{h}(t):=\exp \left(-2 s \min _{x \in \overline{\Omega}} \alpha_{+}(x,t)\right) \max _{x \in \overline{\Omega}} \xi^{3}_{+}(x,t).
	\end{cases*}$$
	For $\lambda_{1}$ and $s_{1}$ large enough, it is easy to check that the function $\hat{h}$ admits a maximum on $[0,T]$ at $t=\frac{T}{2}$ and
	$\check{h}$ admits a minimum on $\left[\kappa T,\left(\frac{1+\kappa}{2}\right)T\right]$ at $t=\kappa_{0}T$ where $\kappa_{0}\in \left[\kappa,\frac{1+\kappa}{2}\right]$. Hence 
	\begin{eqnarray}
		& & \int_{\Omega\times \left(\kappa T,\left(\frac{1+\kappa}{2}\right)T\right)} |\Phi(x,t)|^2 \mathrm{d} x \mathrm{~d} t \leq  C\frac{\hat{h}\left(\frac{T}{2}\right)}{\check{h}(\kappa_{0}T)}\int_{\omega_{T}}|\Phi(x,t)|^2 \mathrm{~d} x \mathrm{~d}t.  \label{ee2}
	\end{eqnarray}
	Using \eqref{ee2} and the transformation \eqref{chang2}, we obtain
	\begin{eqnarray}
		& & \int_{\Omega\times \left(\kappa T,\left(\frac{1+\kappa}{2}\right)T\right)} |\varphi(x,t)|^2 \mathrm{d} x \mathrm{~d} t  \leq  \exp\left(\frac{C_{1}}{\varepsilon}\left(1+\frac{1}{T}\right)\right)\int_{\omega_{T}}|\varphi(x,t)|^2 \mathrm{~d} x \mathrm{~d}t.  \label{e2}
	\end{eqnarray}
	for some $C_{1}$ depends only on $\Omega$, $\omega$, $\kappa$ and $\mathcal{C}_{\infty}(f)$.\\
	Using the dissipation estimate \eqref{A3}
	for the solutions of \eqref{s2} with $\mathfrak{B}=\nabla f$, we get for all $ 0\leqslant t\leqslant s\leqslant T$,
	\begin{eqnarray*}
		\int_{\Omega}|\varphi(x,t)|^{2}\mathrm{~d}x&\leqslant& \exp\left(\|\Delta f\|_{L^{\infty}(\Omega_{T})}T \right)
		\int_{\Omega}|\varphi(x,s)|^{2}\mathrm{~d}x.
	\end{eqnarray*}
	By integrating this inequality on $\left(\kappa T,\left(\frac{1+\kappa}{2}\right)T \right)$, we obtain for all $t\in [0,\kappa T]$,
	\begin{eqnarray}
		\int_{\Omega}|\varphi(x,t)|^{2}dx\leqslant \frac{2\exp\left(\|\Delta f\|_{L^{\infty}(\Omega_{T})}T \right)}{(1-\kappa)T}
		\int_{\Omega\times \left(\kappa T,\left(\frac{1+\kappa}{2}\right)T\right)}|\varphi(x,s)|^{2}\mathrm{~d} x \mathrm{~d}s.\label{e3}
	\end{eqnarray}
	Based on \eqref{e2} and \eqref{e3}, we obtain \eqref{io1}.
\end{proof}

\section{Conclusions} \label{Section 5}
In this paper, we have partially answered the interesting open problem proposed in Remark 3 of \cite{guerrero2007singular}. We have established that the control cost is uniform for a sufficiently small diffusivity when the time control is sufficiently large in the case where each trajectory of the velocity of the posed system coming from the domain
enters the control region in a shorter time for a fixed input time. It can also be established that the controllability cost tends towards $0$ exponentially for a sufficiently small diffusivity when the time control is sufficiently large in the case where each trajectory coming from the domain
exits the domain in a shorter time for a fixed exit time, as shown by Theorem 2 in \cite{guerrero2007singular} in the case of Dirichlet conditions. An interesting question is to establish the same results for general transport field. For this, Agmon inequality and dissipation results are shown in this article, and it remains to establish a Carleman estimate for general transport field. 
\section*{Appendix A} 
\textbf{Proof of Theorem \ref{m1}:}
According to estimates \eqref{io1} and \eqref{disspation 3} we have:\\
For all $0<\kappa<1$, there are two constants $C$ independent of $\varepsilon$, $C_{1}>0$ independent of $\varepsilon$ and $T$ such that, for all $t\in [0,\kappa T]$, we have 
\begin{equation}
	\|\varphi(\cdot,t)\|^{2}_{L^{2}(\Omega)} \leqslant C\exp\left(\frac{C_{1}}{\varepsilon}\left(1+\frac{1}{T}\right)\right)\|\varphi\|^{2}_{L^{2}(\omega_{T})} \label{e4}
\end{equation}
and there is a constant $C_{0}>0$ independent of $\varepsilon$ and $T$ (Note that $C_{0}$ independent of $T$, because it can be taken to depend on $T_{0}$, $r_{0}$ and $\|\mathbf{B}\|_{L^{\infty}(0,\infty;W^{1,\infty}(\Omega)^{d})}$) such that, for any integer $m$ such that $1\leqslant m\leqslant\frac{T}{T_{0}}$, there exists $C^{\prime}>0$  independent of $\varepsilon$ such that, for all $t\in [mT_{0},T]$, we have 
\begin{equation} 
	\|\varphi(\cdot, 0)\|^{2}_{L^{2}(\Omega)}\leqslant C^{\prime}\left(\exp\left(\frac{-mC_{0}}{\varepsilon}\right)\|\varphi(\cdot, t)\|^{2}_{L^{2}(\Omega)}+\|\varphi\|^{2}_{L^{2}(\omega_{T})}\right). \label{e5}
\end{equation}
For $0<\kappa<1$ fixed, taking $m:=\left[\frac{C_{1}}{C_{0}} \right] +1$ where $\left[\frac{C_{1}}{C_{0}} \right]$ denotes the integer part of $\frac{C_{1}}{C_{0}}$ and $\rho_{0}:=\max\left(\frac{m}{\kappa},\frac{1}{T_{0}\left(m\frac{C_{0}}{C_{1}}-1\right)}\right)$.\\
Let $T\geq \rho_{0} T_{0}$. Then $mT_{0}\leq \kappa T$, so applying \eqref{e4} and \eqref{e5}, we get
\begin{equation} \label{e6}
	\|\varphi(\cdot, 0)\|^{2}_{L^{2}(\Omega)}\leqslant C^{\prime}\left(C\exp\left(\frac{C_{1}\left(1+1/T\right)-mC_{0}}{\varepsilon}\right)+1\right)\|\varphi\|^{2}_{L^{2}(\omega_{T})}.
\end{equation} 
On the other hand, the choice of $\rho_{0}$ implies that  $C_{1}\left(1+1/T\right)-mC_{0}\leq 0$ for all $T\geq \rho_{0} T_{0}$. Finally combining \eqref{e6} and \eqref{cost of control}, we obtain \eqref{cost1} for $\varepsilon>0$ small enough and $T\geq \rho_{0} T_{0}$. \qed 
\section*{Appendix B}
\textbf{Proof of Theorem \ref{m2}:} Let $x_{0}\in\Omega$ such that condition \ref{C3m2} of Theorem \ref{m2} is fulfilled. From the continuity of $x\mapsto\Phi(t,T,x)$ uniform in $t$, there exists $r_{0}>0$ such that,
$$\Phi(t,T,x)\in \Omega\setminus\overline{\omega}\quad\forall t\in [0,T],\;\forall x\in \overline{B}(x_{0},4r_{0}).$$
Consider $\varphi_{T}\in D(B(x_{0},r_{0}))$ and $\varphi$ be the weak solution of \eqref{s2} with the data $\varphi_{T}$. If necessary to extend $\mathfrak{B}$ by a function in $L^{\infty}(0,T;W^{1,\infty}(\mathbb{R}^{d})^{d})$ (note that this extension is not unique, but this proof does not depend on the extension), we can apply the Lemma \ref{lemma}, let $\theta$ the function defined in Lemma \ref{lemma} with this choice of $x_{0}$, $r=r_{0}$, $t_{1}=0$, $t_{2}=T$.\\ Let $\vartheta_{1}$ and $\vartheta_{2}$ be regular functions such that
\begin{equation*}
	\left\{
	\begin{aligned}
		\vartheta_{1}(x,t) &=0 & & \forall (x,t)\in \mathcal{D}_{2r_{0}}(x_{0},0,T), \\
		\vartheta_{1}(x,t) &=1 & & \forall (x,t)\notin \mathcal{D}_{3r_{0}}(x_{0},0,T)	
	\end{aligned}
	\right.\quad\text{and}\quad 	
	\left\{
	\begin{aligned}
		\vartheta_{2}(x,t) &=0 & & \forall (x,t)\in \mathcal{D}_{3r_{0}}(x_{0},0,T), \\
		\vartheta_{2}(x,t) &=1 & & \forall (x,t)\notin \mathcal{D}_{4r_{0}}(x_{0},0,T).	
	\end{aligned}
	\right.
\end{equation*}
For reasons of simplicity, we will divide the proof into three steps.\\
\textbf{\underline{Step 1.}} We will show that there are $C_{1},C_{2}>0$ independent of $\varepsilon$ such that
\begin{equation} \label{7}
	\int_{\Omega_{T}}|\vartheta_{1}\varphi|^{2}\mathrm{~d}x\mathrm{~d}t+
	\int_{\Omega_{T}}|\nabla(\vartheta_{1}\varphi)|^{2}\mathrm{~d}x\mathrm{~d}t\leqslant C_{1}\exp\left(-\frac{C_{2}}{\varepsilon}\right)\int_{\Omega}|\psi(x,T)|^{2}\mathrm{~d}x,\\
\end{equation}
where $\psi=\exp\left(\frac{\theta}{\varepsilon}\right)\varphi$ and $T, \varepsilon$ are small enough.\\
Indeed, for all $(x,t)\in\text{supp}(\vartheta_{1})$, we have $\theta(x,t)\geq c_{0}r_{0}^{2}$, then 
\begin{eqnarray} \label{17}
	\int_{\Omega_{T}}|\vartheta_{1}\varphi|^{2}\mathrm{~d}x\mathrm{~d}t&\leqslant &  \exp\left(-\frac{2c_{0}r_{0}^{2}}{\varepsilon}\right)\int_{\Omega_{T}}|\vartheta_{1}\psi|^{2}\mathrm{~d}x\mathrm{~d}t \nonumber\\
	&\leqslant & \|\vartheta_{1}\|^{2}_{\infty}\exp\left(-\frac{2c_{0}r_{0}^{2}}{\varepsilon}\right)\int_{\Omega_{T}}|\psi|^{2}\mathrm{~d}x\mathrm{~d}t.
\end{eqnarray}
From $\nabla(\vartheta_{1}\varphi)=\exp\left(-\frac{\theta}{\varepsilon}\right)\left(\psi\left(\nabla\vartheta_{1}-\vartheta_{1}\frac{\nabla\theta}{\varepsilon}\right)+\vartheta_{1}\nabla \psi\right)$, we obtain
\begin{eqnarray} \label{18}
	\int_{\Omega_{T}}|\nabla(\vartheta_{1}\varphi)|^{2}\mathrm{~d}x\mathrm{~d}t&\leqslant &  4\exp\left(-\frac{2c_{0}r_{0}^{2}}{\varepsilon}\right)\left(\|\nabla\vartheta_{1}\|^{2}_{\infty}+\frac{\|\vartheta_{1}\nabla\theta\|^{2}_{\infty}}{\varepsilon^{2}}\right)\int_{\Omega_{T}}|\psi|^{2}\mathrm{~d}x\mathrm{~d}t \nonumber\\
	&& + 2\exp\left(-\frac{2c_{0}r_{0}^{2}}{\varepsilon}\right) \|\vartheta_{1}\|^{2}_{\infty}\int_{\Omega_{T}}|\nabla \psi|^{2}\mathrm{~d}x\mathrm{~d}t. 
\end{eqnarray}
By application of Agmon inequality \eqref{A1}, we get 
\begin{equation} \label{19}
	\int_{\Omega_{T}}|\psi|^{2}\mathrm{~d}x\mathrm{~d}t\leqslant T\exp\left(\frac{C}{\varepsilon}T\right)\int_{\Omega}|\psi(x,T)|^{2}\mathrm{~d}x
\end{equation}
and 
\begin{equation} \label{20}
	\int_{\Omega_{T}}|\nabla \psi|^{2}\mathrm{~d}x\mathrm{~d}t\leqslant \exp\left(\frac{C}{\varepsilon}T\right)\int_{\Omega}|\psi(x,T)|^{2}\mathrm{~d}x.
\end{equation}
Taking $0<T<\frac{2c_{0}r_{0}^{2}}{C}$ and $\varepsilon$ small enough, from \eqref{17}-\eqref{20}, we obtain \eqref{7}.\\
\textbf{\underline{Step 2.}} We will prove that there are
$C_{1},C_{2}>0$ independent of $\varepsilon$ such that
\begin{equation} \label{10}
	\int_{\omega_{T}}|\varphi|^{2}\mathrm{~d}x\mathrm{~d}t\leqslant C_{1}\exp\left(-\frac{C_{2}}{\varepsilon}\right)\int_{\Omega}|\varphi_{T}|^{2}\mathrm{~d}x,
\end{equation}
for $T$ and $\varepsilon$ are small enough.\\
Let $\phi=\vartheta_{2}\varphi$, then $\phi$ is the solution of system $\mathcal{S}(\varnothing,0,T, F,0,\varepsilon,\mathfrak{B})$, 
where $F:=(\partial_t \vartheta_{2}+\varepsilon\Delta\vartheta_{2}-\mathfrak{B}(x,t)\cdot\nabla\vartheta_{2})\varphi+2\varepsilon\nabla\vartheta_{2}\cdot\nabla\varphi$.\\
From estimate \eqref{dissw}, we obtain  
\begin{equation} \label{diss11}
	\int_{\Omega}|\phi(x,t)|^{2}\mathrm{~d}x\leqslant C\exp\left(\frac{CT}{\varepsilon}\right)\int_{\Omega_{T}}|F|^{2}\mathrm{~d}x\mathrm{~d}t \quad\forall t\in [0,T],
\end{equation}
for $\varepsilon$ is small enough and $C>0$ independent of $\varepsilon$.\\
Since $\vartheta_{1}=1$ on the supports of the functions  $\partial_t \vartheta_{2}$ and $\nabla\vartheta_{2}$, we obtain from \eqref{7} the existence of constants $C_{1},C_{2}>0$ independent of $\varepsilon$ such that
\begin{eqnarray} \label{diss10}
	\int_{\Omega_{T}}|F|^{2}\mathrm{~d}x\mathrm{~d}t &\leqslant& C_{1}\exp\left(-\frac{C_{2}}{\varepsilon}\right)\int_{\Omega}|\psi(x,T)|^{2}\mathrm{~d}x = C_{1}\exp\left(-\frac{C_{2}}{\varepsilon}\right)\int_{\Omega}|\varphi_{T}|^{2}\mathrm{~d}x,
\end{eqnarray}
due to $\theta(\cdot,T)=0$ on the support of $\varphi_{T}$.\\
Using \eqref{diss11} and \eqref{diss10}, there exists a constant $C^{\prime}_{1}>0$ independent of $\varepsilon$ such that
\begin{equation*}
	\int_{\Omega_{T}}|\phi|^{2}\mathrm{~d}x\mathrm{~d} t\leqslant C_{1}^{\prime}\exp\left(-\frac{C_{2}}{\varepsilon}\right)\int_{\Omega}|\varphi_{T}|^{2}\mathrm{~d}x,
\end{equation*}
for $T>0$ and $\varepsilon$ are small enough, since $\vartheta_{2}=1$ on $\omega_{T}$, we deduce that \eqref{10} is true.\\
\textbf{\underline{Step 3.}} Finally, since $\varphi$ is the weak solution of adjoint system \eqref{s2}, from \eqref{weak solution}, we obtain for all $v\in D(]0,T[)$
\begin{eqnarray*}
	\int_{0}^{T}\left(\int_{\Omega}\varphi(x,t)\d x\right) v^{\prime}(t)\d t=-\int_{0}^{T}\mathfrak{a}_{w}(T-t,\varphi(t),v(t))\d t=0,
\end{eqnarray*}
since $v$ independent of $x$. Hence $t\longmapsto \displaystyle\int_{\Omega}\varphi(x,t)\d x$ is weakly differentiable and $\displaystyle\frac{\d}{\d t} \int_{\Omega}\varphi(x,t)\d x=0$, thus
\begin{equation} 
	\int_{\Omega}\varphi(x,0)\mathrm{~d}x=\int_{\Omega}\varphi(x,T)\mathrm{~d}x. \label{e1}
\end{equation}
Choosing the initial data $\varphi_{T}\in D(B(x_{0},r_{0}))$ such that $\displaystyle\int_{\Omega}\varphi_{T}(x)\mathrm{~d}x\neq 0$. By Hölder's inequality and \eqref{e1}, we have 
\begin{equation} \label{11}
	\int_{\Omega}|\varphi(x,0)|^{2}\mathrm{~d}x\geq\frac{1}{|\Omega|}\left|	\int_{\Omega}\varphi_{T}(x)\mathrm{~d}x\right|^{2}.
\end{equation}
Finally, combining \eqref{10}, \eqref{11} and \eqref{cost of control}, we obtain \eqref{cost2}.\qed
\section*{Appendix C} 
\textbf{Proof of Proposition \ref{P4}:} To derive the global estimate \eqref{Carleman}, we will give the proof in several steps. 
Initially, a change of variables is implemented to acquire functions that display decay characteristics at both the initial time $t=0$ and the final time $t=T$.
Subsequently, we assess and approximate the scalar product that arises naturally during the change of variables.
Afterwards, we draw preliminary conclusions by examining the boundary terms on the left-hand side of the inequality.
We then estimate the local gradient term.
Additionally, we simplify the boundary terms and revert the change of variables to obtain the desired estimate.\\
\par 
Throughout the proof $\|\cdot\|_{\infty}$ designates the norm $\|\cdot\|_{L^{\infty}(\Omega_{T})}$. The constants $C$, $c$, $s_{1}$ and $\lambda_{1}$ will denote generic constants which are independent of $\varepsilon$, $s$, $\lambda$ and $\mathfrak{B}$. These constants may vary even from line to line.\\
\par To summarize the proof a little,  
we will use the conclusions of steps 2a, 2b and 2c in \cite[Proof of Proposition 5]{et2024asymptotic}.\\ 
\textbf{\underline{Step 1.} An auxiliary problem.} Using the density argument explained in \cite{barcena2021cost} before the proof of Proposition 3.5, we can make computations with $\Phi$ sufficiently regular that we can proceed to integration by parts involving the Laplacian term, and preserve the boundary conditions of \eqref{S3}. Let $\lambda\geq 1$, $s\geq 1$ parameters to be specified. Define 
\begin{eqnarray}
	\psi_{\pm}:=\exp(-s\alpha_{\pm})\Phi \quad \mbox{and} \quad F_{\pm}:=\exp(-s\alpha_{\pm})(\partial_{t}\Phi+\varepsilon\Delta\Phi-a_{\varepsilon}(f)\Phi). \label{c-4}
\end{eqnarray}
\par
We recall the definition of the tangential derivative $\nabla_{\Gamma}$ of a regular function $h\in \mathcal{C}^{1}(\overline{\Omega})$ is given by  $\nabla_{\Gamma}h:=\nabla h-(\partial_{\n}h)\n$
and that this definition depends only on the image of $h$ on $\Gamma$. Since $\alpha_{+}=\alpha_{-}$ on $\Gamma_{T}$, then 
\begin{eqnarray}
	\psi_{+}=\psi_{-}\quad \mbox{and} \quad \nabla_{\Gamma}\psi_{+}=\nabla_{\Gamma}\psi_{-}\;\;\text{on}\;\;\Gamma_{T}. \label{c5}
\end{eqnarray}
On $\Gamma_{T}$ we will note respectively $\psi$ and $\nabla_{\Gamma}\psi$ instead of $\psi_{\pm}$ and $\nabla_{\Gamma}\psi_{\pm}$.\\
\par
We determine
the problem solved by $\psi_{\pm}$. We first expand the spatial derivatives of $\alpha_{\pm}$ by the chain
rule to bring $\eta$ into play, but we do not expand $\partial_{t}\alpha_{\pm}$. We calculate
\begin{eqnarray}
	\nabla\alpha_{\pm}&=& -\nabla\xi_{\pm}=\mp\lambda\xi_{\pm}\nabla\eta \nonumber\\
	\Delta\alpha_{\pm} &=& 	-\lambda^{2}\xi_{\pm}|\nabla\eta|^{2} \mp\lambda\xi_{\pm}\Delta\eta \nonumber\\
	\partial_{t}\psi_{\pm} &=& \exp(-s\alpha_{\pm})\partial_{t}\Phi-s\partial_{t}\alpha_{\pm}\psi_{\pm} \label{c-6}\\
	\nabla\psi_{\pm}&=&\exp(-s\alpha_{\pm})\nabla\Phi-s\psi_{\pm}\nabla\alpha_{\pm} =\exp(-s\alpha_{\pm})\nabla\Phi \pm s\lambda\xi_{\pm}\psi_{\pm}\nabla\eta \label{c-5}\\
	\partial_{\n}\psi_{\pm} &=& \exp(-s\alpha_{\pm})\partial_{\n}\Phi \pm s\lambda\xi_{\pm}\psi_{\pm}\partial_{\n}\eta = \exp(-s\alpha)\partial_{\n}\Phi \pm s\lambda\xi\psi\partial_{\n}\eta
	\label{c-3}\\
	\Delta\psi_{\pm} &=& \exp(-s\alpha_{\pm})\Delta\Phi+\nabla(\exp(-s\alpha_{\pm}))\cdot\nabla\Phi-s\psi_{\pm}\Delta\alpha_{\pm} -s(\nabla\psi_{\pm}\cdot\nabla\alpha_{\pm}) \nonumber\\
	&=& \exp(-s\alpha_{\pm})\Delta\Phi-s^{2}\psi_{\pm}|\nabla\alpha_{\pm}|^{2} -2s(\nabla\psi_{\pm}\cdot\nabla\alpha_{\pm})-s\psi_{\pm}\Delta\alpha_{\pm} \nonumber\\
	&=& \exp(-s\alpha_{\pm})\Delta\Phi	-s^{2}\lambda^{2}\xi_{\pm}^{2}\psi_{\pm}|\nabla\eta|^{2} \pm 2s\lambda\xi_{\pm}(\nabla\eta\cdot\nabla\psi_{\pm}) \nonumber\\
	& & + s\lambda^{2}\xi_{\pm}\psi_{\pm}|\nabla\eta|^{2}
	\pm s\lambda\xi_{\pm}\psi_{\pm}\Delta\eta. \nonumber
\end{eqnarray}
On $\Omega_{T}$ this yields transformed evolution equations
\begin{eqnarray*}
	\partial_{t}\psi_{\pm}+\varepsilon\Delta\psi_{\pm}-a_{\varepsilon}(f)\psi_{\pm} &=& F_{\pm}-s\partial_{t}\alpha_{\pm}\psi_{\pm}
	-\varepsilon s^{2}\lambda^{2}\xi^{2}_{\pm}|\nabla\eta|^{2}
	\psi_{\pm} \pm 2\varepsilon s\lambda \xi_{\pm}(\nabla\eta\cdot\nabla\psi_{\pm})\\
	&& + \varepsilon s\lambda^{2}\xi_{\pm}|\nabla\eta|^{2}\psi_{\pm}
	\pm  \varepsilon s \lambda\xi_{\pm}\Delta\eta\psi_{\pm}.
\end{eqnarray*}
We rewrite this equality as 
\begin{eqnarray}
	L_{1}\psi_{\pm}+L_{2}\psi_{\pm}=L_{3}\psi_{\pm}, \label{c1}
\end{eqnarray}
where 
\begin{eqnarray}
	L_{1}\psi_{\pm}&:=& - 2\varepsilon s\lambda^{2}\xi_{\pm}|\nabla\eta|^{2}\psi_{\pm}\mp 2\varepsilon s\lambda\xi_{\pm}(\nabla\eta\cdot\nabla\psi_{\pm})+\partial_{t}\psi_{\pm}, \nonumber \\
	L_{2}\psi_{\pm} &:=& \varepsilon s^{2}\lambda^{2}\xi_{\pm}^{2}|\nabla\eta|^{2}\psi_{\pm}+\varepsilon\Delta\psi_{\pm}+s\partial_{t}\alpha_{\pm}\psi_{\pm}-\frac{\mathcal{V}(f)}{\varepsilon}\psi_{\pm}, \label{c-1}\\
	L_{3}\psi_{\pm}&:=& F_{\pm} \pm \varepsilon s\lambda\xi_{\pm}\Delta\eta\psi_{\pm}
	-\varepsilon s\lambda^{2}\xi_{\pm}|\nabla\eta|^{2}\psi_{\pm}-\frac{\Delta f}{2}\psi_{\pm}. \label{c0}
\end{eqnarray}
\begin{remark}
	In this decomposition, we have split the potential term $a_{\varepsilon}(f)$ into two parts $\frac{\Delta f}{2}$ and $\frac{\mathcal{V}(f)}{\varepsilon}$ in order to absorb the terms associated with constraint $s\geq \frac{C}{\varepsilon}$.
\end{remark}
Applying $\|\cdot\|_{L^{2}(\Omega_{T})}^2$ to the equation \eqref{c1}, we obtain 
\begin{eqnarray}
	\|L_{1}\psi_{\pm}\|^{2}_{L^{2}(\Omega_{T})}+2	(L_{1}\psi_{\pm},L_{2}\psi_{\pm})_{L^{2}(\Omega_{T})}+\|L_{2}\psi_{\pm}\|^{2}_{L^{2}(\Omega_{T})}=\|L_{3}\psi_{\pm}\|^{2}_{L^{2}(\Omega_{T})}. \label{c2}
\end{eqnarray}
\textbf{\underline{Step 2.} Estimating the mixed terms in \eqref{c2} from below.}
The main idea is to expand the term $(L_{1}\psi_{\pm},L_{2}\psi_{\pm})_{L^{2}(\Omega_{T})}$ and use the particular structure of $\alpha_{\pm}$ and the fact
that $s$ is large enough in order to obtain large positive terms in this scalar product. Denoting by $\left(L_{i}\psi_{\pm}\right)_{j}$ the $j$-th term in the above expression of $L_{i}\psi_{\pm}$. We have
\begin{eqnarray*}
	(L_{1}\psi_{\pm},(L_{2}\psi_{\pm})_{j})_{L^{2}(\Omega_{T})}&=&\sum_{i=1}^{3}((L_{1}\psi_{\pm})_{i},(L_{2}\psi_{\pm})_{j})_{L^{2}(\Omega_{T})},\; j=1,\cdots,4,\\
	(L_{1}\psi_{\pm},L_{2}\psi_{\pm})_{L^{2}(\Omega_{T})}&=&\sum_{j=1}^{4}(L_{1}\psi_{\pm},(L_{2}\psi_{\pm})_{j})_{L^{2}(\Omega_{T})}.
\end{eqnarray*}
Let us compute each term $(L_{1}\psi_{\pm},(L_{2}\psi_{\pm})_{j})_{L^{2}(\Omega_{T})}$, $j=1,\cdots,4$.\\
\textbf{\underline{Step $2a$.} Estimate from below of $(L_{1}\psi_{\pm},(L_{2}\psi_{\pm})_{1})_{L^{2}(\Omega_{T})}$}.\\
The term $(L_{1}\psi_{\pm},(L_{2}\psi_{\pm})_{1})_{L^{2}(\Omega_{T})}$ is exactly treated in \cite{et2024asymptotic}. From the conclusion of Step 2a in \cite[Proof of Proposition 5]{et2024asymptotic}, we have
	\begin{eqnarray}
		(L_{1}\psi_{\pm},(L_{2}\psi_{\pm})_{1})_{L^{2}(\Omega_{T})}&\geq&  c\,\varepsilon^{2}s^{3}\lambda^{4}\int_{\Omega_{T}}\xi_{\pm}^{3}|\psi_{\pm}|^{2}\mathrm{~d}x\mathrm{~d}t -C\varepsilon^{2}s^{3}\lambda^{4}\int_{\omega^{\prime}_{T}}\xi_{\pm}^{3}|\psi_{\pm}|^{2}\mathrm{~d}x\mathrm{~d}t \nonumber\\
&& \mp\varepsilon^{2}s^{3}\lambda^{3}\int_{\Gamma_{T}}\xi^{3} |\nabla\eta|^{2}\partial_{\n}\eta|\psi|^{2}\mathrm{~d}\sigma\mathrm{~d}t, \label{c8}
\end{eqnarray}
for any $\lambda\geq \lambda_{1}$ and any $s\geq s_{1}\frac{T}{\varepsilon}$.\\\\
\textbf{\underline{Step $2b$.} Estimate from below of $(L_{1}\psi_{\pm},(L_{2}\psi_{\pm})_{2})_{L^{2}(\Omega_{T})}$.}\\
Similarly, the term $(L_{1}\psi_{\pm},(L_{2}\psi_{\pm})_{2})_{L^{2}(\Omega_{T})}$ is exactly treated in \cite{et2024asymptotic}. Using the same computations and arguments of Step 2b in \cite[Proof of Proposition 5]{et2024asymptotic}, one has 
\begin{eqnarray}
	(L_{1}\psi_{\pm},(L_{2}\psi_{\pm})_{2})_{L^{2}(\Omega_{T})}&\geq& -2\varepsilon^{2}s\lambda^{2}\int_{\Gamma_{T}}\xi|\nabla\eta|^{2}\psi\partial_{\n}\psi_{\pm}\mathrm{~d}\sigma\mathrm{~d}t -C\varepsilon^{2}s^{2}\lambda^{4}\int_{\Omega_{T}} \xi_{\pm}^{2}|\psi_{\pm}|^{2}\mathrm{~d}x\mathrm{~d}t + c\,\varepsilon^{2}s\lambda^{2}\int_{\Omega_{T}}\xi_{\pm}|\nabla\psi_{\pm}|^{2}\mathrm{~d}x\mathrm{~d}t\nonumber\\
	&& -C\varepsilon^{2}s\lambda^{2}\int_{\omega^{\prime}_{T}}\xi_{\pm}|\nabla\psi_{\pm}|^{2}\mathrm{~d}x\mathrm{~d}t \mp 2\varepsilon^{2}s\lambda\int_{\Gamma_{T}}\xi\partial_{\n}\eta |\partial_{\n}\psi_{\pm}|^{2}\mathrm{~d}\sigma\mathrm{~d}t \nonumber\\
	&& \pm \varepsilon^{2}s\lambda\int_{\Gamma_{T}}\xi\partial_{\n}\eta|\nabla\psi_{\pm}|^{2}\mathrm{~d}\sigma\mathrm{~d}t+\varepsilon\int_{\Gamma_{T}}\partial_{t}\psi\partial_{\n}\psi_{\pm}\mathrm{~d}\sigma\mathrm{~d}t. \label{c12}
\end{eqnarray}
for any $\lambda\geq\lambda_{1}$ and any $s\geq s_{1}T^{2}$.\\
\textbf{\underline{Step 2c.} Estimate from below of $(L_{1}\psi_{\pm},(L_{2}\psi_{\pm})_{3})_{L^{2}(\Omega_{T})}$.}\\
From the conclusion of Step 2c in \cite[Proof of Proposition 5]{et2024asymptotic}, we have
\begin{eqnarray}
	&& (L_{1}\psi_{\pm},(L_{2}\psi_{\pm})_{3})_{L^{2}(\Omega_{T})}\geq -C\varepsilon^{2} s^{3}\lambda^{2}\int_{\Omega_{T}} \xi_{\pm}^{3}|\psi_{\pm}|^{2}\mathrm{~d}x\mathrm{~d}t \mp \varepsilon s^{2}\lambda\int_{\Gamma_{T}}\xi \partial_{t}\alpha\partial_{\n}\eta|\psi|^{2}\mathrm{~d}\sigma\mathrm{~d}t,
	\label{c16}
\end{eqnarray}
for any $\lambda\geq 1$ and any $s\geq s_{1} \frac{T}{\varepsilon}$.\\
\textbf{\underline{Step 2d.} Estimate from below of $(L_{1}\psi_{\pm},(L_{2}\psi_{\pm})_{4})_{L^{2}(\Omega_{T})}$.}\\
In this step, there are differences from step 2d of \cite[Proof of Proposition 5]{et2024asymptotic} due to the time dependency of $f$.
Let us now consider double products involving $(L_{2}\psi_{\pm})_{4}$. First, we have
\begin{eqnarray*}
	((L_{1}\psi_{\pm})_{1},(L_{2}\psi_{\pm})_{4})_{L^{2}(\Omega_{T})}=2s\lambda^{2}
	\int_{\Omega_{T}} \xi_{\pm}\mathcal{V}(f)|\nabla\eta|^{2}|\psi_{\pm}|^{2}\mathrm{~d}x\mathrm{~d}t.
\end{eqnarray*}
Since $\xi_{\pm}\geq\frac{4}{T^{2}}$, then
\begin{eqnarray}
	|((L_{1}\psi_{\pm})_{1},(L_{2}\psi_{\pm})_{4})_{L^{2}(\Omega_{T})}|
	&\leqslant& C(\|\nabla f\|^{2}_{\infty}+\|\partial_{t}f\|_{\infty})s\lambda^{2}T^{4}\int_{\Omega_{T}} \xi_{\pm}^{3}|\psi_{\pm}|^{2}\mathrm{~d}x\mathrm{~d}t.  \label{c17}
\end{eqnarray}
Next, it is obvious that
\begin{eqnarray*}
	((L_{1}\psi_{\pm})_{2},(L_{2}\psi_{\pm})_{4})_{L^{2}(\Omega_{T})}&=& \pm 2s\lambda\int_{\Omega_{T}} \xi_{\pm}\mathcal{V}(f)(\nabla\eta\cdot\nabla\psi_{\pm})\psi_{\pm}\mathrm{~d}x\mathrm{~d}t\\
	&=& \pm \frac{s\lambda}{2}\int_{\Omega_{T}} \xi_{\pm}|\nabla f|^{2}(\nabla\eta\cdot\nabla\psi_{\pm})\psi_{\pm}\mathrm{~d}x\mathrm{~d}t \pm s\lambda\int_{\Omega_{T}} \xi_{\pm}\partial_{t}f(\nabla\eta\cdot\nabla\psi_{\pm})\psi_{\pm}\mathrm{~d}x\mathrm{~d}t.
\end{eqnarray*}
After an integration by parts, we find
\begin{eqnarray*}
	\pm \frac{s\lambda}{2}\int_{\Omega_{T}} \xi_{\pm}|\nabla f|^{2}(\nabla\eta\cdot\nabla\psi_{\pm})\psi_{\pm}\mathrm{~d}x\mathrm{~d}t &=& \pm  \frac{s\lambda}{4}\int_{\Gamma_{T}}\xi |\nabla f|^{2}\partial_{\n}\eta|\psi|^{2}\mathrm{~d}\sigma\mathrm{~d}t \mp \frac{s\lambda}{4}\int_{\Omega_{T}} \xi_{\pm}(\nabla (|\nabla f|^{2})\cdot\nabla\eta)|\psi_{\pm}|^{2}\mathrm{~d}x\mathrm{~d}t \\
 &&- \frac{s\lambda^{2}}{4}\int_{\Omega_{T}}\xi_{\pm}|\nabla f|^{2}|\nabla\eta|^{2}|\psi_{\pm}|^{2}\mathrm{~d}x\mathrm{~d}t  \mp \frac{s\lambda}{4}\int_{\Omega_{T}} \xi_{\pm}|\nabla f|^{2}\Delta\eta|\psi_{\pm}|^{2}\mathrm{~d}x\mathrm{~d}t.
\end{eqnarray*}	
Using $\nabla|\nabla f|^{2}\cdot\nabla\eta=2\nabla^{2}f(\nabla f,\nabla\eta)$, where $\nabla^{2}f$ denotes
the Hessian matrix of $f$ (it is considered as a symmetrical bilinear form), we obtain 	
\begin{eqnarray}
	\pm \frac{s\lambda}{2}\int_{\Omega_{T}} \xi_{\pm}|\nabla f|^{2}(\nabla\eta\cdot\nabla\psi_{\pm})\psi_{\pm}\mathrm{~d}x\mathrm{~d}t		
	&=& \pm  \frac{s\lambda}{4}\int_{\Gamma_{T}} \xi|\nabla f|^{2}\partial_{\n}\eta|\psi|^{2}\mathrm{~d}\sigma\mathrm{~d}t  \mp \frac{s\lambda}{2}\int_{\Omega_{T}} \xi_{\pm}\nabla^{2}f(\nabla f,\nabla\eta)|\psi_{\pm}|^{2}\mathrm{~d}x\mathrm{~d}t \nonumber\\
 &&- \frac{s\lambda^{2}}{4}\int_{\Omega_{T}} \xi_{\pm}|\nabla f|^{2}|\nabla\eta|^{2}|\psi_{\pm}|^{2}\mathrm{~d}x\mathrm{~d}t  \mp \frac{s\lambda}{4}\int_{\Omega_{T}} \xi_{\pm}|\nabla f|^{2}\Delta\eta|\psi_{\pm}|^{2}\mathrm{~d}x\mathrm{~d}t \label{inte}.
\end{eqnarray}
By Young inequality, $\lambda\geq 1$ and $\xi_{\pm}\geq\frac{4}{T^{2}}$, the first three terms in the right-hand side of \eqref{inte} can be bounded by 
\begin{eqnarray*}
	Cs\lambda^{2}T^{4}\left(\|\nabla^{2} f\|^{2}_{\infty}+\|\nabla f\|^{2}_{\infty}\right)\int_{\Omega_{T}}\xi_{\pm}^{3}|\psi_{\pm}|^{2}\mathrm{~d}x\mathrm{~d}t.
\end{eqnarray*}
Thus, we have
\begin{eqnarray}
	&&\pm \frac{s\lambda}{2}\int_{\Omega_{T}} \xi_{\pm}|\nabla f|^{2}(\nabla\eta\cdot\nabla\psi_{\pm})\psi_{\pm}\mathrm{~d}x\mathrm{~d}t \geq  \pm \frac{s\lambda}{4}\int_{\Gamma_{T}}\xi |\nabla f|^{2}\partial_{\n}\eta|\psi|^{2}\mathrm{~d}\sigma\mathrm{~d}t  - Cs\lambda^{2}T^{4}\left(\|\nabla^{2} f\|^{2}_{\infty}+\|\nabla f\|^{2}_{\infty}\right)\int_{\Omega_{T}}\xi_{\pm}^{3}|\psi_{\pm}|^{2}\mathrm{~d}x\mathrm{~d}t.  \nonumber\\ \label{c18}
\end{eqnarray}
Similarly, we obtain
\begin{eqnarray*}
	\pm s\lambda\int_{\Omega_{T}} \xi_{\pm}\partial_{t}f(\nabla\eta\cdot\nabla\psi_{\pm})\psi_{\pm}\mathrm{~d}x\mathrm{~d}t&=& \pm  \frac{s\lambda}{2}\int_{\Gamma_{T}}\xi \partial_{t}f\partial_{\n}\eta|\psi|^{2}\mathrm{~d}\sigma\mathrm{~d}t \mp \frac{s\lambda}{2}\int_{\Omega_{T}} \xi_{\pm}(\nabla (\partial_{t}f)\cdot\nabla\eta)|\psi_{\pm}|^{2}\mathrm{~d}x\mathrm{~d}t \\
 &&- \frac{s\lambda^{2}}{2}\int_{\Omega_{T}}\xi_{\pm}\partial_{t}f|\nabla\eta|^{2}|\psi_{\pm}|^{2}\mathrm{~d}x\mathrm{~d}t  \mp \frac{s\lambda}{2}\int_{\Omega_{T}} \xi_{\pm}\partial_{t}f\Delta\eta|\psi_{\pm}|^{2}\mathrm{~d}x\mathrm{~d}t \\
 &\geq& \pm  \frac{s\lambda}{2}\int_{\Gamma_{T}}\xi \partial_{t}f\partial_{\n}\eta|\psi|^{2}\mathrm{~d}\sigma\mathrm{~d}t- Cs\lambda^{2}T^{4}\left(\|\nabla(\partial_{t}f)\|_{\infty}+\|\partial_{t}f\|_{\infty}\right)\int_{\Omega_{T}}\xi_{\pm}^{3}|\psi_{\pm}|^{2}\mathrm{~d}x\mathrm{~d}t.
\end{eqnarray*}
Consequently 
\begin{eqnarray*}
	((L_{1}\psi_{\pm})_{2},(L_{2}\psi_{\pm})_{4})_{L^{2}(\Omega_{T})}&\geq&  \pm  \frac{s\lambda}{4}\int_{\Gamma_{T}} \xi|\nabla f|^{2}\partial_{\n}\eta|\psi|^{2}\mathrm{~d}\sigma\mathrm{~d}t  \pm  \frac{s\lambda}{2}\int_{\Gamma_{T}}\xi \partial_{t}f\partial_{\n}\eta|\psi|^{2}\mathrm{~d}\sigma\mathrm{~d}t\\
	&& - Cs\lambda^{2}T^{4}\left(\|\nabla^{2}f\|^{2}_{\infty}+\|\nabla f\|^{2}_{\infty}+\|\nabla(\partial_{t}f)\|_{\infty}+\|\partial_{t}f\|_{\infty}\right)\int_{\Omega_{T}}\xi_{\pm}^{3}|\psi_{\pm}|^{2}\mathrm{~d}x\mathrm{~d}t.
\end{eqnarray*}
By integration in time and $\psi_{\pm}(\cdot,0)=\psi_{\pm}(\cdot,T)=0$, we have
\begin{eqnarray}
	((L_{1}\psi_{\pm})_{3},(L_{2}\psi_{\pm})_{4})_{L^{2}(\Omega_{T})}&=& \frac{1}{2\varepsilon}\int_{\Omega_{T}}\partial_{t}\mathcal{V}(f) |\psi_{\pm}|^{2}\mathrm{~d}x\mathrm{~d}t \nonumber\\
 &\geq&  -\frac{CT^{6}}{\varepsilon}\left(\|\partial_{t}\nabla f\|^{2}_{\infty}+\|\nabla f\|^{2}_{\infty}+\|\partial_{t}^{2} f\|_{\infty}\right)\int_{\Omega_{T}}\xi_{\pm}^{3} |\psi_{\pm}|^{2}\mathrm{~d}x\mathrm{~d}t. \label{c19}
\end{eqnarray}
From \eqref{c17},\eqref{c18} and \eqref{c19}, we conclude that
\begin{eqnarray}
		(L_{1}\psi_{\pm},(L_{2}\psi_{\pm})_{4})_{L^{2}(\Omega_{T})} &\geq& \pm \frac{s\lambda}{4}\int_{\Gamma_{T}} \xi|\nabla f|^{2}\partial_{\n}\eta|\psi|^{2}\mathrm{~d}\sigma\mathrm{~d}t \pm  \frac{s\lambda}{2}\int_{\Gamma_{T}}\xi \partial_{t}f\partial_{\n}\eta|\psi|^{2}\mathrm{~d}\sigma\mathrm{~d}t \nonumber\\
  &&-Cs\lambda^{2}T^{4}\left(\|\nabla^{2} f\|^{2}_{\infty}+\|\nabla f\|^{2}_{\infty}\right)\int_{\Omega_{T}}\xi_{\pm}^{3}|\psi_{\pm}|^{2}\mathrm{~d}x\mathrm{~d}t \nonumber\\
		&& -\frac{CT^{6}}{\varepsilon}\left(\|\partial_{t}\nabla f\|^{2}_{\infty}+\|\nabla f\|^{2}_{\infty}+\|\partial_{t}^{2} f\|_{\infty}\right)\int_{\Omega_{T}}\xi_{\pm}^{3} |\psi_{\pm}|^{2}\mathrm{~d}x\mathrm{~d}t.
		\label{c20}
\end{eqnarray}
\textbf{\underline{Step 3.} First conclusion.}\\
Taking in account \eqref{c8}-\eqref{c16} and \eqref{c20}, for any $\lambda\geq \lambda_{1}$ and $s\geq s_{1}\left(\frac{T}{\varepsilon}+T^{2}\right) $, we obtain
\begin{eqnarray}
	(L_{1}\psi_{\pm},L_{2}\psi_{\pm})_{L^{2}(\Omega_{T})} &\geq&  c\,\varepsilon^{2}s^{3}\lambda^{4}\int_{\Omega_{T}}\xi_{\pm}^{3}|\psi_{\pm}|^{2}\mathrm{~d}x\mathrm{~d}t-C\varepsilon^{2}s^{3}\lambda^{4}\int_{\omega^{\prime}_{T}}\xi_{\pm}^{3}|\psi_{\pm}|^{2}\mathrm{~d}x\mathrm{~d}t \nonumber \mp\varepsilon^{2}s^{3}\lambda^{3}\int_{\Gamma_{T}} \xi^{3}|\nabla\eta|^{2}\partial_{\n}\eta\psi^{2}\mathrm{~d}\sigma\mathrm{~d}t \\
	&& - 2\varepsilon^{2}s\lambda^{2}\int_{\Gamma_{T}}\xi|\nabla\eta|^{2}\psi\partial_{\n}\psi_{\pm}\mathrm{~d}\sigma\mathrm{~d}t  -C\varepsilon^{2}s^{2}\lambda^{4}\int_{\Omega_{T}} \xi_{\pm}^{2}|\psi_{\pm}|^{2}\mathrm{~d}x\mathrm{~d}t + c\,\varepsilon^{2}s\lambda^{2}\int_{\Omega_{T}}\xi_{\pm}|\nabla\psi_{\pm}|^{2}\mathrm{~d}x\mathrm{~d}t \nonumber\\
	&& -C\varepsilon^{2}s\lambda^{2}\int_{\omega^{\prime}_{T}}|\nabla\psi_{\pm}|^{2}\xi_{\pm}\mathrm{~d}x\mathrm{~d}t \mp 2\varepsilon^{2}s\lambda\int_{\Gamma_{T}}\xi\partial_{\n}\eta|\partial_{\n}\psi_{\pm}|^{2}\mathrm{~d}\sigma\mathrm{~d}t  \pm  \varepsilon^{2}s\lambda\int_{\Gamma_{T}}\xi\partial_{\n}\eta|\nabla\psi_{\pm}|^{2}\mathrm{~d}\sigma\mathrm{~d}t\nonumber \\
 && +\varepsilon\int_{\Gamma_{T}}\partial_{t}\psi\partial_{\n}\psi_{\pm}\mathrm{~d}\sigma\mathrm{~d}t -C\varepsilon^{2} s^{3}\lambda^{2}\int_{\Omega_{T}} \xi_{\pm}^{3}|\psi_{\pm}|^{2}\mathrm{~d}x\mathrm{~d}t \pm\varepsilon s^{2}\lambda\int_{\Gamma_{T}}  \xi\partial_{t}\alpha\partial_{\n}\eta|\psi|^{2}\mathrm{~d}\sigma\mathrm{~d}t  \nonumber\\
 &&\pm \frac{s\lambda}{4}\displaystyle\int_{\Gamma_{T}} \xi|\nabla f|^{2}\partial_{\n}\eta|\psi|^{2}\mathrm{~d}\sigma\mathrm{~d}t \pm  \frac{s\lambda}{2}\int_{\Gamma_{T}}\xi \partial_{t}f\partial_{\n}\eta|\psi|^{2}\mathrm{~d}\sigma\mathrm{~d}t \nonumber\\
 &&-Cs\lambda^{2}T^{4}\left(\|\nabla^{2} f\|^{2}_{\infty}+\|\nabla f\|^{2}_{\infty}\right)\int_{\Omega_{T}}\xi_{\pm}^{3}|\psi_{\pm}|^{2}\mathrm{~d}x\mathrm{~d}t \nonumber\\
 && -\frac{CT^{6}}{\varepsilon}\left(\|\partial_{t}\nabla f\|^{2}_{\infty}+\|\nabla f\|^{2}_{\infty}+\|\partial_{t}^{2} f\|_{\infty}\right)\int_{\Omega_{T}}\xi_{\pm}^{3} |\psi_{\pm}|^{2}\mathrm{~d}x\mathrm{~d}t. \label{iint1}
\end{eqnarray}
Using \eqref{c2}, we obtain
\begin{eqnarray}
	&&	\|L_{1}\psi_{\pm}\|^{2}_{L^{2}(\Omega_{T})}+\|L_{2}\psi_{\pm}\|^{2}_{L^{2}(\Omega_{T})}+c\,\varepsilon^{2}s^{3}\lambda^{4}\int_{\Omega_{T}}\xi_{\pm}^{3}|\psi_{\pm}|^{2}\mathrm{~d}x\mathrm{~d}t+c\,\varepsilon^{2}s\lambda^{2}\int_{\Omega_{T}}\xi_{\pm}|\nabla\psi_{\pm}|^{2}\mathrm{~d}x\mathrm{~d}t + 2\; I_{\pm} \nonumber\\
	&& \quad\leqslant C\left( \|L_{3}\psi_{\pm}\|^{2}_{L^{3}(\Omega_{T})}+\varepsilon^{2}s^{3}\lambda^{4}\int_{\omega^{\prime}_{T}}\xi_{\pm}^{3}|\psi_{\pm}|^{2}\mathrm{~d}x\mathrm{~d}t +\varepsilon^{2}s\lambda^{2}\int_{\omega^{\prime}_{T}}\xi_{\pm}|\nabla\psi_{\pm}|^{2}\mathrm{~d}x\mathrm{~d}t   +\varepsilon^{2}s^{2}\lambda^{4}\int_{\Omega_{T}} \xi_{\pm}^{2}|\psi_{\pm}|^{2}\mathrm{~d}x\mathrm{~d}t  \right. \nonumber\\
	&& \left.  \quad\quad + \varepsilon^{2} s^{3}\lambda^{2}\int_{\Omega_{T}} \xi_{\pm}^{3}|\psi_{\pm}|^{2}\mathrm{~d}x\mathrm{~d}t + s\lambda^{2}T^{4}\left(\|\nabla^{2} f\|^{2}_{\infty}+\|\nabla f\|^{2}_{\infty}\right)\int_{\Omega_{T}}\xi_{\pm}^{3}|\psi_{\pm}|^{2}\mathrm{~d}x\mathrm{~d}t  \right. \nonumber
	\\
     &&	\left. \quad\quad -\frac{CT^{6}}{\varepsilon}\left(\|\partial_{t}\nabla f\|^{2}_{\infty}+\|\nabla f\|^{2}_{\infty}+\|\partial_{t}^{2} f\|_{\infty}\right)\int_{\Omega_{T}}\xi_{\pm}^{3} |\psi_{\pm}|^{2}\mathrm{~d}x\mathrm{~d}t
	\right), \label{c37}
\end{eqnarray}
where $I_{\pm}$ is the sum of all integrals on the boundary 
in the right-hand side of \eqref{iint1}. 
The last integral in the right hand side of \eqref{c37} can be absorbed by $\displaystyle c\varepsilon^{2}s^{3}\int_{\Omega_{T}}\xi_{\pm}^{3}|\psi_{\pm}|^{2}\mathrm{~d}x\mathrm{~d}t$ if $s\geq s_{1}\frac{T^{2}}{\varepsilon}\left(\|\nabla\partial_{t} f\|^{2/3}_{\infty}+\|\nabla f\|^{2/3}_{\infty}+ \|\partial_{t}^{2} f\|^{1/3}_{\infty}\right)$. Similarly, the second-to-last term can be absorbed by  $\displaystyle c\,\varepsilon^{2}s^{3}\int_{\Omega_{T}}\xi_{\pm}^{3}|\psi_{\pm}|^{2}\mathrm{~d}x\mathrm{~d}t$ if $\lambda\geq 1 $ and $s\geq s_{1}\frac{T^{2}}{\varepsilon}\left(\|\nabla^{2} f\|_{\infty}+\|\nabla f\|_{\infty}\right)$. Also, one can see that $\displaystyle\varepsilon^{2}s^{2}\lambda^{4}\int_{\Omega_{T}} \xi_{\pm}^{2}|\psi_{\pm}|^{2}\mathrm{~d}x\mathrm{~d}t$ and $\displaystyle\varepsilon^{2} s^{3}\lambda^{2}\int_{\Omega_{T}} \xi_{\pm}^{3}|\psi_{\pm}|^{2}\mathrm{~d}x\mathrm{~d}t$ are absorbed by the same term if we take respectively $s\geq s_{1}T^{2}$ and $\lambda\geq \lambda_{1}$. So we obtain 
\begin{align*}
	&\|L_{1}\psi_{\pm}\|^{2}_{L^{2}(\Omega_{T})}+\|L_{2}\psi_{\pm}\|^{2}_{L^{2}(\Omega_{T})}+c\,\varepsilon^{2}s^{3}\lambda^{4}\int_{\Omega_{T}}\xi_{\pm}^{3}|\psi_{\pm}|^{2}\mathrm{~d}x\mathrm{~d}t +c\varepsilon^{2}s\lambda^{2}\int_{\Omega_{T}}\xi_{\pm}|\nabla\psi_{\pm}|^{2}\mathrm{~d}x\mathrm{~d}t + 2\;I_{\pm}\\
	& \quad\leqslant C\left( \|L_{3}\psi_{\pm}\|^{2}_{L^{3}(\Omega_{T})}+\varepsilon^{2}s^{3}\lambda^{4}\int_{\omega^{\prime}_{T}}\xi_{\pm}^{3}|\psi_{\pm}|^{2}\mathrm{~d}x\mathrm{~d}t +\varepsilon^{2}s\lambda^{2}\int_{\omega^{\prime}_{T}}\xi_{\pm}|\nabla\psi_{\pm}|^{2}\mathrm{~d}x\mathrm{~d}t \right),
\end{align*}
for any $\lambda\geq \lambda_{1}$ and any $$s\geq s_{1}\left(\frac{T}{\varepsilon}+\frac{T^{2}}{\varepsilon}\left(1+\|\nabla^{2} f\|_{\infty}+\|\nabla f\|_{\infty} +\|\nabla\partial_{t} f\|^{2/3}_{\infty}+\|\nabla f\|^{2/3}_{\infty}+ \|\partial_{t}^{2} f\|^{1/3}_{\infty}\right)\right).$$
From \eqref{c0}, we obtain 
\begin{eqnarray}
	\|L_{3}\psi_{\pm}\|^{2}_{L^{3}(\Omega_{T})}&\leqslant&  C\left(\int_{\Omega_{T}} |F_{\pm}|^{2}\mathrm{~d}x\mathrm{~d}t+\varepsilon^{2}s^{2}\lambda^{2}\int_{\Omega_{T}}\xi_{\pm}^{2} |\psi_{\pm}|^{2}\mathrm{~d}x\mathrm{~d}t +\|\Delta f\|_{\infty}^{2}\int_{\Omega_{T}} |\psi_{\pm}|^{2}\mathrm{~d}x\mathrm{~d}t\right) \nonumber\\
	&\leqslant& C\left(\int_{\Omega_{T}} |F_{\pm}|^{2}\mathrm{~d}x\mathrm{~d}t+\left(\varepsilon^{2}s^{2}\lambda^{2}T^{2}+\|\Delta f\|_{\infty}^{2}T^{6}\right)\int_{\Omega_{T}}\xi_{\pm}^{3} |\psi_{\pm}|^{2}\mathrm{~d}x\mathrm{~d}t\right).  \label{c38}
\end{eqnarray}
The last term in the right hand side of \eqref{c38} is absorbed by $\displaystyle c\,\varepsilon^{2}s^{3}\int_{\Omega_{T}}\xi_{\pm}^{3}|\psi_{\pm}|^{2}\mathrm{~d}x\mathrm{~d}t$ for $s\geq s_{1}\left(T^{2}+ \frac{T^{2}}{\varepsilon^{2/3}}\|\Delta f\|_{\infty}^{2/3}\right)$ and $\lambda\geq 1$.\\
Finally, we obtain
    \begin{align}
		& \|L_{1}\psi_{\pm}\|^{2}_{L^{2}(\Omega_{T})}+\|L_{2}\psi_{\pm}\|^{2}_{L^{2}(\Omega_{T})}+c\,\varepsilon^{2}s^{3}\lambda^{4}\int_{\Omega_{T}}\xi_{\pm}^{3}|\psi_{\pm}|^{2}\mathrm{~d}x\mathrm{~d}t +c\,\varepsilon^{2}s\lambda^{2}\int_{\Omega_{T}}|\nabla\psi_{\pm}|^{2}\xi_{\pm}\mathrm{~d}x\mathrm{~d}t + 2\,I_{\pm} \nonumber\\
		& \quad \leqslant C\left( \int_{\Omega_{T}} |F_{\pm}|^{2}\mathrm{~d}x\mathrm{~d}t+\varepsilon^{2}s^{3}\lambda^{4}\int_{\omega^{\prime}_{T}}\xi_{\pm}^{3}|\psi_{\pm}|^{2}\mathrm{~d}x\mathrm{~d}t +\varepsilon^{2}s\lambda^{2}\int_{\omega^{\prime}_{T}}\xi_{\pm}|\nabla\psi_{\pm}|^{2}\mathrm{~d}x\mathrm{~d}t \right),
		\label{c21}
\end{align}
for any $\varepsilon\in (0,1)$, any $\lambda\geq \lambda_{1}$ and any $s\geq \frac{s_{1}}{\varepsilon}(T+T^{2})\mathcal{A}_{T}(f),$
where
$$\mathcal{A}_{T}(f):=1+\|\nabla f\|_{\infty}+\|\nabla^{2}f\|_{\infty}+\|\nabla\partial_{t} f\|^{2/3}_{\infty}+\|\nabla f\|^{2/3}_{\infty}+ \|\partial_{t}^{2} f\|^{1/3}_{\infty}+\|\Delta f\|_{\infty}^{2/3}.$$
\textbf{\underline{Step 4.} Elimination of the integral of $|\nabla\psi_{\pm}|^{2}$ in the right-hand side of \eqref{c21}.}\\
We start by adding integral of $|\Delta \psi_{\pm}|^{2}$ to the left-hand side of \eqref{c21}, so that we can eliminate the last term in the right-hand side of \eqref{c21}. Using \eqref{c-1}, $\xi_{\pm}\geq \frac{4}{T^{2}}$, $s\geq s_{1}T^{2}$ and $|\partial_{t}\alpha_{\pm}|\leqslant T\xi_{\pm}^{2}$, we
obtain
\begin{eqnarray}
	\varepsilon^{2}s^{-1}\int_{\Omega_{T}}\xi^{-1}_{\pm}|\Delta\psi_{\pm}|^{2}\mathrm{~d}x\mathrm{~d}t
	&\leqslant&  C\left(\varepsilon^{2}s^{3}\lambda^{4}\int_{\Omega_{T}}\xi_{\pm}^{3}|\psi_{\pm}|^{2}\mathrm{~d}x\mathrm{~d}t +s T^{2}\int_{\Omega_{T}}\xi^{3}_{\pm}|\psi_{\pm}|^{2}\mathrm{~d}x\mathrm{~d}t\right. \nonumber\\
	&& + \left. \left(\|\nabla f\|^{4}_{\infty}+\|\partial_{t} f\|^{2}_{\infty}\right)\frac{s^{-1}}{\varepsilon^{2}}\int_{\Omega_{T}}\xi^{-1}_{\pm}|\psi_{\pm}|^{2}\mathrm{~d}x\mathrm{~d}t+ \|L_{2}\psi_{\pm}\|^{2}_{L^{2}(\Omega_{T})}\right).\nonumber
\end{eqnarray}
Hence
\begin{align*}
	&\varepsilon^{2}s^{-1}\int_{\Omega_{T}}\xi^{-1}_{\pm}|\Delta\psi_{\pm}|^{2}\mathrm{~d}x\mathrm{~d}t
	\leqslant  C\left(\varepsilon^{2}s^{3}\lambda^{4}\int_{\Omega_{T}}\xi_{\pm}^{3}|\psi_{\pm}|^{2}\mathrm{~d}x\mathrm{~d}t+\|L_{2}\psi_{\pm}\|^{2}_{L^{2}(\Omega_{T})}\right),
\end{align*}
for  all $\lambda\geq 1$ and all $s\geq \frac{s_{1}}{\varepsilon}\left(
T+\left(1+\|\nabla f\|_{\infty}+\|\partial_{t} f\|^{1/2}_{\infty}\right)T^{2}\right)$.
\\
Consequently, we deduce from \eqref{c21} that
\begin{align}
	& \varepsilon^{2}s^{-1}\int_{\Omega_{T}}\xi^{-1}_{\pm}|\Delta\psi_{\pm}|^{2}\mathrm{~d}x\mathrm{~d}t+\varepsilon^{2}s^{3}\lambda^{4}\int_{\Omega_{T}}\xi_{\pm}^{3}|\psi_{\pm}|^{2}\mathrm{~d}x\mathrm{~d}t+\varepsilon^{2}s\lambda^{2}\int_{\Omega_{T}}|\nabla\psi_{\pm}|^{2}\xi_{\pm}\mathrm{~d}x\mathrm{~d}t + c\,I_{\pm} \nonumber\\
	&\quad \leqslant C\left( \int_{\Omega_{T}} |F_{\pm}|^{2}\mathrm{~d}x\mathrm{~d}t+\varepsilon^{2}s^{3}\lambda^{4}\int_{\omega^{\prime}_{T}}\xi_{\pm}^{3}|\psi_{\pm}|^{2}\mathrm{~d}x\mathrm{~d}t+\varepsilon^{2}s\lambda^{2}\int_{\omega^{\prime}_{T}}|\nabla\psi_{\pm}|^{2}\xi_{\pm}\mathrm{~d}x\mathrm{~d}t \right), 
	\label{c23}
\end{align}
for any $\lambda\geq \lambda_{1}$ and any $s\geq \frac{s_{1}}{\varepsilon}\left(T+\mathcal{B}_{T}(f)T^{2}\right),$
where
$$\mathcal{B}_{T}(f):=1+\|\nabla f\|_{\infty}+\|\nabla^{2}f\|_{\infty}+\|\nabla\partial_{t} f\|^{2/3}_{\infty}+\|\nabla f\|^{2/3}_{\infty}+ \|\partial_{t}^{2} f\|^{1/3}_{\infty}+\|\Delta f\|_{\infty}^{2/3}+\|\partial_{t} f\|^{1/2}_{\infty}.$$
\par 
As usual, to eliminate the last term on the right-hand side of \eqref{c23}, let us introduce $\theta\in \mathcal{C}^{2}(\omega)$ a positive cut-off function such that $\theta=1$ in $\omega^{\prime}$, an integration by parts and the Cauchy-Schwarz inequality as in \cite{guerrero2007singular}, we obtain
\begin{eqnarray*}
    C\varepsilon^{2}s\lambda^{2}\int_{\omega^{\prime}_{T}}|\nabla\psi_{\pm}|^{2}\xi_{\pm}\mathrm{~d}x\mathrm{~d}t &=& C\varepsilon^{2}s\lambda^{2}\int_{\omega^{\prime}_{T}}\theta|\nabla\psi_{\pm}|^{2}\xi_{\pm}\mathrm{~d}x\mathrm{~d}t\leq C\varepsilon^{2}s\lambda^{2}\int_{\omega_{T}}\theta|\nabla\psi_{\pm}|^{2}\xi_{\pm}\mathrm{~d}x\mathrm{~d}t\\
    &\leq & \frac{1}{2}\left(\varepsilon^{2}s^{-1}\int_{\Omega_{T}}\xi^{-1}_{\pm}|\Delta\psi_{\pm}|^{2}\mathrm{~d}x\mathrm{~d}t+\varepsilon^{2}s\lambda^{2}\int_{\Omega_{T}}\xi_{\pm}|\nabla\psi_{\pm}|^{2}\mathrm{~d}x\mathrm{~d}t\right) + C\varepsilon^{2}s^{3}\lambda^{4}\int_{\omega_{T}}\xi_{\pm}^{3}|\psi_{\pm}|^{2}\mathrm{~d}x\mathrm{~d}t.
\end{eqnarray*}
Combining this last estimate with \eqref{c23}, we conclude that 
\begin{eqnarray}
	&&\varepsilon^{2}s^{-1}\int_{\Omega_{T}}\xi^{-1}_{\pm}|\Delta\psi_{\pm}|^{2}\mathrm{~d}x\mathrm{~d}t+	
	\varepsilon^{2}s^{3}\lambda^{4}\int_{\Omega_{T}}\xi_{\pm}^{3}|\psi_{\pm}|^{2}\mathrm{~d}x\mathrm{~d}t +\varepsilon^{2}s\lambda^{2}\int_{\Omega_{T}}\xi_{\pm}|\nabla\psi_{\pm}|^{2}\mathrm{~d}x\mathrm{~d}t + c\,I_{\pm} \nonumber\\
	&& \quad \leqslant C\left( \int_{\Omega_{T}} |F_{\pm}|^{2}\mathrm{~d}x\mathrm{~d}t+\varepsilon^{2}s^{3}\lambda^{4}\int_{\omega_{T}}\xi_{\pm}^{3}|\psi_{\pm}|^{2}\mathrm{~d}x\mathrm{~d}t \right), \label{c24}
\end{eqnarray}
for any $\lambda\geq \lambda_{1}$ and any $s\geq\frac{s_{1}}{\varepsilon}\left(T+\mathcal{B}_{T}(f)T^{2}\right)$.\\
\textbf{\underline{Step 5.} Simplification of the boundary terms.}
By summing \eqref{c24} for $i=+,-$, we obtain 
\begin{eqnarray}
	&&\varepsilon^{2}s^{3}\lambda^{4}\int_{\Omega_{T}}\xi_{+}^{3}|\psi_{+}|^{2}\mathrm{~d}x\mathrm{~d}t +\varepsilon^{2}s\lambda^{2}\int_{\Omega_{T}}\xi_{+}|\nabla\psi_{+}|^{2}\mathrm{~d}x\mathrm{~d}t + c\,\left(I_{+}+I_{-}\right) \nonumber\\
	&&\quad\leqslant C\left( \int_{\Omega_{T}} \left(|F_{+}|^{2}+|F_{-}|^{2}\right)\mathrm{~d}x\mathrm{~d}t+\varepsilon^{2}s^{3}\lambda^{4}\int_{\omega_{T}}\left(\xi_{+}^{3}|\psi_{+}|^{2}+\xi_{-}^{3}|\psi_{-}|^{2}\right)\mathrm{~d}x\mathrm{~d}t \right). \label{c25}
\end{eqnarray}
From the definitions of $\xi_{\pm}$ and $\alpha_{\pm}$, we have 
$\xi_{-}\leqslant \xi_{+}\quad \text{and}\quad \alpha_{+}\leqslant \alpha_{-}\;\;\text{in}\;\;\Omega_{T}.$
Then, the estimate \eqref{c25} gives 
\begin{eqnarray*}
	&&\varepsilon^{2}s^{3}\lambda^{4}\int_{\Omega_{T}}\xi_{+}^{3}|\psi_{+}|^{2}\mathrm{~d}x\mathrm{~d}t +\varepsilon^{2}s\lambda^{2}\int_{\Omega_{T}}\xi_{+}|\nabla\psi_{+}|^{2}\mathrm{~d}x\mathrm{~d}t + c\,\left(I_{+}+I_{-}\right) \nonumber\\
	&& \quad\leqslant C\left( \int_{\Omega_{T}} |F_{+}|^{2}\mathrm{~d}x\mathrm{~d}t+\varepsilon^{2}s^{3}\lambda^{4}\int_{\omega_{T}}\xi_{+}^{3}|\psi_{+}|^{2}\mathrm{~d}x\mathrm{~d}t \right).
	\label{c26}
\end{eqnarray*} 
Before simplifying $I_{+}+I_{-}$, we  turn back to our original function $\Phi$. From \eqref{c-4}, we deduce that 
\begin{eqnarray}
	&&\varepsilon^{2}s^{3}\lambda^{4}\int_{\Omega_{T}}\exp(-2s\alpha_{+})\xi_{+}^{3}|\Phi|^{2}\mathrm{~d}x\mathrm{~d}t+\varepsilon^{2}s\lambda^{2}\int_{\Omega_{T}}|\nabla\psi_{+}|^{2}\xi_{+}\mathrm{~d}x\mathrm{~d}t + c\,\left(I_{+}+I_{-}\right) \nonumber\\
	&&\quad \leqslant C\left( \int_{\Omega_{T}} \exp(-2s\alpha_{+})|\partial_{t}\Phi+\varepsilon\Delta\Phi-a_{\varepsilon}(f)\Phi|^{2}\mathrm{~d}x\mathrm{~d}t +\varepsilon^{2}s^{3}\lambda^{4}\int_{\omega_{T}}\exp(-2s\alpha_{+})\xi_{+}^{3}|\Phi|^{2}\mathrm{~d}x\mathrm{~d}t \right). 
	\label{c32}
\end{eqnarray}
For $\nabla\Phi$, we use the identity given in \eqref{c-5}, we have
\begin{eqnarray*}
	\exp(-s\alpha_{+})\nabla\Phi=\nabla\psi_{+} - s\lambda\xi_{+}\nabla\eta\psi_{+}.
\end{eqnarray*}
Applying the triangular inequality to this identity, we find 
\begin{eqnarray*}
 \varepsilon^{2}s\lambda^{2}\int_{\Omega_{T}}\exp(-2s\alpha_{+})\xi_{+}|\nabla\Phi|^{2}\mathrm{~d}x\mathrm{~d}t
	\leqslant C\left(\varepsilon^{2}s\lambda^{2}\int_{\Omega_{T}}\xi_{+}|\nabla\psi_{+}|^{2}\mathrm{~d}x\mathrm{~d}t+\varepsilon^{2}s^{3}\lambda^{4}\int_{\Omega_{T}}\exp(-2s\alpha_{+})\xi_{+}^{3}|\Phi|^{2}\mathrm{~d}x\mathrm{~d}t\right).
\end{eqnarray*}
Consequently, we can add the previous integral of 
$|\nabla\Phi|^{2}$ to the left-hand side of \eqref{c32}:
\begin{eqnarray}
	&&\varepsilon^{2}s^{3}\lambda^{4}\int_{\Omega_{T}}\exp(-2s\alpha_{+})\xi_{+}^{3}|\Phi|^{2}\mathrm{~d}x\mathrm{~d}t +\varepsilon^{2}s\lambda^{2}\int_{\Omega_{T}}\exp(-2s\alpha_{+})\xi_{+}|\nabla\Phi|^{2}\mathrm{~d}x\mathrm{~d}t + c\,\left(I_{+}+I_{-}\right) \nonumber\\
	&&\quad \leqslant C\left( \int_{\Omega_{T}} \exp(-2s\alpha_{+})|\partial_{t}\Phi+\varepsilon\Delta\Phi-a_{\varepsilon}(f)\Phi|^{2}\mathrm{~d}x\mathrm{~d}t +\varepsilon^{2}s^{3}\lambda^{4}\int_{\omega_{T}}\exp(-2s\alpha_{+})\xi_{+}^{3}|\Phi|^{2}\mathrm{~d}x\mathrm{~d}t \right).
	\label{c33}
\end{eqnarray}
Next, we will simplify $I_{+}+I_{-}$. It is clear that 
\begin{eqnarray*}
	 I_{+}+I_{-}&=& - 2\varepsilon^{2}s\lambda^{2}\int_{\Gamma_{T}}\xi|\nabla\eta|^{2}\psi\left[\partial_{\n}\psi_{+}+\partial_{\n}\psi_{-}\right]\mathrm{~d}\sigma\mathrm{~d}t   - 2\varepsilon^{2}s\lambda\int_{\Gamma_{T}}\xi\partial_{\n}\eta\left[|\partial_{\n}\psi_{+}|^{2}-|\partial_{\n}\psi_{-}|^{2}\right]\mathrm{~d}\sigma\mathrm{~d}t \nonumber\\
	&& + \varepsilon^{2}s\lambda\int_{\Gamma_{T}}\xi\partial_{\n}\eta \left[|\nabla\psi_{+}|^{2}-|\nabla\psi_{-}|^{2}\right]\mathrm{~d}\sigma\mathrm{~d}t  +\varepsilon\int_{\Gamma_{T}}\partial_{t}\psi\left[\partial_{\n}\psi_{+}+\partial_{\n}\psi_{-}\right]\mathrm{~d}\sigma\mathrm{~d}t.
\end{eqnarray*}
Using $|\nabla\psi_{\pm}|^{2}=|\nabla_{\Gamma}\psi|^{2}+|\partial_{\n}\psi_{\pm}|^{2}$, \eqref{c5}, \eqref{c-6}, \eqref{c-3}   and $\varepsilon\partial_{\n}\Phi=-\frac{\partial_{\n}f}{2}\Phi$, we obtain
\begin{eqnarray*}
	I_{+}+I_{-}&=& 2\varepsilon s\lambda^{2}	\int_{\Gamma_{T}}\partial_{\n}f
	|\nabla\eta|^{2}\xi\exp(-2s\alpha)|\Phi|^{2}\mathrm{~d}\sigma\mathrm{~d}t +2\varepsilon s^{2}\lambda^{2}
	\int_{\Gamma_{T}}\partial_{\n}f
	|\partial_{\n}\eta|^{2}\xi^{2}\exp(-2s\alpha)|\Phi|^{2}\mathrm{~d}\sigma\mathrm{~d}t\quad 
        \nonumber\\ 
	&& -\int_{\Gamma_{T}}\partial_{\n}f\partial_{t}\psi\psi\mathrm{~d}\sigma\mathrm{~d}t. \quad 
\end{eqnarray*}
Integrating the last integral with respect to time, we obtain 
\begin{eqnarray*}
	I_{+}+I_{-}&=& 2\varepsilon s\lambda^{2}	\int_{\Gamma_{T}}\partial_{\n}f
	|\nabla\eta|^{2}\xi\exp(-2s\alpha)|\Phi|^{2}\mathrm{~d}\sigma\mathrm{~d}t+2\varepsilon s^{2}\lambda^{2}
	\int_{\Gamma_{T}}\partial_{\n}f
	|\partial_{\n}\eta|^{2}\xi^{2}\exp(-2s\alpha)|\Phi|^{2}\mathrm{~d}\sigma\mathrm{~d}t\quad \nonumber\\ 
	&& +\int_{\Gamma_{T}}\partial_{t}(\partial_{\n}f)\exp(-2s\alpha)|\Phi|^{2}\mathrm{~d}\sigma\mathrm{~d}t. \quad 
\end{eqnarray*}
The estimate \eqref{c33} and $\partial_{\n}f\geq 0$ on $\Gamma_{T}$ implies the following

\begin{eqnarray}
	&&\varepsilon^{2}s^{3}\lambda^{4}\int_{\Omega_{T}}\exp(-2s\alpha_{+})\xi_{+}^{3}|\Phi|^{2}\mathrm{~d}x\mathrm{~d}t +\varepsilon^{2}s\lambda^{2}\int_{\Omega_{T}}\exp(-2s\alpha_{+})\xi_{+}|\nabla\Phi|^{2}\mathrm{~d}x\mathrm{~d}t \nonumber\\
 &&+2\varepsilon s\lambda^{2}
	\int_{\Gamma_{T}}\partial_{\n}f
	|\partial_{\n}\eta|^{2}\left(\xi+s\xi^{2}\right)\exp(-2s\alpha)|\Phi|^{2}\mathrm{~d}\sigma\mathrm{~d}t \nonumber\\
	&&\leqslant C\left( \int_{\Omega_{T}} \exp(-2s\alpha_{+})|\partial_{t}\Phi+\varepsilon\Delta\Phi-a_{\varepsilon}(f)\Phi|^{2}\mathrm{~d}x\mathrm{~d}t +\varepsilon^{2}s^{3}\lambda^{4}\int_{\omega_{T}}\exp(-2s\alpha_{+})\xi_{+}^{3}|\Phi|^{2}\mathrm{~d}x\mathrm{~d}t \right. \nonumber\\
	&&\left. \quad +\|\partial_{t}\partial_{\n}f\|_{\infty}\int_{\Gamma_{T}}\exp(-2s\alpha)|\Phi|^{2}\d\sigma\d t\right). \label{cc34}
\end{eqnarray}
for any $\lambda\geq \lambda_{1}$ and any $s\geq \frac{s_{1}}{\varepsilon}\left(T+\mathcal{B}_{T}(f)T^{2}\right).$ Since $ \partial_{\mathbf{n}}f\geq c>0$ on $\Gamma_{T}$ then, $(\partial_{\mathbf{n}}f)^{-1}\in L^{\infty}(\Gamma_{T})$. Consequently, taking $s\geq \frac{s_{1}}{\varepsilon}(\|\partial_{t}\partial_{\n}f\|_{L^{\infty}(\Gamma_{T})}+\|(\partial_{\n}f)^{-1}\|_{L^{\infty}(\Gamma_{T})})T^{2}$, we can absorb the last term in the right-hand side of \eqref{cc34} by $\displaystyle\varepsilon s^{2}\lambda^{2}
\int_{\Gamma_{T}}\partial_{\n}f
|\partial_{\n}\eta|^{2}\xi^{2}\exp(-2s\alpha)|\Phi|^{2}\mathrm{~d}\sigma\mathrm{~d}t$. Finally, using the density argument explained at the beginning of the proof, we can deduce the estimate \eqref{Carleman}. \qed

\section*{ACKNOWLEDGMENTS}
\thanks{J.A.B.P was supported by the Grant PID2021-126813NB-I00 
	funded by MCIN/AEI/10.13039/501100011033 and by “ERDF A way of making 
	Europe” and by the grant~IT1615-22 funded the Basque Government.}\\
    
    \noindent\thanks{The authors would like to thank the anonymous reviewer for their valuable comments and suggestions, which contributed to improving the earlier version of this paper.
    }
\section*{Conflict of interest}
The authors declare no potential conflict of interests.



\begin{thebibliography}{plain}
	
	\bibitem{arendt2014maximal} Arendt, W., Dier, D., Laasri, H., \& Ouhabaz, E. M.:  Maximal regularity for evolution equations governed by non-autonomous forms. Advances in Differential Equations. 19, 1043-1066 (2014). \textcolor{blue}{https://doi.org/10.57262/ade/1408367288}
	
	\bibitem{arendt2011dirichlet} Arendt, W., \& ter Elst, A. F.: The Dirichlet-to-Neumann operator on rough domains. Journal of Differential Equations, 251(8), 2100-2124 (2011). \textcolor{blue}{https://doi.org/10.1016/j.jde.2011.06.017}
	
	\bibitem{bahouri2011fourier} Bahouri, H., Chemin, J.-Y., Danchin, R.:  Fourier analysis and nonlinear partial differential equations (Vol. 343). Springer
	Science \& Business Media (2011)
	
	\bibitem{barcena2021cost} Bárcena-Petisco, J. A.:  Cost of null controllability for parabolic equations with vanishing diffusivity and a transport term. ESAIM: Control, Optimization and Calculus of Variations, 27, 106 (2021). \textcolor{blue}{https://doi.org/10.1051/cocv/2021103}
	
	\bibitem{brezis1983analyse} Brezis, H.: Analyse fonctionnelle. Théorie et applications. Masson, Paris (1983)
	
	\bibitem{carreno2016cost} Carre{\~n}o, N., \& Guzm{\'a}n, P.: On the cost of null controllability of a fourth-order parabolic equation. Journal of Differential Equations, 261(11), 6485-6520 (2016). \textcolor{blue}{https://doi.org/10.1016/j.jde.2016.08.042}
  

 
	\bibitem{carreno2015non}  Carre{\~n}o, N., \& Guerrero, S.: On the non-uniform null controllability of a linear KdV equation. Asymptotic Analysis, 94(1-2), 33-69 (2015). \textcolor{blue}{https://doi.org/10.3233/ASY-151300}
	
	\bibitem{coron2005singular} Coron, J.M, \& Guerrero, S.: Singular optimal control: a linear 1-D parabolic--hyperbolic example. Asymptotic Analysis, 44(3-4), 237-257 (2005)
	
	\bibitem{Lions} Dautray, R., \& Lions, J. L.:  Mathematical analysis and numerical methods for science and technology: volume 1 physical origins and classical methods. Springer Science \& Business Media (2012)
	 

    \bibitem{et2024asymptotic}
Et-Tahri, F., B{\'a}rcena-Petisco, J. A., Boutaayamou, I., \& Maniar, L.:
\textit{Asymptotic behavior of null controllability cost for parabolic equations with vanishing diffusivity under Robin and Neumann boundary conditions}, 
ESAIM: Control, Optimisation and Calculus of Variations, \textbf{30}, 74 (2024).

	
	
	
	\bibitem{fernandez2006null} Fernández-Cara, E., González-Burgos, M., Guerrero, S., \& Puel, J. P.:  Null controllability of the heat equation with boundary Fourier conditions: the linear case. ESAIM: Control, Optimisation and Calculus of Variations, 12(3), 442-465 (2006). \textcolor{blue}{https://doi.org/10.1051/cocv:2006010}
	
	\bibitem{fursikov1996controllability} Fursikov, A., \& Imanuvilov, O. Y.:  Controllability of Evolution Equations, Lecture Notes. 34, Seoul National University. Korea (1996)
	
	\bibitem{glass2010complex} Glass, O.: A complex-analytic approach to the problem of uniform controllability of a transport equation in the vanishing viscosity limit. Journal of Functional Analysis, 258(3), 852-868 (2010). \textcolor{blue}{https://doi.org/10.1016/j.jfa.2009.06.035}
	
	\bibitem{grisvard1985elliptic} Grisvard, P.: Elliptic Problems in Nonsmooth Domains. Monographs and Studies in Mathematics, vol. 24. Pitman, Boston 49--52 (1985)
	
	\bibitem{guerrero2007singular} Guerrero, S., \& Lebeau, G.: Singular optimal control for a transport-diffusion equation. Communications in Partial Differential Equations, 32(12), 1813-1836 (2007). \textcolor{blue}{https://doi.org/10.1080/03605300701743756}
	
	\bibitem{kato2013perturbation} Kato, T.:  Perturbation theory for linear operators (Vol. 132). Springer Science \& Business Media,New York (2013)
	
	\bibitem{laurent2021uniform} Laurent, C., \& Léautaud, M.: On uniform observability of gradient flows in the vanishing viscosity limit. Journal de l’École polytechnique—Mathématiques, 8, 439-506 (2021). \textcolor{blue}{https://doi.org/10.5802/jep.151}
	
	\bibitem{laurent2022uniform} Laurent, C., \& Léautaud, M.: On uniform controllability of 1D transport equations in the vanishing viscosity limit. Comptes Rendus. Mathématique, 361(G1), 265-312 (2023). \textcolor{blue}{https://doi.org/10.5802/crmath.405}
	
	\bibitem{lions1988controlabilite} Lions, J. L.:  Contr{\^o}labilit{\'e} exacte, stabilisation et perturbations de systemes distribu{\'e}s. Tome 1. Contr{\^o}labilit{\'e} exacte. Rech. Math. Appl, 8 (1988)
	
	\bibitem{lissy2012link} Lissy, P.: A link between the cost of fast controls for the 1-D heat equation and the uniform controllability of a 1-D transport-diffusion equation. Comptes Rendus Mathematique, 350(11-12), 591-595 (2012). \textcolor{blue}{https://doi.org/10.1016/j.crma.2012.06.004}
	
	\bibitem{lissy2014application} Lissy, P.: An application of a conjecture due to Ervedoza and Zuazua concerning the observability of the heat equation in small time to a conjecture due to Coron and Guerrero concerning the uniform controllability of a convection–diffusion equation in the vanishing viscosity limit. Systems \& Control Letters, 69, 98-102 (2014). \textcolor{blue}{https://doi.org/10.1016/j.sysconle.2014.04.011}
	
	\bibitem{lissy2015explicit} Lissy, P.: Explicit lower bounds for the cost of fast controls for some 1-D parabolic or dispersive equations, and a new lower bound concerning the uniform controllability of the 1-D transport–diffusion equation. Journal of Differential Equations, 259(10), 5331-5352 (2015). \textcolor{blue}{https://doi.org/10.1016/j.jde.2015.06.031}
   \bibitem{lopez2000null} L{\'o}pez, A., Zhang, X., \& Zuazua, E.: Null controllability of the heat equation as singular limit of the exact controllability of dissipative wave equations. Journal de mathématiques pures et appliquées, 79(8), 741-808 (2000). \textcolor{blue}{https://doi.org/10.1016/S0021-7824(99)00144-0}
	
	\bibitem{phung2002null} Phung, K. D.: Null controllability of the heat equation as singular limit of the exact controllability of dissipative wave equation under the Bardos-Lebeau-Rauch geometric control condition. Computers \& Mathematics with Applications, 44(10-11), 1289-1296 (2002). \textcolor{blue}{https://doi.org/10.1016/S0898-1221(02)00256-0}
	
	\bibitem{russell1978controllability}  Russell, D. L.:  Controllability and stabilizability theory for linear partial differential equations: recent progress and open questions. Siam Review, 20(4), 639-739 (1978). \textcolor{blue}{https://doi.org/10.1137/1020095}
	
	\bibitem{showalter2013monotone} Showalter, R. E.:  Monotone operators in Banach space and nonlinear partial differential equations (Vol. 49). American Mathematical Soc..(2013)

       \bibitem{tucsnak2009observation}  Tucsnak, M., \& Weiss, G.: Observation and control for operator semigroups. Springer Science \& Business Media (2009).
       


    
\end{thebibliography}
\end{document}